\documentclass{amsart}
\usepackage{lmodern}
\usepackage[utf8]{inputenc}
\usepackage[T1]{fontenc}
\usepackage{amsmath}
\usepackage{amssymb}
\usepackage{amsthm}
\usepackage{esint}
\usepackage[francais,english]{babel}
\usepackage{comment}
\usepackage{ulem}
\usepackage{graphicx}

\newcommand{\Z}{\mathbb Z}      
\newcommand{\R}{\mathbb R}

\renewcommand{\L}{\mathrm L}
\newcommand{\ex}{\mathrm{ex}}

\newcommand{\ds}{\displaystyle}

\newcommand{\lt}{\left}
\newcommand{\rt}{\right}

\newcommand{\beq}{\begin{equation}}
\newcommand{\eeq}{\end{equation}}
\newcommand{\bea}{\begin{aligned}}
\newcommand{\eea}{\end{aligned}}
\newcommand{\bal}{\begin{aligned}}
\newcommand{\eal}{\end{aligned}}

\usepackage{bbm}
\def\one{{\mathbbm{1}}}

\DeclareMathOperator{\supp}{supp}

\providecommand{\abs}[1]{\lvert#1\rvert}

\providecommand{\Bigabs}[1]{\Big|#1\Big|}
\providecommand{\Abs}[1]{\left\lvert#1\right\rvert}

\providecommand{\norm}[1]{\lVert#1\rVert}

\providecommand{\sprod}[2]{\langle#1,#2\rangle}

\newcommand{\mw}{\mathcal{W}}
\newcommand{\pa}{\partial}

\newtheorem{theo}{Theorem}
\newtheorem{lem}{Lemma}
\newtheorem{prop}{Proposition}
\newtheorem{coro}{Corollary}

\newtheorem{rem}{Remark}

\AtBeginDocument{

}
\usepackage[linecolor=blue!70!white, backgroundcolor=pink!20!white,bordercolor=red]{todonotes}
\let\origintodo\todo
\newcommand{\xtodo}[2][]{\origintodo[#1]{#2}\xspace}
\let\todo\xtodo
\definecolor{darkgreen}{rgb}{0.,0.7,0.1}

\begin{document}
\title[Linearly Transformed Particle method for aggregation equations]{Convergence of a Linearly Transformed Particle Method for Aggregation Equations}

\author[Campos Pinto]{Martin Campos Pinto}
\address{CNRS, UMR 7598, Laboratoire Jacques-Louis Lions, F-75005, Paris, France \newline
Sorbonne Universit\'es, UPMC Univ Paris 06, UMR 7598, Laboratoire Jacques-Louis Lions, F-75005, Paris, France}

\email{campos@ann.jussieu.fr}

\author[Carrillo]{Jos\'e A. Carrillo}
\address{Department of Mathematics, Imperial College London, SW7 2AZ, United Kingdom}
\email{carrillo@imperial.ac.uk}

\author[Charles]{Fr\'ed\'erique Charles}
\address{Sorbonne Universit\'es, UPMC Univ Paris 06, UMR 7598, Laboratoire Jacques-Louis Lions, F-75005, Paris, France \newline
CNRS, UMR 7598, Laboratoire Jacques-Louis Lions, F-75005, Paris, France}
\email{charles@ann.jussieu.fr}

\author[Choi]{Young-Pil Choi}
\address{Department of Mathematics, Imperial College London, SW7 2AZ, United Kingdom}
\email{young-pil.choi@imperial.ac.uk}


\date{}


\begin{abstract} We study a linearly transformed particle method for the aggregation equation with smooth or singular interaction forces. For the smooth interaction forces, we provide convergence estimates in $L^1$ and $L^\infty$ norms depending on the regularity of the initial data. Moreover, we give 
convergence estimates in bounded Lipschitz distance for measure valued solutions.
For singular interaction forces, we establish the convergence of the error between the approximated and exact flows up to the existence time of the solutions in $L^1 \cap L^p$ norm. 
\end{abstract}

\maketitle
%

\section{Introduction}

In this work, we are interested in showing the convergence of approximated particle schemes to the Cauchy problem for the so-called aggregation equation. This equation determines the evolution of a probability density $\rho(t,x)$ defined by
\begin{equation}\label{main-eq}
\left\{ \begin{array}{ll}
 \partial_t \rho(t,x) + \nabla \cdot (\rho u)(t,x) = 0, \quad x \in \R^d, \quad t >0, & \vspace{.2cm}\\
 u(t,x) = - (\nabla W * \rho(t)) (x),\quad x \in \R^d, \quad t >0, & \vspace{.2cm}\\
 \rho(0,x) = \rho^0(x)\geq 0, \quad x\in \R^d. & 
  \end{array} \right.
\end{equation}
Here, $-\nabla W(x-y)$ measures the interaction force that an infinitesimal particle located at $y\in\R^d$ will exert on a particle located at $x\in\R^d$. As a result, we will call $W$ the interaction potential. Since the total mass is preserved, without loss of generality, we assume
\begin{equation*}
 \int_{\R^d}  \rho(t,x)\,dx =\int_{\R^d} \rho^0(x)\,dx = 1 \qquad \forall t\geq 0.
\end{equation*}
The microscopic dynamics of $\mathcal N$ particles $X_i$, $i=1,\dots,\mathcal{N}$, interacting through the potential $W$ 
are given by
\begin{equation}\label{micro}
\dot{X}_i=-\sum_{j\neq i} m_j \;\nabla W (X_i-X_j) \,,\qquad i=1,\dots,\mathcal{N}\,,
\end{equation}
where the inertia term is assumed to be negligible compared to friction \cite{Mogilner2,Mogilner2003}. 
The macroscopic dynamics \eqref{main-eq} consists of a continuity equation where the velocity field is 
given by $u(t,x) = -(\nabla W * \rho(t))(x)$, which is the mean-field limit of the microscopic system 
when $\mathcal{N}\to \infty$ under certain conditions on the potential \cite{Dob,Golse,CCH,CDFLS}. 

Equation \eqref{main-eq} has attracted lots of attention in the recent years for three reasons: its gradient flow structure \cite{LT,Carrillo-McCann-Villani03,V,AGS,Carrillo-McCann-Villani06}, the blow-up dynamics for fully attractive potentials \cite{BCL,CDFLS,BLL,CJLV}, and the rich variety of steady states and their bifurcations both at the discrete \eqref{micro} and the continuous \eqref{main-eq} level of descriptions \cite{FellnerRaoul1,FellnerRaoul2,BT2,Raoul,CFP,CDFLS2,BLL,BCLR,BCLR2,soccerball,BUKB,BCY,CCP,CDM}. Furthermore, these systems are ubiquitous in mathematical modelling appearing in granular media models \cite{BCP,LT}, swarming models for animal collective behavior \cite{Dorsogna,KCBFL,CHM}, equilibrium states for self-assembly and molecules \cite{Wales1995,Rechtsman2010,Wales2010,Viral_Capside}, and mean-field games in socioeconomics \cite{DLR,BC} among others. 

We will focus the rest of the introduction on the well-posedness of the continuous equation \eqref{main-eq} and the numerical methods proposed for its approximation. The equation \eqref{main-eq} has the formal structure of being a gradient flow of a functional in the set of probability measures. Indeed, defining the interaction energy as
$$
\mathcal{F}[\mu]:=\frac12 \int_{\R^d}\int_{\R^d}W(x-y)\,d\mu(x)\,d\mu(y)
$$
for any probability measure $\mu$, we find $u=-\nabla \frac{\delta \mathcal{F}}{\delta \mu}$ where $\frac{\delta \mathcal{F}}{\delta \mu}$ is the formal variation of the functional $\mathcal{F}[\mu]$. This observation leads to a natural formal Lyapunov functional for the solutions of equation \eqref{main-eq}. In fact, we expect solutions to satisfy the identity
$$
\frac{d}{dt} \mathcal{F}[\rho(t)]=-\int_{\R^d} |\nabla W\ast \rho(t)|^2 \rho(t)\,dx
$$
for all $t\geq 0$. This structure can be rendered fully rigorous for $C^1$-potentials \cite{AGS} and it allows for mildly singular potentials at the origin \cite{CDFLS,CDFLS2,CJLV} provided the interaction potential has some convexity property called $\lambda$-convexity.

On the other hand, global in time unique weak measure solutions can be constructed for any probability measure as initial data under suitable smoothness assumptions on the interaction potential. In this work, whenever we refer to {\it smooth potentials}, we mean that the interaction potential satisfies $\nabla W \in \mw^{1,\infty}(\R^d)$. For smooth potentials, the approach introduced by Dobrushin for the Vlasov equation \cite{Dob} using the bounded Lipschitz distance between probability measures, see \cite{Golse,CCH,CCR} for further details, gives a well-posedness theory of weak measure solutions.

However, many of the interesting features related to blow-up dynamics and stationary states happen for potentials that are singular at the origin. 
Typical examples to bear in mind are combinations of repulsive attractive power-law potentials of the form $W(x)=\tfrac{|x|^a}{a}-\tfrac{|x|^b}{b}$ with $-d\leq b <a$ and the convention $\tfrac{|x|^0}{0}=\log |x|$, or fully attractive potentials $W(x)=\tfrac{|x|^a}{a}$ 
with $a>-d$, suitably cut-off at infinity. 
In this work, whenever we refer to {\it singular potentials} we mean that the interaction potential 
is not smooth but satisfies 
\begin{equation*}
|  \nabla W(x)|\leq \frac{C}{|x|^{\alpha}} \quad \textrm{and} \quad
|  D^2 W(x)|\leq \frac{C}{|x|^{1+\alpha}} \quad \mbox{with} \quad - 1 < \alpha  < d-1
\end{equation*}
for some constant $C>0$, and in addition we assume that $\nabla W$ is bounded away from the origin if
$\alpha < 0$. These conditions allow for singularities at the origin up to Newtonian but not including it. 
In particular, our singular potentials are such that $\nabla W \in \mw^{1,q}_{\rm loc}(\R^d)$ with a range depending on $\alpha$: 
$1 \le q < \frac{d}{\alpha+1}$. Note that the power-law potentials satisfy locally the conditions of being a singular potential in the range $2-d<b < 2$ for repulsive-attractive and in the range $2-d<a < 2$ for fully attractive. The various well-posedness theories for measure solutions fail as soon as the potential becomes singular at the origin.  However, weak solutions in Lebesgue spaces can be obtained. A local-in-time well-posedness theory was obtained in \cite{BTR,CCH} for initial data in $(L^1\cap L^{p})(\R^d)$ with $p = q'$ the conjugate exponent of $q$,
and in \cite{BCL,BLL} a local-in-time well-posedness theory for initial data in $(L^1\cap L^\infty) (\R^d)$ was developed  for singularities up to and including a Newtonian singularity at the origin, corresponding to 
$\alpha = d-1$. In this work, we will use the setting introduced in \cite{CCH}. The Newtonian case is very specific because of the relation between the divergence of the velocity field and the density becomes local.

Under the above assumptions of either smooth or singular potentials, the proofs of the global-in-time well-posedness of weak measure solutions and the local-in-time well-posedness of weak solutions for initial data in $(L^1 \cap L^p)(\R^d)$ spaces are essentially based on the fact that the velocity field is regular enough to have meaningful characteristics. It is proved in \cite{Dob, Golse, BTR,CCH} that the velocity field of the constructed solutions is continuous in time and Lipschitz continuous in space. Then, the flow map $\Phi_t (x)$, defined by the unique solution of the characteristic system
$$
\left\{
\begin{aligned}
&\frac{dX}{dt}(t)=u(t,X(t)),\\
&X(s)=x,
\end{aligned}
\right.
$$
is a diffeomorphism for all times $t\geq 0$. In all cases, the solution built in \cite{Dob,Golse,BTR,CCH} is obtained by characteristics and given by $\rho(t) = \Phi_t \# \rho^0$. Here,  $\mathcal{T}\#\mu$ denotes the push-forward of a measure through a measurable map $\mathcal{T}:\R^d \longrightarrow \R^d$ defined as $\mathcal{T}\#\mu [K] := \mu[\mathcal{T}^{-1}(K)]$ for all Borel sets $K\subset\R^d$, or equivalently
$$ 
\int_{\R^d} \varphi\, d(\mathcal{T}\#\mu) = \int_{\R^d} (\varphi\circ \mathcal{T})\, d\mu \qquad \mbox{for all } \varphi\in C_b(\R^d)\,.
$$

A very interesting question is the rigorous derivation of the continuum description \eqref{main-eq} starting from the microscopic dynamics \eqref{micro} for both regular and singular potentials. This is the so-called mean-field limit problem. The mean-field limit results contain as a by-product convergence results for the classical particle method. More precisely, proving that \eqref{main-eq} is the mean-field limit of the system \eqref{micro} as $\mathcal{N}\to \infty$ is equivalent to show that the empirical measure
$$
\mu_{\mathcal{N}}(t) = \frac1{\mathcal{N}}\sum_{i=1}^\mathcal{N} \delta_{X_i(t)}
$$
converges weakly in measure sense to the solution of \eqref{main-eq} provided that this weak convergence holds initially. Even if the particle method is proved to be convergent of order $\frac1{\mathcal{N}}$, the convergence error is only controlled in the bounded Lipschitz or Wasserstein-type distances between measures \cite{Dob,Golse,CCH,CDFLS}.

Vortex-blob methods, originally introduced for the 2D  Euler equations for incompressible fluids, see \cite{MB} and the references therein, have also been adapted recently to the aggregation equation \cite{BC-blob} with fixed shapes, where the approximate densities are shown to converge with arbitrary orders but only in negative Sobolev norms.
Particle methods were also used in plasma physics for the Vlasov-Poisson system \cite{CR}, 
where they are usually called smooth Particle In Cell (PIC) methods. 

In the Linearly Transform Particle (LTP) method, introduced by Campos Pinto in \cite{Campos} following an idea of Cohen and Perthame \cite{CP},  particles are pushed on to discrete times according to an approximation of  the exact flow as in standard particle methods. Moreover, particles have their own shape, which is transformed in the discrete evolution in order to better approach the local flow using a linearization of the exact flow. To our knowledge the LTP method has only been used for a linear transport equation \cite{CP} or for a Vlasov-Poisson system \cite{CC} involving measure-preserving characteristic flows. The technical difficulties posed by the deformation of the flows in our present case have been overcome by detailed estimates of the Jacobian matrices and determinants. These estimates have allowed us to control the error on the densities via the errors of the flows to finally obtain the convergence results. Certain Sobolev regularity is needed on the initial data to obtain convergence of the LTP method in Lebesgue spaces for both smooth and singular potentials. 
However, a general result of convergence for weak measure solutions is obtained in an appropriate distance for measures. 
    
Let us finally mention that other numerical methods have been proposed in the literature for the aggregation equation. In \cite{CCH2}, the authors proposed a finite volume scheme which is shown to be energy preserving, i.e., it keeps the property that the energy functional is dissipated along the semidiscrete flow. Finite volume and finite difference schemes have been shown to be convergent to weak measure solutions of the aggregation equations for mildly singular potentials in \cite{JV,CJLV}.

In this work, we extend the LTP method to the aggregation equation seen as one of the most important representatives of a class of nonlinear continuity equations with non divergence free velocity fields in any dimensions. We start by summarizing the basic ideas of the numerical LTP method in Section 2 together with the preliminaries and notations used in this work. Section 3 is devoted to give convergence results for smooth potentials in Lebesgue spaces. Depending on the regularity of the initial data, we will be able for smooth potentials to control errors in $L^1$ and $L^\infty$. For initial data just being a probability measure, we will show in Section 4 the convergence in bounded Lipschitz distance. In the case of singular potentials, we will control in Section 5 the error upto the existence time of the solution of \eqref{main-eq} in $L^1$ and $L^p$ with $p$ suitably chosen. We finally show in Section 6 the performance of this method in one dimension validating the numerical implementation with explicit solutions and making use of it to study certain not well-known qualitative features of the evolution of \eqref{main-eq} with several smooth and singular potentials.

\section{Preliminaries}
\setcounter{equation}{0}

\subsection{Basic properties of the exact flow}
In the setting of our main results, the velocity field of the exact solution to \eqref{main-eq} is always continuous in $t$ and Lipschitz continuous in $x$. The solution of the characteristic system
\begin{equation*} 
\left\{
\begin{aligned}
&\frac{dX}{dt}(t)=u(t,X(t))\\
&X(s)=x,
\end{aligned}
\right.
\end{equation*} 
is well-defined and it has unique global in time solutions for all initial data $x\in\R^d$. Moreover, the general solution of the characteristic system is a diffeomorphism in $\R^d$. The general flow map will be denoted by $F^{s,t}(x)$ for all $t,s\in\R$ and $x\in \R^d$. 

As discussed in the introduction, the solutions to \eqref{main-eq} can always be expressed as $\rho(t) = F^{0,t} \# \rho^0$ or equivalently as
$$
\rho(t,x)=\rho^0\left(F^{t,0}(x)\right) j^{t,0}(x)
\quad
\textrm{with} \quad
 j^{t,0}(x) = \det( J^{t,0}(x)),\quad  J^{t,0}(x) =   D F^{t,0}(x)\,.
$$
The flow map satisfies
\begin{equation} \label{expFex}
 F^{s,t}(x) = x + \int_s^t u(\tau,F^{s,\tau}(x))d\tau
            = x - \int_s^t (\nabla W * \rho(\tau))(F^{s,\tau}(x))d\tau,
\end{equation}
and the Jacobian matrix and its determinant satisfy the differential equations
\begin{equation}\label{matrixodes}
\frac{d}{dt}J^{s,t}(x) = D u(t,F^{s,t}(x)) J^{s,t}(x) \quad \mbox{and} \quad \frac{d}{dt}j^{s,t}(x) = \nabla \cdot u(t,F^{s,t}(x)) j^{s,t}(x).
\end{equation}
Using  $u(\tau,y) = -(\nabla W * \rho(\tau))(y)$, this yields 
\begin{align}
J^{s,t}(x) - I_{d} &= \int_s^t \! D u(\tau, F^{s,\tau}(x)) J^{s,\tau}(x)d\tau \nonumber\\
           &= - \int_s^t \! (D^2 W * \rho(\tau))(F^{s,\tau}(x)) J^{s,\tau}(x)d\tau \label{expJex}
\end{align}
and
\begin{align}
j^{s,t}(x) &= \exp\lt( -\int_s^t \nabla \cdot u(\tau,F^{s,\tau}(x) d\tau\rt) \nonumber \\
&= \exp\lt( -\int_s^t (\Delta_x W * \rho(\tau))(F^{s,\tau}(x)) d\tau\rt).\label{exp_ja}
\end{align}
Estimates are then easily derived when
$u \in L^\infty(0,\infty;\mw^{1,\infty}(\R^d))$. We will write $L:=\sup_{t\in[0,\infty)} \|u(t,\cdot)\|_{\mw^{1,\infty}}$. For instance, using \eqref{matrixodes} and $J^{s,s}(x) = I_d$ we find
\begin{equation}\label{majJex}
\sup_{x \in \R^d} \abs{J^{s,t}(x)} \leq \exp \lt( C L \abs{t-s}\rt),
\end{equation}
and in particular the characteristic flow is Lipschitz (relative to any norm in $\R^d$),
\begin{equation}  \label{majFexlip} 
|F^{s,t}|_{Lip} \leq \exp \lt( C L \abs{t-s}\rt).
\end{equation}
Furthermore, we derive from \eqref{expJex} and \eqref{majJex} that
\begin{equation} \label{maj1moinsJex}
\sup_{x \in \R^d} \abs{I_{d} - J^{s,t}(x)} \leq (t-s)\exp \left(  C L \abs{t-s}\right)
\end{equation}
and using \eqref{exp_ja} we also find
\beq\label{est_reja}
\exp\lt(-C L \abs{t-s}\rt) \le j^{s,t}(x) \leq \exp\lt(C L \abs{t-s}\rt) \quad \text{ for } \quad x \in \R^d
\eeq
and 
\beq\label{est_reja-1}
\|j^{s,t}-1\|_{L^\infty} \leq C L |t-s| \exp\lt(C L(t-s)\rt).
\eeq
Let us remark that the previous estimates \eqref{majJex}-\eqref{est_reja-1} can also be obtained in a time interval $[0,T]$ for locally Lipschitz velocity fields $u \in L^\infty(0,T;\mw^{1,\infty}(\R^d))$ for some $T>0$, with constant $L_T:=\sup_{t\in[0,T]} \|u(t,\cdot)\|_{\mw^{1,\infty}}$. These estimates will be used in Section 5, where the dependence on T of the Lipschitz constant will be omitted for the sake of simplicity.

\subsection{Linearly Transformed Particles} \label{secmethodeLTP}

As in standard particle methods, the density $\rho$ is represented with weighted macro-particles, and as in smooth particle methods, particles have here a finite and smooth shape. 
Thus, we approximate the initial density $\rho^0$ on a Cartesian grid of size $h > 0$ by
\begin{equation}
 \label{defrhoh0}
\rho_{h}^0(x)=\sum_{k\in \Z^d} \omega_k \varphi_{h,k}^0(x)
\end{equation}
with particle shapes obtained by scaling and translating a reference function, i.e.,
\begin{equation}
 \label{phik0}
\varphi_{h,k}^0(x)=\frac{1}{h^d} \varphi\left(\frac{x-x_k^0}{h}\right),
\qquad x_k^0=kh.
 \end{equation}
Here the reference shape is assumed to have a compact support $\supp(\varphi) \subset B(0,R_o)$,
be bounded and satisfy
$$ \label{prop-phi}
\sum_{k\in\Z^d} \varphi(x - k) = 1 \quad \text{for $x \in \R^d$}
\qquad \text{ and } \qquad 
\int_{\R^d} \varphi =1.
$$  
In this work we will require that the shape functions are Lipschitz, and
we can either consider for the reference shape the tensor-product hat function 
\begin{equation} \label{B1-spline}
    \varphi(x) = \prod_{1\le i\le d} \max(1-|x_i|,0).   
\end{equation} 
or the  B3-spline
 \begin{equation}\label{B3-spline}
\varphi(x)= \frac{1}{6}\left\{ 
\begin{array}{cc}
\left(2-|x|\right)^3                               
    \quad & \textrm{if }1 \leq |x| < 2,   \\
   \displaystyle  4-6x^2 +3 |x|^3
    \quad & \textrm{if } 0\leq |x| < 1,   \\
    0  \quad & \textrm {otherwise.}
\end{array}
\right.
 \end{equation}
As for the weights $\omega_k = \omega_k(h,\rho^0)$, they are usually defined as 
\beq \label{simple-wk}
\omega_k = \int_{x_k^0 + \lt[-\frac h2, \frac h2\rt]^d} \rho^0(x) dx,
\eeq
however this will not be sufficient to prove the convergence of our particle scheme 
without additional smoothness assumptions on the initial density $\rho^0$. Indeed, using standard arguments (see e.g. \cite{Campos, Ciarlet:1991}) based on the fact that the approximation $\rho^0 \mapsto \rho^0_h = \sum_{k\in \Z^d} \omega_k \varphi_{h,k}^0$ is local, bounded in any $L^p$ space and preserves the affine functions, one easily verifies the following estimate.
    \begin{prop}\label{prop-pos-weights} If $\rho^0_h$ is initialized as in \eqref{defrhoh0} with 
    weights and shape function given by \eqref{simple-wk} and \eqref{phik0}, respectively, 
    then we have
    \beq \label{rho0-error}
    \norm{\rho^0-\rho^0_h}_{L^p} \le C h^s \norm{\rho^0}_{\mw^{s,p}}
    \eeq
    for $s \in \{0, 1, 2\}$, $1\le p \le \infty$ and a constant $C$ independent of $\rho^0$.
    \end{prop}
    In our analysis we will need second-order estimates which are then available for $\rho^0 \in \mw^{2,p}(\R^d)$.
    However, if we allow negative weights then second-order estimates are also available in a dual norm,
    as follows. Consider weights defined as 
    \beq \label{defweights}
        \omega_k = \int_{\R^d} \tilde \varphi_{h,k}^0(x) \rho^0(x)dx
    \eeq
    with integration kernels bi-orthogonal
    to the shape functions in the sense that
    \begin{equation} \label{duality}
    \int_{\R^d} \varphi_{h,k}^0 \tilde \varphi_{h,k'}^0 = \delta_{k,k'}
    \end{equation}
    holds with $\delta_{k,k'}$ the Kronecker symbol.
    Similar to the shape functions, they can be obtained by scaling and translating 
    a reference $\tilde \varphi$ (assumed again compactly supported, bounded and satisfying \eqref{prop-phi})
    with a different normalization, namely
    \beq \label{tilde-phik0}
    \tilde \varphi_{h,k}^0(x)=\tilde \varphi\left(\frac{x-x_k^0}{h}\right).
    \eeq
For instance if $\varphi$ is  the above tensor-product hat function \eqref{B1-spline}
 then for the integration kernel we may take 
    $\tilde \varphi(x) = \prod_{i\le d} 
    \big(\frac 32 \one_{[-\frac 12,\frac 12]} - \tfrac 12 \one_{[-1,-\frac 12] \cup [\frac 12,1]}\big)(x_i),$ 
see Figure \ref{figB1etB1ortho}.

\begin{figure}[h]
\centering
\includegraphics[width=5cm]{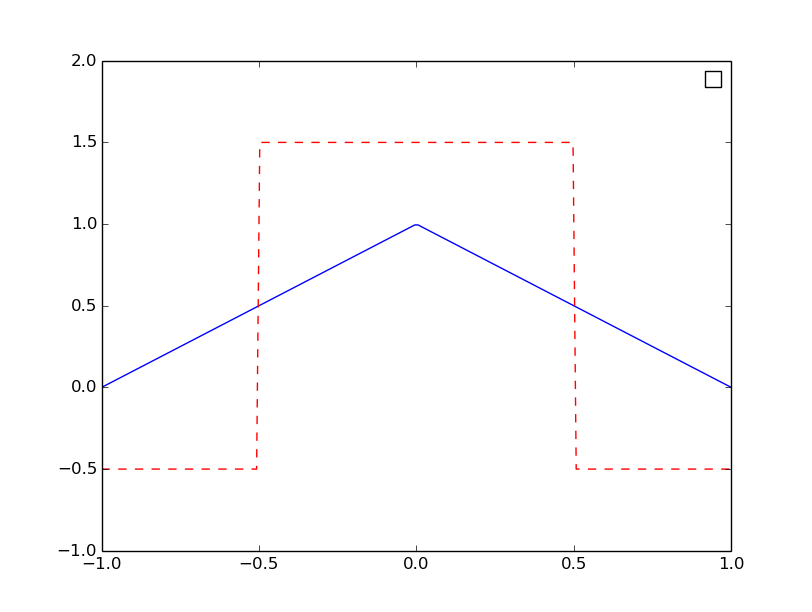}
\caption{\label{figB1etB1ortho} A piecewise affine shape function and its bi-orthogonal kernel (dotted line). 
Both functions vanish outside $[-1,1]$.}
\end{figure}   

Notice that estimate~\eqref{rho0-error} still holds with these weights.
    Now, from the duality \eqref{duality} we can derive a convenient second-order estimate
    which only relies on the first-order smoothness of $\rho^0$. It is expressed
in the dual norm
$$
\ds\norm{w}_{\mw^{-1,p}} := \sup_{v \in \mw^{1,q}(\R^d)} \sprod{w}{v} / \norm{v}_{\mw^{1,q}}\,,
$$
where $q$ is the conjugate exponent of $p$ and 
$\sprod{w}{v}$ is the duality pair that coincides with the integral of the product $wv$ 
as soon as the latter is integrable. 

\begin{prop}\label{prop_ini_app}If $\rho^0_h$ is initialized as in \eqref{defrhoh0} with shape functions 
and weights satisfying properties \eqref{phik0}-\eqref{prop-phi} and \eqref{defweights}-\eqref{tilde-phik0}, 
we have 
\beq
\label{rho0-error-neg}
\norm{\rho^0-\rho^0_h}_{\mw^{-1,p}} \le C h^2 \norm{\rho^0}_{\mw^{1,p}}
\eeq    
for $1 \le p \le \infty$, with a constant $C$ independent of $h$.
\end{prop} 
\begin{proof}
It follows from the duality relation \eqref{duality} that
$\sprod{\rho^0 - \rho^0_h}{\tilde \varphi_{h,k}^0} = 0$ for all $k$. In particular,
given $v \in \mw^{1,\infty}(\R^d)$ we have    
$$
\sprod{\rho^0 - \rho^0_h}{v} = \sprod{\rho^0 - \rho^0_h}{v - \tilde v_h} 
$$
with $\tilde v_h := \sum_{k \in \Z^d} \sprod{v}{\varphi_{h,k}^0} \tilde \varphi_{h,k}^0$ and standard arguments
show that the approximation $v \mapsto \tilde v_h$ satisfies an error estimate similar to \eqref{rho0-error}
for $s=1$.
Using the H{\"o}lder inequality this gives
$$
\sprod{\rho^0 - \rho^0_h}{v} \le \norm{\rho^0 - \rho^0_h}_{L^p} \norm{v - \tilde v_h}_{L^q}
\le C h^2 \norm{\rho^0}_{\mw^{1,p}}\norm{v}_{\mw^{1,q}}
$$
and the proof is completed due to the definition of the $\mw^{-1,p}(\R^d)$ norm.
\end{proof}

We observe that both the above initializations yield
\begin{equation} \label{majweight}
\sup_{k\in \Z^d} \abs{\omega_k}  \leq  C h^{d/q} \|\rho^0\|_{L^{p}} \qquad \text{ where } \qquad \tfrac 1q + \tfrac 1p = 1\,,
\end{equation}
and since the shape functions are assumed to be non-negative, \eqref{prop-phi} gives
\begin{equation} \label{majweight2}
\| \rho_h^0\|_{L^1} \le \sum_{k\in \Z^d} \abs{\omega_k} \le C \| \rho^0 \|_{L^1} \le C\,,
\end{equation}
with a constant depending only on $\tilde \varphi$.

We now describe the LTP method. As mentioned in the introduction, compared to standard particle methods, the LTP method follows the shape of smooth particles. Therefore we need to track not only the particle positions but also their deformations given by the Jacobian matrices.
Given discrete trajectories $x^n_k$ approximating the exact ones $F^{0,t_n}(x^0_k)$ on the discrete times 
$$
t_n := n\Delta t,  \qquad n = 0, 1, \ldots, N := T / \Delta t,
$$ 
and non singular approximations $J^n_k$ of the forward Jacobian matrices
$J^{t_n,t_{n+1}}(x^n_k)$, 
the particle shapes $\varphi^{n+1}_{h,k}$ are recursively defined as the push-forward of $\varphi^{n}_{h,k}$ along the affine flow 
\begin{equation} \label{Fnk}
    F_{h,k}^{n}: \, x \mapsto x_k^{n+1} + J_{k}^{n} (x-x_k^{n})\,,
\end{equation}
which approximates the exact flow $F^{t_{n},t_{n+1}}$ around $x_k^{n}$.
Here $x_k^{n+1}$ can also be seen as an approximation to $F^{t_{n},t_{n+1}}(x_k^{n})$, as will be specified below. In short, we define
\begin{equation*}
\varphi^{n+1}_{h,k} := F_{h,k}^{n}  \# \varphi^{n}_{h,k} = \frac{1}{j_k^{n}} \varphi^{n}_{h,k}\circ (F_{h,k}^{n})^{-1}\,,
\end{equation*}
where $j_k^{n}:=\det(J_k^n) > 0$. Starting from $\varphi^{0}_{h,k}$ defined as in \eqref{phik0}, this gives particles of the form
\begin{equation}
 \label{phikn}
     \varphi_{h,k}^n(x) := \frac{1}{h_k^n} \varphi \left( \frac{D_k^n (x- x_k^n)}{h}  \right)\,,
\end{equation}
where the deformation matrix $D_k^{n}$ and the particle volume $h_k^{n}$ are defined by
\begin{equation} \label{defhkn}
    \begin{cases}
    D_k^{n+1} := D_k^{n} \left( J_k^{n}\right)^{-1}
    \\
    h_k^{n+1} := j_{k}^{n} h_k^{n} = \det(J_k^{n}) h_k^{n}
    \end{cases}
    \text{ with } \quad
    \begin{cases}
    D_k^{0} := I_{d}
    \\
    h_k^{0} := h^d
    \end{cases}\,.
\end{equation} 
It follows from the above process that $D_k^n$ is an approximation to the backward 
Jacobian matrix $J^{t_n,0}(x^n_k)$, whereas $h_k^n$ approximates the elementary volume 
$h^d$ multiplied by the Jacobian determinant of the forward flow $F^{0,t_n}$ 
at $x^0_k$. Moreover, the particle shape $\varphi_{h,k}^n$ is the push-forward of 
$\varphi^0_{h,k}$ along the integrated flow 
\beq\label{def_fbar}
\overline F_{h,k}^n := F^{n-1}_{h,k} \circ \cdots \circ F^{0}_{h,k} : x \mapsto x^n_k + \overline J^n_k (x-x^0_k)
\quad \text{ where } \quad  \overline J^n_k := (D^n_k)^{-1}
\eeq
which can be seen as a linearization of $F^{0,t_n}$ around $x^0_k$ (for $n=0$ we set
$\overline F_{h,k}^0 = I$ since $D^0_k = I_d$).
Indeed, it follows from the above definitions that
\beq \label{phikn-push-phi0k}
    \varphi_{h,k}^n = \overline F_{h,k}^{n}  \# \varphi^{0}_{h,k}\,,
\eeq
and we easily verify that 
$$
h_k^n = h^d
\det(J_{h,k}^{n-1}) \, \cdots \, \det(J_{h,k}^0)  =  \frac{h^d}{\det(D_{h,k}^n)} \approx h^d \det(J^{0,t_n}(x^0_k)).
$$
Finally, the LTP approximation of the density at time $t_n$ is defined as
\begin{equation}\label{defrhohn}
        \rho_{h}^n(x) := \sum_{k\in \Z^d} \omega_k \varphi_{h,k}^n(x)
\end{equation}
with weights $\omega_k$ constant in time and computed as in \eqref{simple-wk} or \eqref{defweights}.
According to \eqref{phikn-push-phi0k}, we have 
    $\int \varphi^{n}_{h,k} =\int \varphi_{h,k}^0 =\int \varphi $, 
    and thus the conservation of mass ($\int \rho_{h}^n = \int \rho_{h}^0$) holds at the discrete level. Moreover,
    using the fact that the particle shapes are non-negative, we find as in \eqref{majweight2}
\begin{equation} \label{L1rhon}
    \| \rho_h^n\|_{L^1} \le \sum_{k\in \Z^d} \norm{\omega_k \varphi^{n}_{h,k}}_{L^1} 
    = \sum_{k\in \Z^d} \abs{\omega_k} \le C \| \rho^0 \|_{L^1} = C, \qquad  n\geq 0.
\end{equation}

\subsection{Approximated Jacobian matrices and particle positions} 

To complete the description of the numerical method 
\eqref{phikn}-\eqref{defhkn}, \eqref{defrhohn},
we are left to specify how to compute the particle center $x^{n+1}_k$ and the discrete Jacobian matrix $J^n_k$ involved in the affine flow \eqref{Fnk}.
Before doing so we observe that if the matrices $D^2W(x)$ and $D^2W(y)$ commute for all $x$ and $y$, then  the exact solution to the ODE \eqref{matrixodes} takes an exponential form. However, in the general case the matrix $J^{t_n,t_{n+1}}(x)$ is {\it not} equal to
\beq \label{expJtilde}
\tilde J^{t_n,t_{n+1}}(x) := \exp \left( -\int_{t_n}^{t_{n+1}} (D^2 W * \rho(\tau))(F^{t_n,\tau}(x))d\tau \right)
\eeq
but the difference is small, as shown next.

\begin{prop} \label{propJexJtilde}
If $u \in L^{\infty}(0,T;\mw^{1,\infty}(\R^d))$, then we have
$$
\abs{\tilde J^{t_n,t_{n+1}}(x) - J^{t_n,t_{n+1}}(x) } \le C (\Delta t)^2 \quad \text{ for } \quad x \in \R^d,
$$
with a constant $C$ independent of $n \le N-1$ and $\Delta t$.
\end{prop}

\begin{proof}
Given $n \le N-1$ and $x \in \R^d$, we denote for simplicity
$$
B(\tau) = B(\tau, t_n, x) := (D^2 W * \rho(\tau))(F^{t_n,\tau}(x))
$$
and we observe that  $\abs{B(\tau)} \le L = \sup_{t\le T} \abs{u(t)}_{\mw^{1,\infty}}$
for all  $\tau \in [t_n,t_{n+1}]$. From \eqref{expJex} we have $J^{t_n,t_{n+1}}(x) = I_d - \int_{t_n}^{t_{n+1}} B(\tau) d\tau + \int_{t_n}^{t_{n+1}} B(\tau)(I_d - J^{t_n,\tau}(x)) d\tau$, hence the difference $E(x) := \tilde J^{t_n,t_{n+1}}(x) - J^{t_n,t_{n+1}}(x)$ can be decomposed into
$$ 
E(x)
    = \underbrace{\sum_{m=2}^\infty \frac{(-1)^m}{m !} \lt( \int_{t_n}^{t_{n+1}} \mspace{-10mu} B(\tau) d\tau\rt)^m}_{=: (a)}
    + \underbrace{\int_{t_n}^{t_{n+1}}\mspace{-10mu}  B(\tau) (I_{d} - J^{t_n,\tau}(x))d\tau}_{=: (b)}.
$$
From the above bound for $B$ we readily find
$(a) \leq \sum_{m=2}^\infty \frac{1}{m !}(C\Delta t)^m \leq C(\Delta t)^2 $.
Turning to $(b)$, we use again \eqref{expJex} to write
$$
\abs{(b)} = \lt| \int_{t_n}^{t_{n+1}} \mspace{-10mu} B(\tau)
                  \lt( \int_{t_n}^\tau \mspace{-5mu} B(t) J^{t_n,t}(x) dt \rt) d\tau \rt|
    \leq C \int_{t_n}^{t_{n+1}} \int_{t_n}^\tau \abs{J^{t_n,t}(x)} dt d\tau
    \leq C  (\Delta t)^2
$$
where we have used \eqref{majJex} in the last inequality. The result follows.
\end{proof}

At time $t_{n+1}$, $x_k^{n+1}$ is an approximation of  $F^{t_{n},t_{n+1}} (x_k^{n})$ 
which is the solution at time $t_{n+1}$ of the ODE
\begin{equation} \label{eqcharaapproached}
\left\{
\begin{aligned}
&\frac{d\tilde{X}_{k}(t)}{dt}= u(t,\tilde{X}_{k}(t)) = - (\nabla W * \rho(t))(\tilde{X}_{k}(t)),\\
&\tilde{X}_k(t_n)=x_k^n.
\end{aligned}
\right.
\end{equation} 
Then we can define $x_k^{n+1}$ as the approximation given by a numerical scheme discretizing \eqref{eqcharaapproached} when replacing the exact density $\rho$ at discrete times in $[t_n,t_{n+1}]$ by its LTP approximation $\rho_h^n$. 
In the convergence analysis, we consider particle trajectories $x_k^n$ and approached Jacobian matrices $J_k^n$ defined by an explicit Euler scheme:
\beq \label{schemEuler}
\left\{
\begin{aligned}
&x_k^{n+1}  :=x_k^n -\Delta t \,(\nabla W * \rho_h^n)(x_k^n),\\
&J_k^{n}   := e^{- \Delta t \,(D^2W * \rho_h^n)(x_k^n)} =\sum_{m=0}^\infty \frac{(-1)^m}{m !} 
        \left[\Delta t \, ((D^2 W * \rho^n_h)(x_k^n)\right]^m.
\end{aligned}
\right.
\eeq
Note that this expression can be seen as an approximation to \eqref{expJtilde} using 
a rectangular rule in the time integral (here will not take into account the approximation error 
of convolution products). Accordingly, we set
\beq\label{est_jk}
j_k^n = \det (J_k^n) 
        = \exp \lt( -\Delta t (\Delta_x W * \rho^n_h)(x_k^n)\rt).
\eeq
Using \eqref{majD2Wrho} and the $L^1$ bound \eqref{L1rhon} on $\rho^n_h$, we see that this approximation yields
$$ 
| J_k^n - I_d | = \lt| \sum_{m=1}^\infty \frac{(-1)^m}{m !} (\Delta t)^m ((D^2 W * \rho^n_h)(x_k^n))^m\rt| \leq \sum_{m=1}^\infty \frac{1}{m !}(C\Delta t)^m \leq C\Delta t e^{C \Delta t}.
$$
Clearly, higher-order time discretizations are also possible. 

\begin{rem}
When $d>1$, computing the exponential of a $d \times d$ matrix is costly. Another possibility is to 
approximate $J^{t_n,t_{n+1}}(x_k^n)$ by
$$
 \tilde{J_k^n} = I_{d} - \Delta t \,  (D^2W * \rho^n_h)(x_k^n)\,.
$$
It is easily verified that the difference between these approximations satisfies 
$$
\sup_{0\leq n\leq \frac{T}{\Delta t}} \sup_{k\in \Z} 
\left\| \tilde{J_k^n} - J_k^n   \right\|
= {\mathcal O} \left({\Delta t^2} \right)
$$
as long as we have $\nabla W\in \mw^{1,q}(\R^d)$ and $\sup_{0\leq n\leq \frac{T}{\Delta t}} \|\rho_h^n \|_{L^p} \le C$ with $p = q'$. 
\end{rem}

\subsection{General strategy of the convergence proofs}

In order to establish error estimates for the approximation of the density $\rho(t_n)$ by $\rho^n_h$ 
we will use Gronwall arguments that involve errors on the flows and on the Jacobian determinants. 
Since the velocity fields depend nonlinearly on the densities, we need to couple these errors 
with the density approximation error, and since the $k$-th particle is pushed forward by the 
approximated flow $F_{h,k}^n$ during the time interval $[t_n,t_{n+1}]$, we need to control the
local error between this approximation and the exact flow $F^{t_n,t_{n+1}}$. To this end we define
a first error term on the support of the smooth particles,
\begin{equation} \label{eFn}
e_F^n        := \sup_{k\in \Z^d} \| F^{t_{n},t_{n+1}} -F_{h,k}^{n}  \|_{L^{\infty}(S^n_{h,k}) } 
\quad \text{ with } \quad S^n_{h,k} := \supp\big(\varphi_{h,k}^n \big)\,.
\end{equation}
In our analysis, we shall also need to track the error on an extended domain which accounts for the deformation of the particle support by the exact flow, namely
\begin{equation} \label{tilde-eFn}
\tilde e_F^n := \sup_{k\in \Z^d} \| F^{t_{n},t_{n+1}} -F_{h,k}^{n}  \|_{L^{\infty}(\tilde S^n_{h,k}) } 
\,\, \text{ with } \,\,  \tilde S^n_{h,k} := S^n_{h,k} \cup F^{t_{n+1},t_n}(S^{n+1}_{h,k}).
\end{equation}
The error corresponding to the integrated flow \eqref{def_fbar} is then defined as
\begin{equation*}
\overline{e_F}^n:= \sup_{k\in \Z^d} \| F^{0,t_{n}} -\overline{F}_{h,k}^{n}  \|_{L^{\infty}(S^0_{h,k}) }.
\end{equation*}
Using the fact that the exact 
flow is Lipschitz, see \eqref{majFexlip}, it is easy to bound this term by accumulating the local flow errors,
$\overline{e_F}^n \le  C \exp(CT) (e^0_F + \cdots + e^{n-1}_F)$,
but in the analysis we will need a finer control, see Lemma~\ref{estbareFn} below.
We will also need to control the error of the Jacobian determinants for each particle, thus we define   
\begin{equation}\label{def_ejn}
e_j^n:= \sup_{k\in \Z^d} \left\| \frac{1}{ j^{t_{n},t_{n+1}}(x)}-\frac{1}{j_{k}^{n}}  \right\|_{L^{\infty}(S^n_{h,k}) }.
\end{equation}

Finally we will need to track carefully the particles that affect the local value of the approximated density. 
For this purpose, we let 
$$ 
\mathcal{K}_{n}(x):= \{k\in \Z^d: \, x\in S^n_{h,k} \}.
$$

\section{$L^1$ and $L^\infty$ convergence for smooth potentials} \label{sec:conv}
\setcounter{equation}{0}

In this section we assume that the potential is smooth, as defined in the introduction. This means that 
$\nabla W \in \mw^{1,\infty}(\R^d)$.
In this case, the Lipschitz norm of $u$ is bounded by $\|\nabla W\|_{\mw^{1,\infty}}$: indeed letting $\abs{\cdot}$ denote the Euclidean norm in $\R^d$ as well as its associated matrix norm,
we have for all $x \in \R^d$, $t \in [0,T]$,
\begin{align} \label{majD2Wrho}
\begin{aligned}
\abs{Du(t,x)} &= \abs{(D^2 W * \rho(t)) (x)}
    \\
    &\le C \max_{1 \leq i,j \leq d} \abs{(\partial_{ij} W * \rho (t))(x)}
    \leq C \|\rho^0\|_{L^1} \|\nabla W\|_{\mw^{1,\infty}}.
\end{aligned}
\end{align} 
and similarly for $u$, so that estimates~\eqref{majJex}-\eqref{est_reja-1} hold with $L = C \|\nabla W\|_{\mw^{1,\infty}}$. However, to obtain convergence rates in $L^p$-spaces we need more regularity on the solutions. In turn we assume that $\rho^0 \in \mw^{1,1}_+(\R^d)$ in this section
    and we compute the weights with the formula \eqref{defweights} involving the dual kernels.
According to the propagation of regularity of solutions to \eqref{main-eq} in Proposition \ref{rhoW11} in the Appendix, this ensures that the unique solution to \eqref{main-eq} satisfies $\rho \in L^\infty(0,T;\mw^{1,1}(\R^d))$ for all $T>0$. 

Given the solution $\rho$ to \eqref{main-eq}, we will use the shortcut notation, $\rho^n(x):= \rho(t_n,x)$ for $x \in \R^d$.  From now on, $C$ denotes a generic constant independent of $h$ and $\Delta t$, depending only on $L= \sup_{t\le T} \abs{u(t)}_{\mw^{1,\infty}}$, $d$ and the exact solution.

Moreover, we assume that both $h$ and $\Delta t$ are bounded by an absolute constant.
We denote by 
\begin{equation} 
\label{defthetan}
\theta_n := \|\rho^n - \rho^n_h\|_{L^1}, \qquad  \tilde \theta_n := \max_{0 \leq m \leq n} \theta_m, \quad \mbox{and} \quad \varepsilon_n := \|\rho^n - \rho^n_h\|_{L^{\infty}}
 \end{equation}
the errors in $L^1$ and $L^{\infty}$ norms.

\subsection{Estimates on the flows and related terms }
%
We first control the particle overlapping from the approximation error on the flow.
\begin{lem}\label{lem_majKn}
There exists a constant $C$ independent on $h$ and $\Delta t$ such that
\begin{equation} 
 \label{majKn}
\kappa_n :=\sup_{x\in \R^d} \, \# \mathcal{K}_{n}(x) \leq C \left( 1+ \frac{\overline{e_F}^n}{h}\right)^d.
\end{equation}
\end{lem}
\begin{proof}
Given $x\in \R^d$ and $k\in  \mathcal{K}_{n}(x)$, we denote $z= F^{t_{n},0}(x)$ and $z_k=\big(\overline{F}_{h,k}^{n}\big)^{-1}(x)$.
From \eqref{def_fbar} we see that $z_k \in S^0_{h,k}$. Using the Lipschitz bound \eqref{majFexlip} we then write
$$ \begin{aligned}  |z- kh | 
& \leq  |z- z_k |+ |z_k- kh | \leq \left|F^{t_{n},0}
   \left(\overline{F}_{h,k}^{n}(z_k))- F^{0,t_{n}}(z_k) \right) \right|+  |z_k-x_k^0 | \\
&  \leq \left|F^{t_{n},0} \right|_{Lip}  \overline{e_F}^n + C h 
    \leq C( \overline{e_F}^n +  h).
\end{aligned}
$$
This gives
$   
\left|k -\frac{z}{h} \right| \leq C \big( 1 + \frac{ \overline{e_F}^n}{h}\big),
$
and the result follows.
\end{proof}

Using the formulas \eqref{schemEuler}, \eqref{est_jk} and the a priori $L^1$ bound \eqref{L1rhon} on the approximated
densities $\rho^n_h$ we easily derive uniform estimates for the approximated Jacobian matrices and the particle supports.
%
\begin{lem}\label{lem_jh} 
The approximated Jacobian determinants satisfy
\begin{equation*}  
e^{-C\|\Delta_x W\|_{L^\infty}  \Delta t} \leq j_k^n \leq e^{C\|\Delta_x W\|_{L^\infty}\Delta t}\,
\end{equation*}
for a constant uniform in $k$ and $n \le N$. In particular, $J_k^n$ is always invertible and
\begin{equation} \label{enchkn}
e^{-C\|\Delta_x W\|_{L^\infty}T} \leq \frac{h_k^n}{h^d} \leq e^{C\|\Delta_x W\|_{L^\infty}T}.
\end{equation}
As for the deformation matrices $D^n_k = (J^{n-1}_k \cdots J^0_k)^{-1}$, they satisfy
\begin{equation} \label{est-Dnk}
\max(|D_k^n|, |(D_k^n)^{-1}|) \leq C
\end{equation}
for another constant uniform in $k$ and $n \le N$.
\end{lem}

We next show that the support of the particle approximation is of order $h$.
\begin{lem}\label{lem_supp} If $\nabla W \in \mw^{1,\infty}(\R^d)$, then we have
\beq\label{est_supp_Snk}
|x - x_k^n| \leq C h \qquad \text{ for }  x \in S^n_{h,k}
\eeq
and
\beq\label{est_supp_tSnk}
|x - x_k^n| \leq C(h+\Delta t) \qquad \text{ for }  x \in \tilde S^n_{h,k}
\eeq
with constants $C$ independent of $\Delta t$ and $h$.
\end{lem}
\begin{proof} 
From $\supp(\varphi) \subset B(0,c)$, we easily infer that
$ \big|D_k^n(x-x_k^n)\big| \leq c h $ holds on $\supp(\varphi^n_{h,k})$, see \eqref{phikn}, 
thus \eqref{est_supp_Snk} holds for $n \le N$, using \eqref{est-Dnk}.
To complete the proof we then observe that \eqref{expFex} gives
$$
\abs{x^{n+1}_k - F^{t_{n},t_{n+1}}(x^n_k)} =
\abs{\int_{t_n}^{t_{n+1}} (\nabla W * \rho_h^n)(x_k^n) - (\nabla W * \rho(\tau))(F^{t_n,\tau}(x^n_k))d\tau}
\le C \Delta t,
$$
so that if $x$ is such that $F^{t_n,t_{n+1}}(x) \in \supp(\varphi^{n+1}_{h,k})$, we have
$$
\bal
\abs{x-x^n_k} 
    &= \abs{F^{t_{n+1},t_n}(F^{t_n,t_{n+1}}(x)) - F^{t_{n+1},t_n}(F^{t_n,t_{n+1}}(x^n_k))} 
    \\
    & \le \abs{F^{t_n,t_{n+1}}}_{Lip} \big( 
            \abs{F^{t_n,t_{n+1}}(x) - x^{n+1}_k} 
            + \abs{x^{n+1}_k - F^{t_{n},t_{n+1}}(x^n_k)}  
            \big)
    \\
    & \le C (h + \Delta t),
\eal
$$
by using the Lipschitz estimate \eqref{majFexlip} and the bound \eqref{est_supp_Snk} on $S^{n+1}_{h,k}$.
\end{proof}
To control the approximation errors for the velocity and the Jacobian matrices, we next introduce the 
generic error
\begin{equation} \label{def-xi}
 \tilde \xi_n(K) := \sup_{\tau \in [t_n,t_{n+1}]}\sup_{k\in \Z^d}\, \sup_{x\in \tilde S^n_{h,k}}  
                \left| (K * \rho(\tau))(F^{t_n,\tau}(x)) - (K * \rho_h^n)\ (x_k^n) \right|,  
\end{equation}
for some given $K\in \mw^{1,\infty}(\R^d)$ and $0\leq n \leq N = T / \Delta t$.

\begin{prop} \label{propK1} The discrete velocity $u^n_k := -(\nabla W * \rho^n_h)(x_k^n)$ satisfies
\begin{equation} \label{estvelobis}
|u(\tau,F^{t_n,\tau}(x_k^n)) - u^n_k| \le C(h^2 \norm{\rho^0}_{\mw^{1,1}} + \Delta t + \overline{e_F}^n)
\end{equation}
for $\tau \in [t_n,t_{n+1}]$, $0 \leq n \leq N-1$ and with a constant $C$ independent of $\Delta t$ and $h$.
\end{prop}

\begin{proof}
Using that $u(\tau,y) = -(\nabla W * \rho(\tau))(y) = -(\nabla W * (F^{0,\tau} \# \rho^0))(y)$, we write
$$ 
\begin{aligned}
u(\tau,F^{t_n,\tau}(x^n_k)) 
    &=-\int_{\R^d} \nabla W (F^{t_n,\tau}(x^n_k)-y)   \rho(\tau,y)dy
    \\
    &=-\int_{\R^d} \nabla W(F^{t_n,\tau}(x^n_k)- F^{0,\tau}(z) ) \rho^0(z) dz
    \\
    & = (a) +(b)+(c) - \int_{\R^d} \nabla W(x_k^n-y)   \rho_h^n(y)dy
\end{aligned}
$$
with
$$ \begin{aligned}
&(a) := -\int_{\R^d} \left[ \nabla W(F^{t_n,\tau}(x^n_k)- F^{0,\tau}(z) ) - \nabla W(x_k^n - F^{0,t_n}(z))\right]\rho^0(z)  dz \\
&(b) := -\int_{\R^d} \nabla W(x_k^n - F^{0,t_n}(z)  ) \left[ \rho^0(z) - \rho_h^0(z)\right]  dz \\
& (c) := -\sum_{l\in \Z^d} \omega_l \int_{S^n_{h,l}} 
  \left[\nabla W(x_k^n - F^{0,t_n}((\overline{F}_{h,l}^{n-1})^{-1}(y) )) - \nabla W(x_k^n-y)\right] 
        \varphi_{h,l}^n(y)dy,
\end{aligned}
$$
so that
$ \left| u(\tau,F^{t_n,\tau}(x_k^n)) - u^n_k \right|      \leq |(a)|+|(b)|+|(c)|$. 
For the first term we write 
$$
|(a)|  
     \leq  \|\nabla W\|_{\mw^{1,\infty}} \int_{\R^d} \abs{A(z)} \rho^0(z) dz 
     \leq C \norm{A}_{L^\infty}
$$
with $A(z) := (F^{t_n,\tau}(x_k^n)-F^{0,\tau}(z)) - (x_k^n - F^{0,t_n}(z))$.
Using the expression \eqref{expFex} for the exact flow, estimate \eqref{est_supp_tSnk} and the equality $\norm{\rho(s)}_{L^1} = 1$ 
gives then
$$
\abs{A(z)} \le
    \int_{t_n}^\tau \Abs{ (\nabla W*\rho(s))(F^{t_n,s}(x^n_k)) + 
                                      (\nabla W*\rho(s))(F^{0,s}(z)) } ds
    \le 2 \Delta t \norm{\nabla W}_{L^\infty}
$$
so that $\abs{(a)} \le C \Delta t $.
For $(b)$, using the Lipschitz regularity of the flow \eqref{majFexlip} and the error bound \eqref{rho0-error-neg} on the initial data we find
$$
\begin{aligned}
    |(b)| &\leq e^{C T } \norm{\nabla W}_{\mw^{1,\infty}} \norm{\rho_h^0 - \rho^0}_{\mw^{-1,1}} 
    \leq C h^2 \norm{\rho^0}_{\mw^{1,1}}.
\end{aligned} 
$$
Finally, we observe that for $y\in S^n_{h,l}$ we have $(\overline{F}_{h,l}^{n})^{-1}(y)\in {S^0_{h,l}}$
from \eqref{def_fbar}, and
$$
\left|  F^{0,t_n}\left((\overline{F}_{h,l}^n)^{-1}(y)\right) -y \right| 
\leq \left|  F^{0,t_n}\left((\overline{F}_{h,l}^{n})^{-1}(y)\right) -  
     \overline{F}_{h,l}^{n} \left((\overline{F}_{h,l}^{n})^{-1}(y) \right)
 \right| 
\leq \overline{e_F}^n,
$$
and arguing as in \eqref{L1rhon} this gives
$$\begin{aligned} 
    |(c)|
    & \leq  
    \|\nabla W \|_{\mw^{1,\infty}} \sum_{l\in \Z^d} \abs{\omega_l} \int_{S^n_{h,l}} 
                \left|  F^{0,t_n}\left((\overline{F}_{h,l}^{n})^{-1}(y)\right) -y \right| \varphi_{h,l}^n(y)dy   
    \\
     &\leq C \overline{e_F}^n \sum_{l\in \Z^d} \abs{\omega_l} 
     \leq C \overline{e_F}^n.
%
\end{aligned}
$$
By gathering the above estimates, we complete the proof.
\end{proof}

\begin{prop} \label{propK2} If the initial density satisfies $\rho^0 \in \mw^{1,1}_+(\R^d)$, then the estimate
$$
 \tilde \xi_n(D^2 W) \leq  C\left(\theta_n + \Delta t+h \right)
$$ 
holds with a constant $C$ depending only on $d$, $T$, $L$, and $\|\rho^0\|_{\mw^{1,1}}$. Moreover, at $x=x_k^n$, we have
$$
\sup_{k\in \Z^d}\, \sup_{\tau \in [t_n,t_{n+1}]} \left| (D^2 W * \rho(\tau))(F^{t_n,\tau}(x_k^n))- (D^2 W * \rho_h^n)\ (x_k^n) \right|  
\leq C\left(\theta_n + \Delta t \right).
$$
\end{prop}
\begin{proof}
Given $x\in \tilde S^n_{h,k}$ and $\tau \in [t_n,t_{n+1}]$, we write
\begin{align*}
\begin{aligned}
&\left| (D^2 W * \rho(\tau))(F^{t_n,\tau}(x))- (D^2 W * \rho_h^n)\ (x_k^n) \right|\cr
&\qquad =\int_{\R^d} D^2 W(y)\left[\rho(\tau,F^{t_n,\tau}(x)-y) -    \rho_h^n(x_k^n-y) \right]dy = (a)+(b),
\end{aligned}
\end{align*}

with 
$$
(a) := \int_{\R^d} D^2 W(y)\left[\rho(\tau,F^{t_n,\tau}(x)-y) -    \rho(t_n,x_k^n-y) \right]dy,
$$
$$
(b) :=\int_{\R^d} D^2 W(y)\left[\rho(t_n,x_k^n -y) -    \rho_h^n(x_k^n-y) \right]dy.
$$
The second term is estimated by
$$ 
|(b)| \leq \|D^2 W\|_{L^{\infty}} \|\rho(t_n,\cdot)-\rho_h^n\|_{L^1} \leq L \theta_n.
$$
And using $\rho(\tau) = F^{t_n,\tau} \# \rho(t_n)$ we rewrite the first term as $(a) = (c) + (d)$ with 
$$
\bal
(c) &:= \int_{\R^d}  D^2 W(y)\rho(t_n,F^{\tau,t_n}(F^{t_n,\tau}(x)-y))\left[j^{\tau,t_n}(F^{t_n,\tau}(x)-y) -1 \right]dy\cr
(d) &:= \int_{\R^d}  D^2 W(y)[\rho(t_n,F^{\tau,t_n}(F^{t_n,\tau}(x)-y))- \rho(t_n,x_k^n-y)]\,dy.
\eal
$$
For $(c)$ we use the one-to-one mapping
$\Phi : y \mapsto F^{\tau,t_n}(F^{t_n,\tau}(x)-y)$ with Jacobian determinant 
$\abs{\det \Phi(y)} = j^{\tau,t_n}(F^{t_n,\tau}(x)-y)$. The change of variable formula yields
$$
\int_{\R^d} \rho(t_n,F^{\tau,t_n}(F^{t_n,\tau}(x)-y)) dy 
    \le C \int_{\R^d} \rho(t_n,\Phi(y)) \abs{\det \Phi(y)} dy
    = C \norm{\rho(t_n)}_{L^1} 
    \le C
$$
where we have used \eqref{est_reja} in the first inequality. Using \eqref{est_reja-1} this allows to bound 
$$
\abs{(c)} \leq C\Delta t \|D^2 W\|_{L^\infty} \int_{\R^d} \rho(t_n,F^{\tau,t_n}(F^{t_n,\tau}(x)-y)) dy \leq C \Delta t.
$$ 
Turning next to the $(d)$ term, we introduce
$$
\Xi_\alpha : y \mapsto \alpha (F^{\tau,t_n}(F^{t_n,\tau}(x)-y)) + (1-\alpha)(x_k^n-y) \quad \text{for} \quad \alpha \in [0,1],
$$
so that 
$$
\bal
\abs{(d)} 
    &\le \|D^2 W\|_{L^\infty}\int_{\R^d} \abs{\rho(t_n,\Xi_1(y)) - \rho(t_n,\Xi_0(y))} dy
    \\
    &\le C \int_{\R^d} \int_0^1 \abs{\nabla \rho(\Xi_\alpha(y))} \abs{F^{\tau,t_n} (F^{t_n,\tau}(x)-y) -(x_k^n-y)} \, d\alpha dy
    \\
    &\le C (h + \Delta t) \int_{\R^d} \int_0^1 \abs{\nabla \rho(\Xi_\alpha(y))} \, d\alpha dy
\eal
$$
where in the last inequality we have used (see \eqref{expFex} and Lemma~\ref{lem_supp})
$$
\begin{aligned}
&|F^{\tau,t_n} (F^{t_n,\tau}(x)-y) - (x_k^n-y)| \cr
&\quad = \left|F^{t_n,\tau}(x)-x^n_k -\int_{\tau}^{t_n} (\nabla W * \rho(s))(F^{\tau,s}(F^{t_n,\tau}(x)-y))ds \right| \cr
&\quad \leq (\abs{x-x^n_k} + 2\Delta t \|\nabla W\|_{L^\infty}) \leq C (h+ \Delta t).
\end{aligned}
$$
To end the proof we will show that up to a sign and a translation, $\Xi_\alpha$ is uniformly 
close to the identity mapping. Let $G(y) := (F^{\tau,t_n} - I)(F^{t_n,\tau}(x)-y)$ so that 
$\Xi_\alpha(y) = -y + \alpha G(y) + (1-\alpha) x^n_k + \alpha F^{t_n,\tau}(x)$. 
From \eqref{maj1moinsJex} we infer
$$
\abs{DG(y)} = \abs{I_d - J^{\tau,t_n}(F^{t_n,\tau}(x)-y)} \le C \Delta t
$$
hence there exists a constant $\gamma$ independent of $h$, $\Delta t$ and $n$, such that
$$
\abs{G(y)-G(y')} \le \gamma \Delta t \abs{y-y'}.
$$
This shows that $\Xi_\alpha$ is injective for $\Delta t$ small enough, indeed if $\Xi_\alpha(y) = \Xi_\alpha(y')$ for $y \neq y'$ then $y-y' = \alpha(G(y)-G(y'))$ leads to a contradiction for $\gamma \Delta t < 1$. Moreover, using $D \Xi_\alpha(y) = -I_d + \alpha DG(y)$ and the Jacobi formula for $\partial_\alpha \det(D \Xi_\alpha)$ we find
$$
\abs{\det(D \Xi_\alpha)(y) + 1} \le C \Delta t\,,
$$
which shows that for $\Delta t$ small enough, $\abs{\det(D \Xi_\alpha)}$ is bounded from below by a positive constant $\tilde \gamma$. Using again the change of variable theorem this gives
$$
\tilde \gamma \int_{\R^d} \abs{\nabla \rho(\Xi_\alpha(y))} dy 
\le  \int_{\R^d} \abs{\nabla \rho(\Xi_\alpha(y))} \abs{\det(D \Xi_\alpha)(y)} dy
\le \int_{\R^d} \abs{\nabla \rho(z)} dz \le \norm{\rho}_{\mw^{1,1}}\,.
$$
The desired bound $\abs{d} \le C(h+\Delta t)$ follows by gathering the above steps.
\end{proof}

We can now compute an estimate for the error of the Jacobian determinants.
\begin{coro}\label{Cor_ej}
Assume that $\rho^0 \in \mw^{1,1}_+(\R^d)$, then the following estimate holds
\begin{equation}  \label{majeJn1bis}
e_j^n  \leq C  \Delta t\left(  \theta_n+\Delta t+ h \right) \quad \mbox{for all} \quad 0 \leq n \leq N,
\end{equation}
where $C$ is a positive constant depending only on  $T$, $L$, and $\|\rho\|_{L^\infty(0,T:\mw^{1,1})}$.
\end{coro}

\begin{proof}
According to \eqref{exp_ja} and \eqref{est_jk}, we have
$$ 
\bal
    \frac{1}{j_k^n} - \frac{1}{j^{t_n,t_{n+1}}(x)}
    & = \exp(\beta^n_k) - \exp(\beta^n(x))
    \\
    & = (\beta^n_k-\beta^n(x)) \int_0^1 \exp\big(r\beta^n_k + (1-r) \beta^n(x)\big) dr
\eal
$$    
with 
$\beta^n_k := \Delta t (\Delta_x W * \rho_h^n)(x_k^n)$ 
and 
$\beta^n(x) := \int_{t_n}^{t_{n+1}} (\Delta_x W * \rho(\tau))(F^{t_n,\tau}(x))  d\tau $.
Since $e_j^n$ involves the above difference for $x\in S^n_{h,k} \subset \tilde S^n_{h,k}$, 
see \eqref{def_ejn}, we infer from \eqref{def-xi} that 
$\abs{\beta^n_k - \beta^n(x)} \le C \Delta t \, \tilde \xi_n(D^2 W)$. 
Using the $L^1$ bound \eqref{L1rhon} on $\rho^n_h$ this yields
$$
e_j^n \leq  C \Delta t \, \tilde \xi_n(D^2 W) \exp\left( C \Delta t \|\Delta_x W\|_{L^\infty} \right)\,,
$$
so that Proposition \ref{propK2} gives the desired result.
\end{proof}

From Proposition \ref{propK2} we also derive an estimate for the error between Jacobian matrices.
\begin{coro}\label{Cor_ej2} 
If $\rho^0 \in \mw^{1,1}_+(\R^d)$, then for $0 \leq n \leq N$ the following estimate holds 
\[ 
|J_k^n - J^{t_n,t_{n+1}}(x)| \leq C\Delta t\lt(\theta_n + h + \Delta t\rt) \qquad \mbox{for } x \in S^n_{h,k},
\] 
with a constant $C$ independent of $\Delta t$ and $h$. 
At $x = x_k^n$, we have 
\beq \label{errbound-J}
|J_k^n - J^{t_n,t_{n+1}}(x_k^n)| \leq C\Delta t\lt(\theta_n  + \Delta t\rt). 
\eeq
\end{coro}

\begin{proof} 
Using the matrix $\tilde J^{t_n,t_{n+1}}(x)$ defined by 
\eqref{expJtilde}, Proposition \ref{propJexJtilde} gives
$$
\abs{J_k^n - J^{t_n,t_{n+1}}(x)} \le \abs{J_k^n - \tilde J^{t_n,t_{n+1}}(x)} + C (\Delta t)^2
$$
and to bound the remaining error we proceed as in the proof of Corollary~\ref{Cor_ej}:
denoting
$B^n_k := -\Delta t \,(D^2W * \rho_h^n)(x_k^n)$ 
and 
$B^n(x) := -\int_{t_n}^{t_{n+1}} (D^2 W * \rho(\tau))(F^{t_n,\tau}(x)) d\tau $,
we use the exponential matrix expressions \eqref{schemEuler} and \eqref{expJtilde} to compute
$$ 
\bal
    J_k^n - J^{t_n,t_{n+1}}(x)
    & = \exp(B^n_k) - \exp(B^n(x))
    \\
    & = (B^n_k-B^n(x)) \int_0^1 \exp\big(rB^n_k + (1-r) B^n(x)\big) dr.
\eal
$$    
For $x\in S^n_{h,k}$ we have $\abs{B^n_k - B^n(x)} \le C \Delta t \, \xi_n(D^2 W)$
and using \eqref{L1rhon} this yields
$$
|J_k^n - J^{t_n,t_{n+1}}(x)| \leq  C \Delta t \, \xi_n(D^2 W) \exp\left( C \Delta t \|D^2 W\|_{L^\infty} \right)
$$
so that the desired result follows again from Proposition \ref{propK2}.
\end{proof}

\begin{rem}\label{rmk_jaco} 
If $\rho^0$ is only assumed to be an $L^1(\R^d)$ function (or a Radon measure), 
then $\xi_n(D^2 W)$ can be bounded by a constant using the $L^1$ bound on $\rho^n_h$, see \eqref{L1rhon}, and the $\mw^{1,\infty}(\R^d)$ smoothness of $\nabla W$. Arguing as in the proof above we then find an error estimate for the Jacobian matrices on the order of $\Delta t$.
\end{rem}

We next turn to the approximation errors involving the forward characteristic flows and we establish a series of estimates.
\begin{lem}  \label{estbareFn}
For $0 \leq n\leq N-1$, the following estimate holds
\begin{equation}  \label{lienefnefnbar}
\overline{e_F}^{n+1} \leq e^{C\Delta t}\overline{e_F}^n  + \tilde e_F^n
\end{equation}
with a constant $C$ independent of $\Delta t$ and $h$.
\end{lem}

\begin{proof} 
 Given $x\in S^0_{k,h}$ we write $y=F^{0,t_n}(x)$ and $\tilde{y}_k= \overline{F}_{h,k}^n(x) \in S^n_{h,k}$. We have
$$
\begin{aligned}
 \left|\overline{F}_{h,k}^{n+1}(x) - F^{0,t_{n+1}}(x) \right|
    & = \left| F_{h,k}^n(\tilde{y}_k)  - F^{t_n,t_{n+1}}(y)\right| 
    \\
    & \leq \left| F^{t_n,t_{n+1}}( \tilde{y}_k )-  F^{t_n,t_{n+1}}(y)\right|
       + \left| F^{t_n,t_{n+1}}( \tilde{y}_k ) -  F_{h,k}^n(\tilde{y}_k)\right| 
   \\
    & \leq \left|  F^{t_n,t_{n+1}}  \right|_{Lip}| \tilde{y}_k -y|
        + \| F^{t_n,t_{n+1}} - F_{h,k}^n\|_{L^{\infty}(S^n_{h,k})}  
    \\
    & \leq e^{C\Delta t} \overline{e_F}^n+ \tilde e_F^n
\end{aligned} 
$$
by using $S^n_{h,k} \subset \tilde S^n_{h,k}$ and the Lipschitz bound \eqref{majFexlip} on the exact flow.
\end{proof}

\begin{prop}  \label{esteFn}
If $\rho^0 \in \mw^{1,1}_+(\R^d)$, then the following estimate holds
\beq \label{majprop42}
\tilde e_F^{n } \leq  C \Delta t( \Delta t +  h^2+
   (h+\Delta t) \theta_n +  \overline{e_F}^n) \quad \mbox{for} \quad 0 \leq n \leq N,
\eeq
with a constant $C$ independent of $\Delta t$ and $h$.
\end{prop}

\begin{proof}
Given $x\in \tilde S^n_{h,k}$, we rewrite the linearized flow \eqref{Fnk} as follows,
$$\begin{aligned}
F_{h,k}^n(x) 
& =   F_{h,k}^n(x_k^n) + J_{k}^n(x -x_k^n) = (a) + (b) +  (c) + F^{t_n,t_{n+1}}(x) 
\end{aligned}
$$
with
$$
\begin{aligned} 
 &(a) := F_{h,k}^n(x_k^n) - F^{t_n,t_{n+1}}(x_k^n) \\
 &(b) := \left( J_{k}^n- J^{t_n,t_{n+1}}(x_k^n)  \right)(x -x_k^n)  \\
 &(c) := F^{t_n,t_{n+1}}(x_k^n)  + J^{t_n,t_{n+1}}(x_k^n)(x-x_k^n)-F^{t_n,t_{n+1}}(x).
\end{aligned} 
$$
Using \eqref{schemEuler} and the expression \eqref{expFex} for the exact flow, we then compute
\[
|(a)| = \int_{t_n}^{t_{n+1}} \left| (\nabla W * \rho^n_h)(x_k^n) + u(\tau, F^{t_n,\tau}(x_k^n)) \right| d\tau  
        \leq C\Delta t \left( h^2 + \Delta t + \overline{e_F}^n \right)
\]
where the inequality follows from \eqref{estvelobis} (note that here $C$ depends on $\norm{\rho^0}_{\mw^{1,1}}$).
For $(b)$, we easily get using estimate \eqref{errbound-J} in Corollary \ref{Cor_ej2} and Lemma~\ref{lem_supp} that
$$
\begin{aligned} 
|(b)| \leq \abs{J_k^n - J^{t_n,t_{n+1}}(x_k^n)}\abs{x -x_k^n} \leq C \Delta t (\theta_n+\Delta t )(h + \Delta t).
\end{aligned}
$$
Turning to $(c)$ we next differentiate \eqref{expJex} and obtain for $1 \leq i,j,m \leq d$,
$$\begin{aligned}
\pa_m\lt(J^{t_n,t_{n+1}}\rt)_{ij} &= -\sum_{l = 1}^d\int_{t_n}^{t_{n+1}} (\pa_{il} W * \nabla \rho(\tau))(F^{t_n,\tau}(x))\pa_m F^{t_n,\tau}(x)\lt(J^{t_n,\tau}(x)\rt)_{lj} d\tau\cr
&\quad -\sum_{l=1}^d \int_{t_n}^{t_{n+1}} (\pa_{il}W * \rho(\tau))(F^{t_n,\tau}(x))\pa_m\lt(J^{t_n,\tau}(x)\rt)_{lj} d\tau.
\end{aligned}$$
This yields
$$\begin{aligned}
\abs{\pa_m J^{t_n,t_{n+1}}(x)} \leq C\Delta t + C\int_{t_n}^{t_{n+1}} \abs{\pa_m J^{t_n,\tau}(x)} d\tau,
\end{aligned}$$
where we used that $\rho \in L^\infty(0,T;\mw^{1,1}(\R^d))$, $\nabla W \in \mw^{1,\infty}(\R^d)$ and $|\pa_m F^{t_n,\tau}| \leq C$ for some $C$, see \eqref{majJex}. Invoking the Gronwall Lemma, we then obtain
\[
\abs{\pa_m J^{t_n,t_{n+1}}(x)} \leq C\Delta t e^{C\Delta t}, \qquad m = 1, \cdots, d, 
\]
where $C$ only depends on $d$, $T$, $L$ and $\norm{\rho^0}_{\mw^{1,1}}$. With a Taylor expansion this gives
$$
|(c)|\leq \frac{1}{2} \left|D^2 F^{t_n,t_{n+1}}
     (\eta_k^n)\right||x-x_k^n|^2 \leq C \Delta t (h+\Delta t)^2
$$ 
for some $\eta_k^n$ between $x$ and $x_{k}^n$ and a constant $C$ that only depends on $d$, $T$, $L$ 
and $\norm{\rho^0}_{\mw^{1,1}}$. Combining the above estimates yields the desired result.
\end{proof}
 
We finally provide estimates for $\overline{e_F}^n$ and $\tilde e_F^n$.

\begin{coro} \label{majebarFn} If $\rho^0 \in \mw^{1,1}_+(\R^d)$, then the following estimates hold for $0 \leq n \leq N$,
$$
\overline{e_F}^{n} \leq C(h^2 + \Delta t + h\tilde\theta_{n-1}) 
\quad \text{and} \quad 
\tilde e_F^n \leq C\Delta t (h^2 + \Delta t + h\tilde\theta_{n-1}),
$$
with $\tilde\theta_n := \max_{m\le n} \theta_m$, see \eqref{defthetan}, 
and a constant $C$ independent of $\Delta t$ and $h$.
\end{coro}
\begin{proof}
Using \eqref{lienefnefnbar}, \eqref{majprop42} and the fact that $e^{C\Delta t} + C\Delta t \leq e^{2C\Delta t}$, we find
$$ 
    \overline{e_F}^{n+1} \leq e^{2C\Delta t}\overline{e_F}^n + C\Delta t(h^2 + \Delta t + h\tilde \theta_n),
$$ 
hence
    \[
    \overline{e_F}^{n+1} \leq e^{2CN\Delta t} (\overline{e_F}^0 + N \Delta t (h^2 + \Delta t + h\tilde \theta_n)) 
                         \leq C(h^2 + \Delta t + h\tilde\theta_n), \quad n\le N-1,
    \]
follows by a summation using $\overline{e_F}^0 = 0$. The bound on $\tilde e_F^n$ is obtained with \eqref{majprop42}.
\end{proof}

\subsection{Proof of $L^1$ and $L^{\infty}$ convergence results}

\begin{theo} \label{esthetan} 
Assume $\Delta t \le C h$. If $\rho^0 \in \mw^{1,1}_+(\R^d)$ and $\nabla W \in \mw^{1,\infty}(\R^d)$, then 
$$
\max_{0\leq n\leq N}\|\rho(t_n) - \rho^n_h\|_{L^1}
 \leq C \left(\|\rho^0 - \rho^0_h\|_{L^1} +\frac{\Delta t}{h} +  h \right)  
$$
holds with a constant $C$ depending only on $d$, $T$, $L$, and $\norm{\rho^0}_{\mw^{1,1}}$.
\end{theo}

\begin{proof}
Let $y\in \R^d$. Using the relation $\rho(t_n) = F^{t_n,t_{n-1}} \# \rho(t_{n-1})$ and the 
form \eqref{defrhohn} of the approximate solution together with the fact that $h^n_k = h^{n-1}_k j^{n-1}_k$,  
we decompose the error $\rho(t_n,y) - \rho_h^n(y)$ into three parts as 
\begin{equation}\label{est_main}
\begin{aligned}
 \rho(t_n,y) - \rho_h^n(y) &=
  \underbrace{ \left[\rho\left(t_{n-1}, F^{t_{n},t_{n-1}}(y)\right) 
     - \rho_h^{n-1}\left(F^{t_n,t_{n-1}}(y)\right) \right]j^{t_n,t_{n-1}}(y)     }_{A_n(y)} \\
&+  \underbrace{ \sum_{k\in \Z^d} \frac{\omega_k }{h_k^{n-1}}  \varphi \left( \frac{D_k^{n-1}}{h}\lt( F^{t_n,t_{n-1}}(y)- x_k^{n-1}\rt) \right)  \left[   j^{t_n,t_{n-1}}(y) - \frac{1}{j_{k}^{n-1}} \right]}_{B_n(y)} \\
&+ \underbrace{\sum_{k\in \Z^d} \frac{\omega_k }{h_k^n} \left[ \varphi \left( \frac{D_k^{n-1}}{h} (F^{t_n,t_{n-1}}(y)- x_k^{n-1} ) \right) -  \varphi \left( \frac{D_k^{n}}{h}  (y- x_k^{n})  \right) \right]}_{C_n(y)}.
\end{aligned}
\end{equation}

\noindent $\diamond$ Estimate of $\|A_n\|_{L^1}$: Using the one-to-one change of variable $x= F^{t_{n},t_{n-1}}(y)$, we easily find that
$$ \int_{\R^d} |A_n(y)|dy 
= \int_{\R^d} | \rho(t_{n-1},x)-\rho_h^{n-1}(x)|dx=\theta_{n-1}.
$$ 

\noindent $\diamond$ Estimate of $\|B_n\|_{L^1}$: By means of the same change of variable and the 
relation $j^{t_n,t_{n-1}}(y) = (j^{ t_{n-1},t_n}(x))^{-1}$, we obtain
$$ \begin{aligned} \int_{\R^d} |B_n(y)|dy 
&\leq  \int_{\R^d}   \sum_{k\in \Z^d}  \abs{\omega_k}
  \varphi_{h,k}^{n-1}(x)  \left|  \frac{1}{ j^{ t_{n-1},t_n}(x) } 
   - \frac{1}{j_{k}^{n-1}}  \right| j^{t_{n-1},t_n}(x)dx\\
& \leq  e_{j}^{n-1}   \norm{ j^{ t_{n-1},t_n}}_{L^\infty} \int_{\R^d}\sum_{k\in \Z^d} \abs{ \omega_k}  \varphi_{h,k}^{n-1}(x) dx 
    \leq  C  e_{j}^{n-1},
\end{aligned}
$$
due to \eqref{majJex}, \eqref{L1rhon} and \eqref{def_ejn}, indeed $x$ can be taken in $S^{n-1}_{h,k}$ in the $k$-th term.

\noindent $\diamond$ Estimate of $\|C_n\|_{L^1}$: Writing again $x = F^{t_{n},t_{n-1}}(y)$, we observe that in the $k$-th term, we must consider the cases where $y \in S^n_{h,k}$ and those where $x \in S^{n-1}_{h,k}$. Thus, $x$ must be taken in the extended particle support $\tilde S^{n-1}_{h,k}$, see \eqref{tilde-eFn}. Using the incremental relation \eqref{defhkn} we then estimate
$$
    \abs{D_k^{n-1} (x - x_k^{n-1} ) - D_k^{n} (y- x_k^{n})}
    = \abs{D^n_k (x_k^{n} + J^{n-1}_k(x-x^{n-1}_k) - F^{t_{n-1},t_n}(x))}
    \le \abs{D^n_k} \tilde e^{n-1}_F 
$$
see \eqref{Fnk}, \eqref{eFn}. To obtain a global bound we next observe that the measure of $\tilde S^{n-1}_{h,k}$ is of order $(h+\Delta t)^d \le Ch^d$ according to Lemma~\ref{lem_supp} and the assumption $\Delta t \leq C h$, as well as that of  $F^{t_{n-1},t_{n}}(\tilde S^{n-1}_{h,k})$ according to \eqref{est_reja}. Using the above observations and the fact 
that the reference shape $\varphi$ is assumed to be Lipschitz, we find
\beq
\label{est-Cn-L1}
\int_{\R^d} |C_n(y)|dy 
    \le C h^d \sum_{k \in \Z^d} 
            \frac{\abs{\omega_k}}{h_k^n} \frac{\abs {D^n_k}}{h} \tilde e^{n-1}_F
    \le C \frac{\tilde e^{n-1}_F}{h},
\eeq
where the last inequality follows from the uniform bounds on the matrices $D^n_k$ and their determinants (Lemma~\ref{lem_jh}), and from the estimates inside \eqref{L1rhon}.

\noindent $\diamond$ Conclusion:  We now combine all the estimates above and \eqref{majeJn1bis} in Corollary \ref{Cor_ej} to obtain
\[
\theta_n \leq \theta_{n-1} + Ce_j^{n-1} + C\frac{\tilde e_F^{n-1}}{h}
\leq (1 + C\Delta t)\theta_{n-1} + C\Delta t(\Delta t + h) + C\frac{\tilde e_F^{n-1}}{h}.
\]
Using Corollary \ref{majebarFn} to estimate the flow error yields
\[
\tilde \theta_n \leq (1 + C\Delta t)\tilde\theta_{n-1} + C\Delta t \left(\Delta t + h +\frac{\Delta t}{h}\right).
\]
Since $h\leq 1$, we conclude that
\[
\tilde \theta_n \leq e^{CN\Delta t}\theta_0 + e^{CN\Delta t}\left(h+\frac{\Delta t}{h}\right) \leq C\left(h+\theta_0 + \frac{\Delta t}{h}\right).
\]
\end{proof}

We next derive $L^\infty$-estimates. Here the required regularity propagates in time. As proved in the Appendix, Proposition~\ref{rhoW11}, the unique solution to \eqref{main-eq} belongs to $\rho \in L^\infty(0,T;(\mw^{1,1}_+\cap L^\infty)(\R^d))$ provided that $\rho^0 \in (\mw^{1,1}_+ \cap L^\infty)(\R^d)$.

\begin{theo}\label{prop_inf}
If $\Delta t \le C h$, $\rho^0 \in \mw^{1,1}_+(\R^d)\cap L^\infty(\R^d)$ and  $\nabla W \in \mw^{1,\infty}(\R^d)$, then 
$$
\max_{0\leq n\leq N}\|\rho(t_n) - \rho^n_h\|_{L^{\infty}} 
\leq C\left(h + \|\rho^0 - \rho^0_h\|_{L^{\infty}}+ \|\rho^0 - \rho^0_h\|_{L^{1}}  + \frac{\Delta t}{h}\right)
$$
holds with a constant independent of $h$ and $\Delta t$.

\end{theo}
\begin{proof}
Given $y\in \R^d$, we decompose $ \rho(t_n,y) - \rho_h^n(y)$ into three terms as in \eqref{est_main}. \newline

\noindent $\diamond$ Estimate of $\|A_n\|_{L^\infty}$: Using the bound \eqref{est_reja} on the exact Jacobian determinant,
we find
$$   \|A_n\|_{L^{\infty}}
 \leq     e^{C\Delta t} \varepsilon_{n-1}.
$$  

\noindent $\diamond$ Estimate of $\|B_n\|_{L^\infty}$: 
Writing again $x = F^{t_{n},t_{n-1}}(y)$, we observe that the $k$-th term vanishes if 
$x \not\in S^{n-1}_{h,k}$. In particular, the sum can be restricted to the indices $k$ in 
the set $\mathcal{K}_{n-1}(x)$. Gathering the bounds \eqref{enchkn} on $h^n_k$, \eqref{majweight} on $\omega_k$
and \eqref{majKn} on $\kappa_n := \sup_{x\in \R^d} \#(\mathcal{K}_{n-1}(x))$, we compute
$$
|B_n(y)| 
\le C \#(\mathcal{K}_{n-1}(x)) \norm{\rho^0}_{L^\infty} \norm{\varphi}_{L^\infty} e^{n-1}_j
\le C \left( 1+ \frac{\overline{e_F}^n}{h}\right)^d e^{n-1}_j.
$$

\noindent $\diamond$ Estimate of $\|C_n\|_{L^\infty}$:
Similarly as in the proof of Theorem~\ref{esthetan}, we observe that the $k$-th summand in $C_n(y)$
must be considered when $y \in S^n_{h,k}$ or when $x \in S^{n-1}_{h,k}$ (or both). Clearly the cardinality 
of the corresponding index set satisfies
$$
\#(\{k \in \Z^d : y \in S^n_{h,k} \text{ or } x \in S^{n-1}_{h,k}\}) \le 
\#(\mathcal{K}_n(y)) + \#(\mathcal{K}_{n-1}(x)) \le  \kappa_n + \kappa_{n-1}.
$$
Using the Lipschitz smoothness of the reference shape function $\varphi$ as in \eqref{est-Cn-L1}, and
again the bounds \eqref{enchkn} on $h^n_k$, \eqref{majweight} on $\omega_k$ and \eqref{majKn} on $\kappa_n$, we write
$$
|C_n(y)| \le C ( \kappa_n + \kappa_{n-1}) \frac{\tilde e^{n-1}_F}{h}
\le C \left( \lt(1+ \frac{\overline{e_F}^n}{h}\rt)^d + \lt(1+\frac{\overline{e_F}^{n-1}}{h}\rt)^d\right) \frac{\tilde e_F^{n-1}}{h}.
$$

\noindent $\diamond$ Conclusion: Combining the estimates above, we have
\beq\label{est_linf}
\varepsilon_n 
\leq   e^{C\Delta t} \varepsilon_{n-1}+ C\left(1+\frac{\overline{e_F}^{n-1} }{h} \right)^d e_{j}^{n-1} + C\lt(1+ \frac{\overline{e_F}^n + \overline{e_F}^{n-1}}{h} \rt)^d \frac{\tilde e_F^{n-1}}{h}.
\eeq
Now, with the assumptions made here Theorem~\ref{esthetan} applies, hence Corollaries~\ref{Cor_ej} 
and \ref{majebarFn} provide error estimates for the Jacobian and flow errors. Specifically, we have
$$
e_j^{n-1} \leq C\Delta t(\Delta t + h ), 
\quad 
    \overline{e_F}^n \leq C(h^2 +\Delta t+h\theta_0), 
\quad
    \tilde e_F^{n-1} \leq C\Delta t(h^2  +\Delta t+h\theta_0)\,.
$$
Plugging these estimates into \eqref{est_linf} yields then
\[ 
\varepsilon_n \leq   e^{C \Delta t} \varepsilon_{n-1} + C\Delta t\left(h+\theta_0 +\frac{\Delta t}{h}\right),
\]
due to $\Delta t \lesssim h \lesssim 1$ and $\theta_0 \leq 2$. 
We again conclude with the discrete Gronwall Lemma.
\end{proof}

\begin{rem}
Under the condition $\Delta t\leq C h$ in the convergence theorems, we have obtained the convergence estimates in $L^1$ and $L^\infty$ with the terms of the form $\Delta t/h$. We obviously need the assumption $\Delta t=o(h)$ to get the convergence results.
\end{rem}

\section{Convergence for measure solutions with smooth potentials}
\setcounter{equation}{0}

In this part, we consider measure valued solutions to the system \eqref{main-eq} using the bounded Lipschitz distance. More precisely, let $\rho_1,\rho_2 \in \mathcal{M}(\R^d)$ be two Radon measures. Then the bounded Lipschitz distance $d_{BL}(\rho_1,\rho_2)$ between $\rho_1$ and $\rho_2$ is given by
\[
d_{BL}(\rho_1,\rho_2) : = \sup \left\{\left|\int_{\R^d} \psi d\rho_1 -\int_{\R^d}\psi d\rho_2 \right| :
\psi \in \mw^{1,\infty}(\R^d) ~~ \mbox{and} ~~  \|\psi\|_{\mw^{1,\infty} }\leq 1 \right\}.
\]
Since the interaction potential $W$ satisfies $\nabla W \in \mw^{1,\infty}(\R^d)$, a well-posedness theory for measure valued solutions to \eqref{main-eq} can be developed by using the classical results of Dobrushin \cite{Dob}, see \cite{Golse,CCH} for related results.

To estimate the error between the exact flow and its local linearizations we now revisit some results from the previous section, namely Proposition~\ref{esteFn}, given the low regularity of the solutions. 
As in the previous section, we denote
$\rho^n = \rho(t_n)$.

\begin{prop}\label{prop_measure}Let $\rho^0$ be an initial Radon measure on $\R^d$, and $\rho^n_h$ be the approximation constructed in \eqref{defrhohn}. If $W$ satisfies $\nabla W \in \mw^{1,\infty}(\R^d)$, then the flow error defined on the particles support \eqref{eFn}
\[
e_F^n \leq C\Delta t\lt(d_{BL}(\rho^n, \rho^n_h) + h + \Delta t \rt)
\]
holds for $0 \leq n \leq N$ with a constant $C$ independent of $h$ and $\Delta t$.
\end{prop}
\begin{proof}
Let $x\in S^n_{h,k}$. We decompose the linearized flow as in Proposition~\ref{esteFn}, 
$$
F_{h,k}^n(x) =   F_{h,k}^n(x_k^n) + J_{k}^n(x -x_k^n) = (a) + (b) +  (c) + F^{t_n,t_{n+1}}(x) 
$$
with
$$\begin{aligned} 
 &(a) := F_{h,k}^n(x_k^n) - F^{t_n,t_{n+1}}(x_k^n) \\
 &(b) := \left( J_{k}^n- J^{t_n,t_{n+1}}(x_k^n)  \right)(x -x_k^n)  \\
 &(c) := F^{t_n,t_{n+1}}(x_k^n)  + J^{t_n,t_{n+1}}(x_k^n)(x-x_k^n)-F^{t_n,t_{n+1}}(x).
\end{aligned} $$
We next rewrite 
$(a) = \int_{t_n}^{t_{n+1}} \lt( (\nabla W * \rho^n_h)(x_k^n) - (\nabla W * \rho(\tau))(F^{t_n,\tau}(x_k^n))\rt)d\tau $
using \eqref{schemEuler} and \eqref{expFex}, and estimate the integrand by
$$\begin{aligned}
&(\nabla W * \rho^n_h)(x_k^n) - (\nabla W * \rho(\tau))(F^{t_n,\tau}(x_k^n))\cr
&\qquad = \int_{\R^d} \nabla W(x_k^n - y)\rho^n_h(y) - \nabla W(F^{t_n,\tau}(x_k^n) - y)\rho(\tau,y)\,dy \cr
&\qquad = \int_{\R^d} \nabla W(x_k^n - y)\lt( \rho^n_h(y) - \rho^n(y)\rt)dy \cr
&\qquad \quad + \int_{\R^d} \nabla W(x_k^n - y)\rho^n(y) - \nabla W(F^{t_n,\tau}(x_k^n) - y)\rho(\tau,y)\,dy \cr
&\qquad =: (d) + (e).
\end{aligned}$$
From $\nabla W \in \mw^{1,\infty}(\R^d)$, we infer $|(d)| \leq C d_{BL}(\rho^n, \rho^n_h)$.
Using next a change of variable and the relation $\rho(\tau) = F^{t_{n},\tau} \# \rho^{n}$ we get
$$\begin{aligned}
(e) &= \int_{\R^d} \lt(\nabla W(x_k^n - y) - \nabla W(F^{t_n,\tau}(x_k^n) - F^{t_n,\tau}(y))\rho^n(y)\rt)dy\cr
&\leq \int_{\R^d} \|D^2 W\|_{L^\infty}\lt|x_k^n - y - (F^{t_n,\tau}(x_k^n) - F^{t_n,\tau}(y)) \rt|\rho^n(y)\,dy \leq C\Delta t.
\end{aligned}$$
Combining the estimates above, we obtain
\[
|(a)| \leq C\Delta t\lt(d_{BL}(\rho^n, \rho^n_h) + \Delta t \rt).
\]
For the estimate of $(b)$, we easily get from Remark \ref{rmk_jaco} that $|(b)| \leq C h \Delta t$. Finally, we observe that 
$(c)$ cannot be estimated as in the proof of Proposition~\ref{esteFn}, due to the lesser regularity of the densities.
We then proceed as follows,
$$\begin{aligned}
|(c)|& = \left| (x_k^n -x)\left(I_{d}- J^{t_n,t_{n+1}}(x_k^n)\right) +
         \int_{t_n}^{t_{n+1}}\left[u(\tau, F^{t_n,\tau}( x_k^n))-  u(\tau, F^{t_n,\tau}( x )) \right]d\tau 
        \right| \\
& \leq |x_k^n-x| \|I_{d} - J^{t_n,t_{n+1}}(\cdot) \|_{L^{\infty}} + \|D^2 W\|_{L^\infty} \int_{t_n}^{t_{n+1}}
     \left|  F^{t_n,\tau}( x_k^n) -F^{t_n,\tau}( x )  \right|d\tau \\
& \leq C h \, \Delta t  + C  \int_{t_n}^{t_{n+1}}
     \left|  F^{t_n,\tau}  \right|_{Lip} |x_k^n-x | d\tau \leq  C h \, \Delta t,
\end{aligned}$$
where we used estimate \eqref{est_supp_Snk} for $x \in S^n_{h,k}$, and the estimates \eqref{majFexlip} and \eqref{maj1moinsJex}. 
\end{proof}

\begin{theo}\label{thm_measure} Let $\rho^0$ be an initial probability measure on $\R^d$, and $\rho^n_h$ be the approximation constructed in \eqref{defrhohn}. Assume that the interaction potential $W$ satisfies $\nabla W \in \mw^{1,\infty}(\R^d)$, then the estimate
\[
\max_{0\leq n\leq \tfrac{T}{\Delta t}} d_{BL}(\rho^n,\rho^n_h) \leq C(d_{BL}(\rho^0,\rho^0_h) + h + \Delta t) 
\]
holds, where $C$ depends only on $d$ and $L$.
\end{theo}

\begin{rem}
Observe that a convergence condition on the approximation of the initial data in Theorem {\rm\ref{thm_measure}} such as $d_{BL}(\rho^0,\rho^0_h) \lesssim h$ is easily achieved by using a uniform quadrangular mesh of size $h^d$ and approximating the initial data $\rho^0$ by a sum of Dirac deltas via transporting the mass of $\rho^0$ inside each $d$-dimensional cube to its center. A cut-off procedure to leave small mass outside a large ball allows us to reduce to a finite number of Dirac deltas in this approximation. Finally, the error produced between smoothed particles and Dirac deltas is obviously of order $h$ in the $d_{BL}$ distance.
\end{rem}

\begin{proof}[Proof of Theorem \ref{thm_measure}] Since $\rho^n = F^{t_{n-1},t_n} \# \rho^{n-1}$ and $\varphi^n_{h,k} = F^{n-1}_{h,k} \# \varphi^{n-1}_{h,k}$, we obtain
\[
\int_{\R^d} \psi(x) d\rho^n(x) = \int_{\R^d} \psi(F^{t_{n-1},t_n}(x)) d\rho^{n-1}(x),
\]
and 
\[
\int_{\R^d} \psi(x) d\rho^n_h(x) = \sum_{k \in \Z^d}\omega_k\int_{\R^d} \psi(x) \varphi^n_{h,k}(x)\,dx = \sum_{k\in \Z^d}\omega_k\int_{\R^d} \psi(F^{n-1}_{h,k}(x))\varphi^{n-1}_{h,k}(x)\,dx,
\]
for $\psi \in \mw^{1,\infty}(\R^d)$ with $\|\psi\|_{\mw^{1,\infty}} \leq 1$.
Thus, we deduce
$$\begin{aligned}
&\int_{\R^d} \psi(x) \left( d\rho^n(x) - d\rho^n_h(x)\right) \cr
&\quad = \int_{\R^d} \psi(F^{t_{n-1},t_n}(x))\left(d\rho^{n-1}(x) - d\rho^{n-1}_h(x) \right)\cr
&\qquad + \sum_{k \in \Z^d}\omega_k\int_{\R^d} \left(\psi(F^{t_{n-1},t_n}(x)) - \psi(F^{n-1}_{h,k}(x)) \right) \varphi^{n-1}_{h,k}(x)\,dx\cr
&\quad =: (a) + (b).
\end{aligned}$$
Using
$\norm{\nabla(\psi \circ F^{t_{n-1},t_n})}_{L^\infty} \le \norm{(J^{t_{n-1},t_n})^{\sf T}}_{L^\infty} 
\norm{\nabla \psi}_{L^\infty}$,
it next follows from \eqref{majJex} that
$\abs{(a)} \leq d_{BL}(\rho^{n-1},\rho^{n-1}_h) e^{L\Delta t}$ and we estimate $(b)$ with
$$\begin{aligned}
\abs{(b)} 
    & \le \sum_{k\in\Z^d} \abs{\omega_k} 
     \int_{S^{n-1}_{h,k}} \Bigabs{ \psi(F^{t_{n-1},t_n}(x)) - \psi(F^{n-1}_{h,k}(x)) } \varphi^{n-1}_{h,k}(x)\,dx
     \\
    & \le e_F^{n-1} \sum_{k \in \Z^d}\abs{\omega_k} \int_{S^{n-1}_{h,k}} \varphi^{n-1}_{h,k}(x)\,dx
    \leq Ce_F^{n-1}
\end{aligned}
$$
where the last inequality uses the estimates inside \eqref{L1rhon}. This leads to
\[
d_{BL}(\rho^n,\rho^n_h) \leq d_{BL}(\rho^{n-1},\rho^{n-1}_h) e^{L\Delta t} + C e_F^{n-1}
\]
and using Lemma \ref{prop_measure} we obtain
\[
d_{BL}(\rho^n,\rho^n_h) \leq d_{BL}(\rho^{n-1},\rho^{n-1}_h) e^{C\Delta t} + C\Delta t(h + \Delta t)
\]
with constants independent of $\Delta t$ and $h$. The proof is then completed using 
Gronwall's inequality as in Theorem~\ref{esthetan}.

\end{proof}

\section{$L^1$ and $L^p$ convergence for singular potentials}
\setcounter{equation}{0}

In this part, we are interested in $L^p$-convergence between the solution and its approximation allowing for more singular potentials. With this aim, we consider the solutions of the equation \eqref{main-eq} in $L^\infty(0,T; L^\infty(\R^d) \cap \mw^{1,1}(\R^d) \cap \mw^{1,p}(\R^d))$ with $1 \leq p \leq \infty$ to be determined depending on the singularity of the potential. Since we are dealing with both attractive and repulsive potentials, we can only expect local in time existence and uniqueness of solutions as in \cite{BTR,CCH}. In those references, a local in time well-posedness theory in $L^\infty(0,T; L^1(\R^d) \cap L^{p}(\R^d))$ was developed under suitable assumptions on the potentials. The solutions are constructed by characteristics since the velocity fields are still Lipschitz continuous in $x$. However, to prove convergence rates we need more regularity on the solutions. For the existence of solutions to \eqref{main-eq} in $L^\infty(0,T;  L^\infty(\R^d) \cap \mw^{1,1}(\R^d) \cap \mw^{1,p}(\R^d))$, we provide a priori estimates in Appendix \ref{sec_appendix}, Proposition \ref{prop_lp}. These estimates combined with the existing literature \cite{CCH,BTR} show the well-posedness of solutions in the desired class. In our presentation we will follow the setting of local existence introduced in \cite{CCH}.

Let us remind the set of hypotheses on the interaction potential called singular potentials in the introduction. We assume that there exists $\tilde L>0$ such that
\beq \label{hypW}
|  \nabla W(x)|\leq \frac{\tilde L}{|x|^{\alpha}} \quad \textrm{and} \quad
|  D^2 W(x)|\leq \frac{\tilde L}{|x|^{1+\alpha}} \quad \mbox{with} \quad 0 \leq \alpha  < d-1,
\eeq
and for $- 1 \leq \alpha < 0$
\beq\label{hypW2}
|\nabla W(x)| \leq \tilde L\min \left\{ \frac{1}{|x|^\alpha}, 1\right\} \quad \mbox{and} \quad |D^2 W(x)| \leq \frac{\tilde L}{|x|^{1 + \alpha}}.
\eeq
In particular, singular potentials satisfy $\nabla W \in \mw^{1,q}_{\rm loc}(\R^d)$ for all $1 \le q < \frac{d}{\alpha+1}$. Note that \eqref{hypW} implies (see \cite{CCH, Hauray})
\beq \label{proprieteW}
|\nabla W(x)-\nabla W(y)| \leq \frac{C|x-y|}{\min(|x|,|y|)^{\alpha+1}}.
\eeq
We remind the reader that these assumptions are enough to guarantee that the velocity fields are bounded and Lipschitz continuous with respect to $x$ locally in time for densities in $(L^1\cap L^p)(\R^d)$ where $p$ is the conjugate exponent of $q$. Note that $ q = p' < \frac{d}{\alpha + 1}$ is equivalent to $\alpha < -1 + \frac{d}{p'}$, giving us the condition on the initial data for the well-posedness theory. Indeed, it follows from \eqref{hypW} that 
\begin{align}
\|D u(t,\cdot)\|_{L^\infty} &\leq \int_{\R^d} |D^2 W(x-y)| \rho(y)\,dy
    \leq \int_{\R^d} \frac{\tilde L\rho(y)}{|x-y|^{\alpha+1}}\,dy \nonumber\\
&\leq \left( \int_{|x-y| \geq 1} + \int_{|x-y| \leq 1} \right)\frac{\tilde L\rho(y)}{|x-y|^{\alpha+1}}\,dy 
\leq 
C (\|\rho(t,\cdot)\|_{L^1} + \|\rho(t,\cdot)\|_{L^p}), \label{est-Du}
\end{align}
for some constant $C$ depending on $\tilde L$, $q$ and $d$, and a similar estimate holds for $u$ using \eqref{hypW2} and the fact that $\nabla W$ is bounded away from the origin.

Let $T^*$ be the maximal time of existence of weak solutions $\rho \in L^\infty(0,T;(L^1\cap L^p)(\R^d))$ with $T<T^*$ constructed  in \cite{CCH}. Additional regularity will be needed on these solutions ensured by Proposition \ref{prop_lp} of Appendix \ref{sec_appendix} under suitable initial data assumptions. In this section we consider $T < T^*$, and we denote again $t_n = n\Delta t$ with $0\le n\le N$ and 
$\Delta t = T/N$ for some given positive integer $N$.
We introduce the following notations:
\begin{align*}
\begin{aligned}
\| \cdot \|&:=\|\cdot\|_{L^1}+ \|\cdot\|_{L^p}, \quad \Gamma^n_h &:= \|\rho^n - \rho^n_h\|, \quad \mbox{and} \quad \widetilde{\Gamma^n_h} := \sup_{0 \leq m \leq n}\Gamma^m_h.
\end{aligned}
\end{align*}

As for the convergence analysis, we point out that the proof of Section 3 cannot be directly applied. 
Indeed, it is not obvious to obtain an a piori bound on   
    $$
    \ds \sup_{0\leq n \leq N} \|\rho_h^n\|_{L^p}
    $$ 
    uniformly in $h$ and $\Delta t$, which we need to estimate
    $(\nabla W * \rho_h^n)$ and $(D ^2 W * \rho_h^n)$. 
    In order to do that, we will prove by induction that there is some $h_* > 0$ for which
    $$ 
    \sup_{0 < h \leq h_*}\, \widetilde{\Gamma^N_h} =
    \sup_{0 < h \leq h_*}\, \sup_{0\leq n \leq N} \Gamma^n_h  \le 1.
    $$%
We remind the reader that our error analysis between exact and approximated solutions for singular potentials requires non-negative weights for the particles, and this imposes us to give higher regularity on the initial data $\rho^0 \in \mw^{2,p}(\R^d)$, see Proposition \ref{prop-pos-weights}. Using the results in \cite{CCH} and Appendix \ref{sec_appendix}, we can obtain the existence and uniqueness of a solution 
$\rho \in L^\infty(0,T;(L^1\cap W^{2,p})(\R^d))$. However, in the next results we need less regularity in 
the solutions than on the initial data. Therefore, we prefer to keep both the assumptions stating the needed properties on the solution $\rho$ and the initial data $\rho^0$ to emphasize this fact.
 
Under the (induction) assumption that $\widetilde{\Gamma^n_h}$ is bounded uniformly in $h$ and $\Delta t$, we can derive the following estimates.

\begin{lem}\label{lem_ind}
If $M > 0$ and $n \leq N$ are such that $\widetilde{\Gamma^n_h} \leq M$,
and if the solution to \eqref{main-eq} satisfies $\rho \in L^\infty(0,T;(L^1\cap L^p)(\R^d))$, 
then we have 
\[
\sup_{0 \leq m \leq n}\|\rho^m_h\| \leq C_M \quad \mbox{and} \quad 
\sup_{0 \leq m \leq n}\lt( \sup_{x \in \tilde S^m_{h,k}} |x-x_k^m|\rt) \leq C_M(h+\Delta t),
\]
with a constant $C_M$ depending on $M$ but not on $h$ and $\Delta t$.
\end{lem}

\begin{proof}A straightforward computation yields
\[
\sup_{0 \leq m \leq n}\|\rho^m_h\| \leq \widetilde{\Gamma^n_h} + \sup_{0 \leq t \leq T}\|\rho(t)\| \leq C_M.
\]
In a similar way to \eqref{est-Du}, we also bound products like $\norm{D^{(i)} W * \rho^m_h}_{\L^{\infty}}$ by $C_W \|\rho^m_h\|$ 
with $C_W = \max( \norm{D^{(i)}W}_{L^q(B(0,1))}, \norm{D^{(i)} W }_{L^\infty (\R^d \setminus B(0,1))})$, $i\in\{1,2\}$, from which we derive estimates similar to those of Lemma~\ref{lem_jh}.
In particular, following the proof of Lemma \ref{lem_supp} we find that for $x \in  S^m_{h,k}$,
\beq \label{singbound-Smk}
|x - x^m_k| \leq \tilde L h \abs{(D^m_k)^{-1}}
\leq \tilde L h \exp\lt( \Delta t \sum_{l=0}^{m-1} \lt|(D^2 W * \rho^l_h)(x_k^l)\rt|\rt) 
\leq C_Mh,
\eeq
and for $x \in \tilde S^m_{h,k}$ we find 
$|x - x^m_k| \leq C_M(h+\Delta t)$. Note that this latter estimate involves bounding \eqref{singbound-Smk} on $S^{m+1}_{h,k}$ which only requires the norm $\norm{\rho^l_h}$ for $l \le m$, so that the resulting estimate indeed involves a constant depending on $M$.
\end{proof}

We next give the estimates of $ u(\tau,F^{t_m,\tau})-u_k^m$ for $\tau \in [t_{m},t_{m+1}]$ and $\tilde \xi_m(D^2 W)$ for $0 \leq m \leq n-1$. The proof can be obtained by using similar arguments as in Proposition \ref{propK2} with the help of Lemma \ref{lem_ind}
and a second-order estimate provided either by Proposition~\ref{prop_ini_app} or by a standard $L^p$ error estimate as described in Proposition \ref{prop-pos-weights}. We omit its proof, but point out that the crucial point is the smoothness assumptions \eqref{proprieteW} on the singular potential and the Lipschitz bound \eqref{est-Du} on the velocity field.

\begin{lem}\label{lem_prop1}
If $M > 0$ and $n \leq N$ are such that $\widetilde{\Gamma^n_h} \leq M$,
and if the solution $\rho \in L^\infty(0,T; \mw^{1,1}(\R^d) \cap \mw^{1,p}(\R^d))$ to \eqref{main-eq} with initial data $\rho^0 \in \mw^{2,p}(\R^d)$, then we have 
$$
\sup_{\tau \in [t_{m},t_{m+1}]}
| u(\tau,F^{t_m,\tau}(x_k^m))-u_k^m | \leq C_M\left(h^2 +\Delta t +\bar{e}_F^n \right)
$$
and
$$
\tilde \xi_m(D^2 W)  \leq C_M\left( h  + \Delta t + \Gamma^m_h\right) 
$$
for $0 \leq m \leq n$, with constants $C_M$ depending on $M$ but not on $h$ and $\Delta t$.
\end{lem}

We can also adapt the proof of Corollary \ref{Cor_ej}, Lemma \ref{lem_prop1}, and Proposition \ref{esteFn} to obtain the following result. 

\begin{lem}\label{lem_est_ej}
If $M > 0$ and $n \leq N$ are such that $\widetilde{\Gamma^n_h} \leq M$,
and if the solution $\rho \in L^\infty(0,T; \mw^{1,1}(\R^d) \cap \mw^{1,p}(\R^d))$ to \eqref{main-eq} with initial data $\rho^0 \in \mw^{2,p}(\R^d)$,
then we have 
$$
 e_j^m \leq C_M \Delta t(h + \Delta t + \Gamma^m_h)
$$
and
\beq\label{est_ef}
\tilde e_F^m \leq C_M \Delta t\lt( h^2 + \Delta t + \overline{e_F}^m + (h+\Delta t)\Gamma^m_h \rt)
\eeq
for $0 \leq m \leq n$, with constants $C_M$ depending on $M$ but not on $h$ and $\Delta t$.
\end{lem}

We finally connect the errors to the $L^1\cap L^p$ bounds on the densities.

\begin{lem}\label{lem_est_efs} 
If $M > 0$ and $n \leq N$ are such that $\widetilde{\Gamma^n_h} \leq M$,
and if the solution $\rho \in L^\infty(0,T; \mw^{1,1}(\R^d) \cap \mw^{1,p}(\R^d))$ to \eqref{main-eq} with initial data $\rho^0 \in \mw^{2,p}(\R^d)$,
then we have 
$$
 \overline{e_F}^{m+1} \leq C_M( h^2 + \Delta t + h \widetilde{\Gamma^{m}_h}),
$$
and
$$
 \tilde e_F^m \leq C_M \Delta t\lt( h^2 + \Delta t + (h+\Delta t) \widetilde{\Gamma^m_h} \rt)
$$
for all $0\leq m \le n$, with constants $C_M$ depending on $M$ but not on $h$ and $\Delta t$.
\end{lem}
 \begin{proof} Since Lemma~\ref{estbareFn} only relies on the Lipschitz smoothness of the exact flow, we have
 \[
 \overline{e_F}^{m+1} \leq e^{C\Delta t}\overline{e_F}^m + \tilde e_F^m
 \]
 for all $m$. Then from \eqref{est_ef} we derive
 \[
  \overline{e_F}^{m+1} \leq e^{(C+C_M)\Delta t}\overline{e_F}^m + 
  C_M \Delta t\lt( h^2 + \Delta t + h\Gamma^m_h \rt)
 \]
 for $m \le n$, so that Gronwall's inequality (together with $\widetilde{\Gamma^m_h} = \max_{m'\le m} \Gamma^{m'}_h$) yields
 \[
 \overline{e_F}^{m+1} \leq C_M(h^2 + \Delta t + h\widetilde{\Gamma^m_h})
 \]
 due to $\overline{e_F}^0 = 0$. Using this together with \eqref{est_ef} completes the proof.
 \end{proof}

We are now in a position to show the uniform $L^1 \cap L^p$ bounds on the density.
 
\begin{prop}\label{prop_lp2}
Assume that the interaction potential $W$ is singular in the sense of \eqref{hypW}-\eqref{hypW2}, and let $\rho$ be a solution to the equation \eqref{main-eq} up to time $T>0$, such that $\rho \in L^\infty(0,T;(\mw^{1,1} \cap \mw^{1,p} \cap L^\infty)(\R^d))$ with initial data $\rho^0 \in \mw^{2,p}(\R^d)$, $-1 \leq \alpha < -1 + d/p'$, and $1 < p \leq \infty$. Assume in addition that $\Delta t \lesssim h^2 \leq 1$. Then for all $M > 0$, there exists $h_*(M) > 0$ such that 
\[
\sup_{0 < h \leq h_*(M)} \, \sup_{0 \leq n \leq N} \Gamma^n_h \leq M.
\]
\end{prop}
\begin{proof} We use an induction argument on $n$. Since $\widetilde{\Gamma^0_h} = \Gamma^0_h \lesssim h^2$, clearly there exists $h_0(M)$ 
such that $\Gamma^0_h \leq M$ for all $h < h_0(M)$. We then assume that $n < N$ and $h_n(M) > 0$ are such that
\[
\sup_{0 < h \leq h_n(M)}  \widetilde{\Gamma^n_h} \leq M.
\]
For the remaining of the proof we then consider $m \le n$ and $h \le h_n(M)$. In particular, we observe that
the Lemmas above can be used with this value of $M$.
Decomposing the error as in Theorem \ref{esthetan}, we write
\begin{equation*}
\begin{aligned}
 &\rho^{m+1}(y) - \rho_h^{m+1}(y)\cr
 &\quad = \underbrace{ \left[\rho\left(t_{m}, F^{t_{m+1},t_{m}}(y)\right) 
     - \rho_h^{m}\left(F^{t_{m+1},t_{m}}(y)\right) \right]j^{t_{m+1},t_{m}}(y)     }_{A_{m+1}(y)} \\
&\quad+  \underbrace{ \sum_{k\in \Z^d} \frac{\omega_k }{h_k^m}  \varphi \left( \frac{D_k^m}{h}\lt( F^{t_{m+1},t_{m}}(y)- x_k^{m}\rt) \right)  \left[   j^{t_{m+1},t_{m}}(y) - \frac{1}{j_{k}^{m}} \right]}_{B_{m+1}(y)} \\
&\quad + \underbrace{\sum_{k\in \Z^d} \frac{\omega_k }{h_k^{m+1}} \left[ \varphi \left(  \frac{D_k^{m}}{h}(F^{t_{m+1},t_{m}}(y)- x_k^{m} ) \right) -  \varphi \left( \frac{D_k^{m+1}}{h}  (y- x_k^{m+1})  \right) \right]}_{C_{m+1}(y)}.
\end{aligned}
\end{equation*}
Using arguments similar than in Theorem \ref{esthetan} we find
$$ \| A_{m+1}\|_{L^p} 
 \leq  e^{C \Delta t} \|\rho_h^m-\rho^m\| _{L^p}  \quad \mbox{and} \quad
 \|B_{m+1}\|_{L^p} \leq Ce_j^{m}\|\rho^{m}_h\|_{L^p} \leq C_M e_j^m.
$$  
For the estimate of $C_{m+1}(y)$, we use the interpolation inequality and the estimates in Theorems \ref{esthetan} and \ref{prop_inf} to get
$$
\begin{aligned}
\|C_{m+1}\|_{L^p} &\leq \|C_{m+1}\|_{L^1}^{1/p}\|C_{m+1}\|_{L^\infty}^{1/q} \cr
    &\leq C_M\frac{(\tilde e_F^{m})^{1/p}}{h^{1/p}}
        \left( 1 + \frac{\overline{e_F}^{m} + \overline{e_F}^{m+1}}{h}\right)^{d/q}\frac{(\tilde e_F^{m})^{1/q}}{h^{1/q}} \cr
    &= C_M\left( 1 + \frac{\overline{e_F}^{m} + \overline{e_F}^{m+1}}{h}\right)^{d/q}\frac{\tilde e_F^{m}}{h}.
\end{aligned}
$$
Using Lemma~\ref{lem_est_efs} and the fact that $\widetilde{\Gamma^m_h} \leq M$ and $\Delta t \lesssim h$
we find that both $\overline{e_F}^{m}$ and $\overline{e_F}^{m+1}$ are bounded by $C_M h$, thus 
\[
\|C_{m+1}\|_{L^p} \leq C_M \frac{\tilde e_F^{m}}{h},
\]
and the above estimates yield
\[
\|\rho^{m+1} - \rho_h^{m+1}\|_{L^p} \leq e^{C\Delta t}\|\rho^m - \rho_h^m\|_{L^p} + C_M \lt(e_j^{m} + \frac{\tilde e_F^{m}}{h}\rt).
\]
We also observe that in the proof of Theorem \ref{esthetan}, all the steps leading to the estimate
\[
\theta_{m+1} \leq \theta_m + C_M \lt(e_j^{m} + \frac{\tilde e_F^{m}}{h}\rt )
\]
(where we remind that $\theta_m = \|\rho^{m} - \rho_h^{m}\|_{L^1}$) 
are valid in the case of singular potentials. This yields
\[
\Gamma^{m+1}_h \leq e^{C\Delta t}\Gamma^m_h + C_M\lt( e_j^{m} + \frac{\tilde e_F^{m}}{h}\rt).
\]
On the other hand, it follows from Lemmas \ref{lem_est_ej} and  \ref{lem_est_efs} that
\[
e_j^m \leq C_M \Delta t(h + \Delta t + \Gamma^m_h) \leq C_M\Delta t\lt( h + \Gamma^m_h\rt),
\]
and
\[ 
\frac{\tilde e_F^{m}}{h} \leq C_M\Delta t \lt(h + \frac{\Delta t}{h} + \lt(1 + \frac{\Delta t}{h}\rt)\widetilde{\Gamma^m_h}\rt) \leq C_M\Delta t (h + \widetilde{\Gamma^m_h}),
\]
where we used the assumption $\Delta t \lesssim h^2$. Thus we find 
\[
\widetilde{\Gamma^{m+1}_h} \leq e^{(C+C_M)\Delta t}\widetilde{\Gamma^m_h} + C_Mh\Delta t.
\]
Since this is valid for all $m\le n$, it follows from Gronwall's lemma that $\widetilde{\Gamma^{n+1}_h} \leq C_Mh$ holds for some constant $C_M > 0$. We remind the reader that $C_M$ is the generic constant depending on $M$ but independent of $h$ and $\Delta t$. In particular, setting $h_{n+1}(M) := \min(h_n(M), M/C_M)$ allows to write
\[
\sup_{0 < h \leq h_{n+1}(M)} \widetilde{\Gamma^{n+1}_h} \leq M.
\] 
This ends the induction argument and the proof, by taking $h_*(M) = h_N(M)$.
\end{proof}

Putting together all the results in this section, we obtain the main convergence result in $(L^1 \cap L^p)(\R^d)$.

\begin{theo} \label{ThcvLp} 
Assume that the interaction potential $W$ is singular in the sense of \eqref{hypW}-\eqref{hypW2}, and let $\rho$ be a solution to the equation \eqref{main-eq} up to time $T>0$, such that $\rho \in L^\infty(0,T;(\mw^{1,1} \cap \mw^{1,p} \cap L^\infty)(\R^d))$ with initial data $\rho^0 \in \mw^{2,p}(\R^d)$, $-1 \leq \alpha < -1 + d/p'$, and $1 < p \leq \infty$. 
Assume in addition that $\Delta t \lesssim h^2 \leq 1$. Then
\begin{equation*} 
\sup_{0 < h \leq h_*} \sup_{0\leq n \leq N} \|\rho_h^n-\rho^n\| \leq Ch
\end{equation*}
holds with $h_* = h_*(1)$ given by Proposition \ref{prop_lp2} and a constant $C$ independent of $h$ and $\Delta t$.
 \end{theo}

\section{Numerical Results}
\setcounter{equation}{0}

We will present in this Section some numerical examples in one dimension, with different interaction potentials and initial densities to showcase some of the features already observed in numerical and theoretical analysis of the aggregation equation \eqref{main-eq} in \cite{FellnerRaoul1, FellnerRaoul2,BT2,HB,BLL,BCLR2}. In this way, we first validate our numerical implementation in order to explore some less-known properties about the behavior of its solutions in one dimension. A further more complete numerical study in 2D of this method will be reported elsewhere. These examples already show the wide range of different behaviors of solutions to the aggregation equation. 

\subsection{Numerical method: validation and implementation}
We have implemented the numerical method described in Section \ref{secmethodeLTP} using Python. 
We  use different initial conditions depending on the behaviors we would like to show. Specifically, 
we consider as initial densities
\beq \label{rhoini1}
    \rho^0_1(x)=(e^{-30(x-0.5)^2}+ 2e^{-50(x+0.3)^2})\one_{[-1,1]}(x),
\eeq
\beq \label{rhoini3}
\rho^0_2(x)=
    \one_{[-1,1]}(x),
\eeq
\beq \label{rhoini4}
    \rho^0_3(x)= e^{(x^2-1)^{-1}} \one_{[-1,1]}(x),
\eeq 
in order to have asymmetric, discontinuous symmetric and compactly supported smooth initial data respectively. These initial densities have been normalized to have unit mass.
\begin{figure}[h!]
   \centering
   \begin{minipage}[c]{0.48 \textwidth}
        \includegraphics[width=6.8cm]{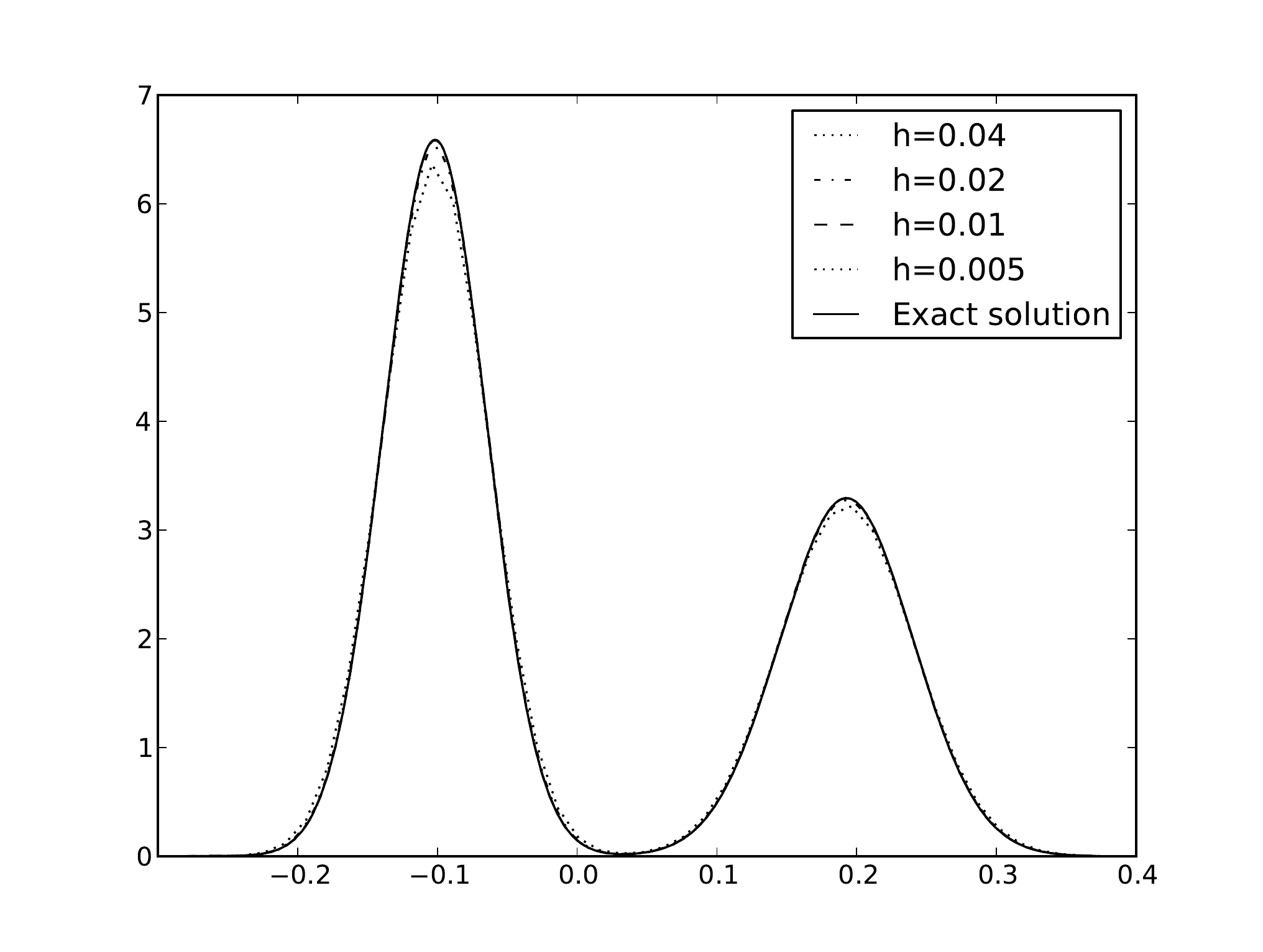}
   \end{minipage}
   \begin{minipage}[c]{0.48 \textwidth}
      \includegraphics[width=6.8cm]{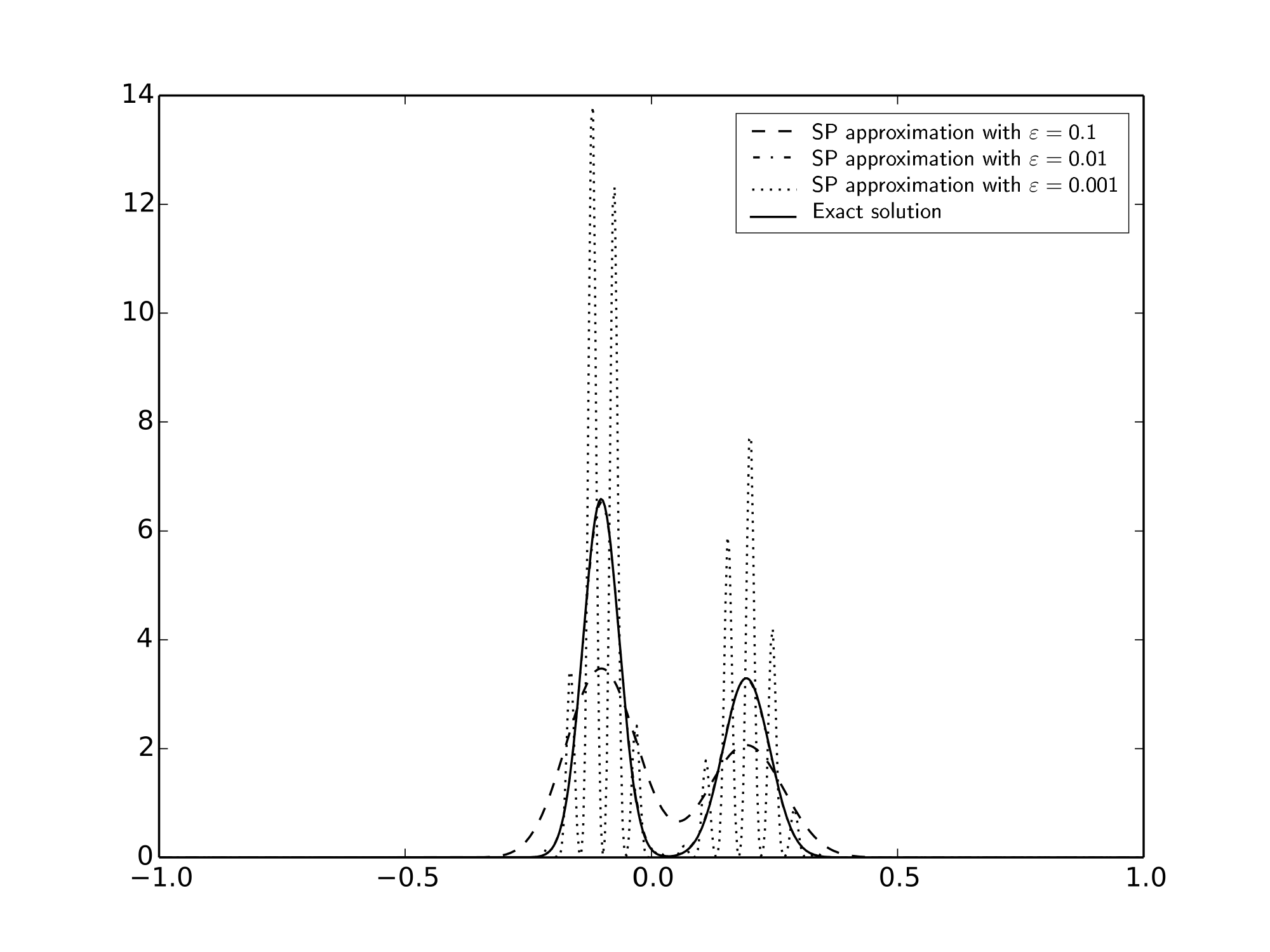}
   \end{minipage}
        \caption{\label{figvalid}
Comparisons for $W(x)=x^2$, at $t=0.5$ with initial data \eqref{rhoini1} and $\Delta t=10^{-4}$.
 Left Figure: Comparison between  the exact solution $\rho(t,x)$ given by \eqref{solexactW2} and  $\rho_h^n$ computed by LTP method for various $h$.
Right Figure: Approximated values $\rho_h^n$ obtained with a classical Smooth Particle (SP) method for $h=0.01$ and different fixed particle sizes $\varepsilon$.}
\end{figure}
Shape functions for the particle method are here
 B3-splines  given by 
 \eqref{B3-spline}.  
We first examine the validation of our code by comparison of the numerical solution and the exact solution 
of \eqref{main-eq} with $W(x)= x^2$. Due to the conservation of the center of mass,
$$
 \forall t \geq 0, \qquad \int_\R x\rho(t,x)dx=  \int_\R x\rho^0(x)dx:=\lambda\,,
$$
the solution is explicitly given by 
\beq \label{solexactW2}
\rho(t,x)=\rho^0\left((x-\lambda)e^{2t}+\lambda   \right)e^{2t}\,,
\eeq
using the method of characteristics. Figure \ref{figvalid} (left) shows the exact solution of \eqref{main-eq} 
with initial data \eqref{rhoini1} and the numerical solution computed with the LTP method, together 
with the $L^1$ and $L^{\infty}$ errors with respect to $h$. 

Let us now compare the results with classical particle methods. 
One of the drawbacks of classical particle methods in which the density is reconstructed with shape function of same size
$$
\rho_{\varepsilon}^n(x)= \sum_{k\in \Z} \omega_k \frac{1}{\varepsilon} \varphi \left( \frac{x-x_k^n}{\varepsilon} \right) \,,
$$
is the need to choose adequate values of $\varepsilon$. Indeed, if $\varepsilon$ is too  small compared to the distance between two particles, the reconstructed density will vanish   between particles and is thus irrelevant; and if   $\varepsilon$ is too large the reconstructed density will be too spread out and the results lack accuracy,  as it is demonstrated in Figure \ref{figvalid} (right). 

\begin{figure}[h]
   \centering
   \begin{minipage}[c]{0.46 \textwidth}       
        \includegraphics[width=6.5cm]{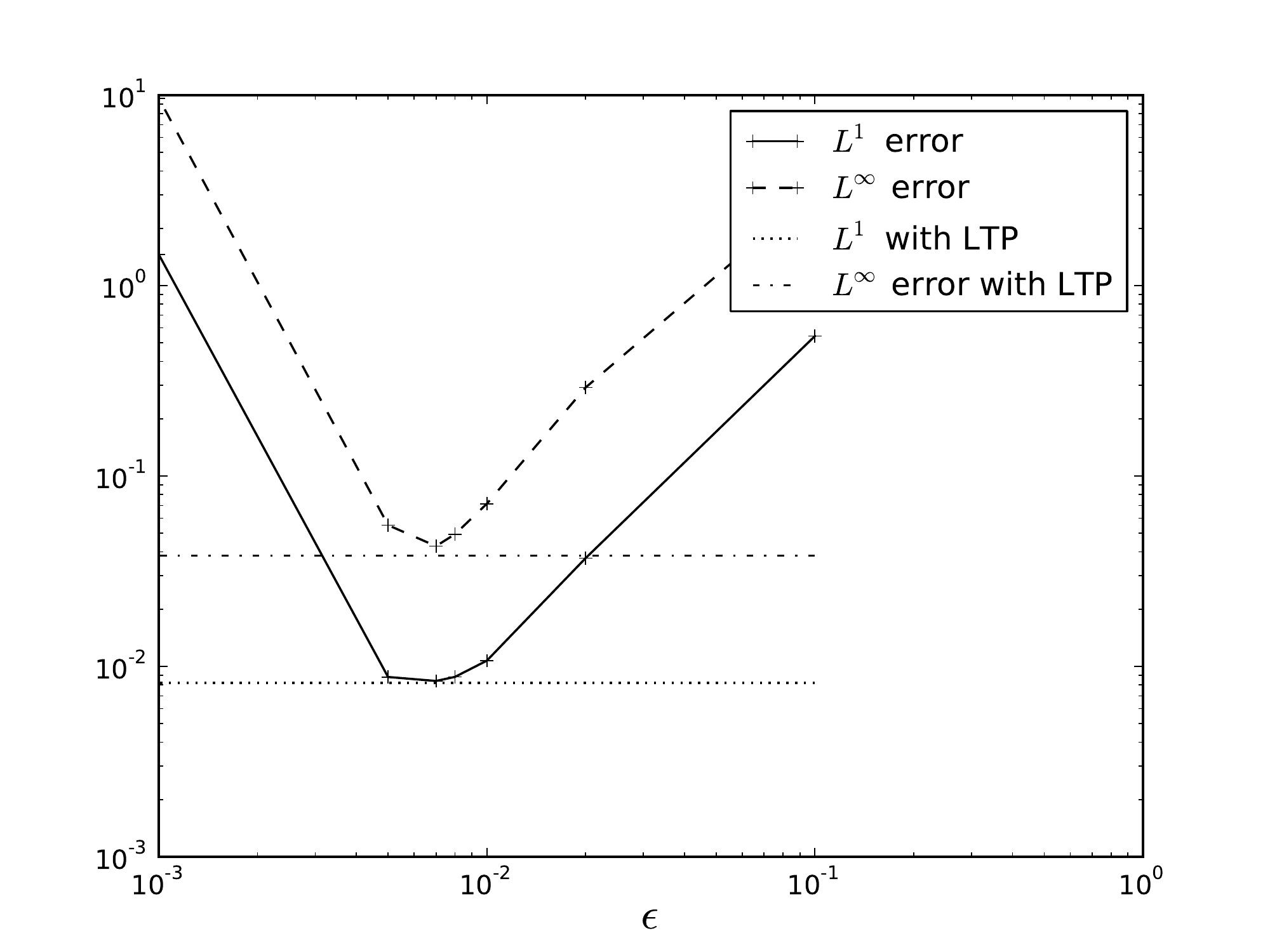}
   \end{minipage}
  \begin{minipage}[c]{0.46 \textwidth}
         \includegraphics[width=6.5cm]{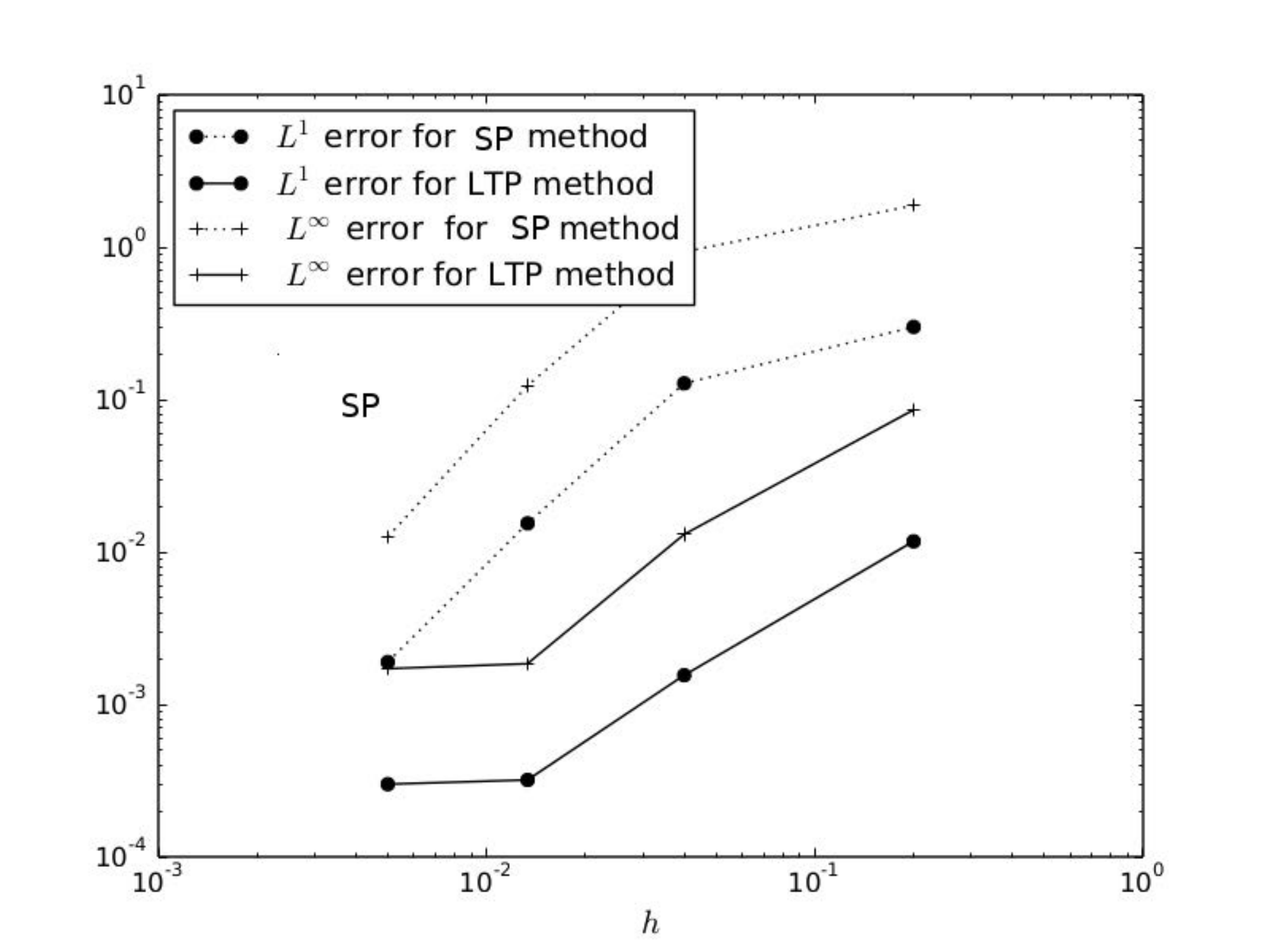}
   \end{minipage}
        \caption{\label{figcompapart2} 
Left Figure: Log-Log Plot of the $L^1$ and $L^{\infty}$ errors of the SP method using various values for the particle radius $\varepsilon$, versus those of the LTP method (both using $h=0.01$ and $\Delta t=0.01$).
Right Figure: Log-Log Plot of the $L^1$ and $L^{\infty}$ errors for the SP and LTP methods with $\Delta t=0.0001$ and different values 
of $h=\varepsilon$. In both cases, $W(x)=x^2$ and $\rho^0$ is given by \eqref{rhoini1} with errors computed at $t=0.5$.}        
\end{figure}

Figure \ref{figcompapart2} presents the $L^1$ and $L^{\infty}$ errors between a standard Smooth Particle (SP) method (with different values of $\varepsilon$) and our LTP method for the potential $W(x)=x^2$ for which solutions are explicit by the method of characteristics, see \eqref{solexactW2}. On the left picture, we observe that the optimal $\varepsilon$, at this instant $t$, for a classical particle method is well captured with the LTP method. 

One could object that the gain is not significantly better with the LTP method. However, since particles aggregate, the average distance between two particles decreases exponentially in time, and consequently the optimal size $\varepsilon$ for reconstruction in classical particle method is not the same during the whole simulation. Therefore, an evolution in time of $\varepsilon$ is much better adapted. Notice that the case of potential $W(x)=x^2$ is not particularly the best example to show the higher accuracy of the LTP method with respect to the classical particle method since all particles have the same size at time $t$ because $j_{\ex}^{0,t}(x)=e^{-2t}$,  see \eqref{solexactW2}. 
Moreover, the gain of accuracy with the LTP method is  even clearer while using B1-spline as shape functions instead of B3-spline, as it is demonstrated on Figure \ref{comparaisonB1etB3}.
\begin{figure}[h]
   \centering
   \begin{minipage}[c]{0.46 \textwidth}       
        \includegraphics[width=6.5cm]{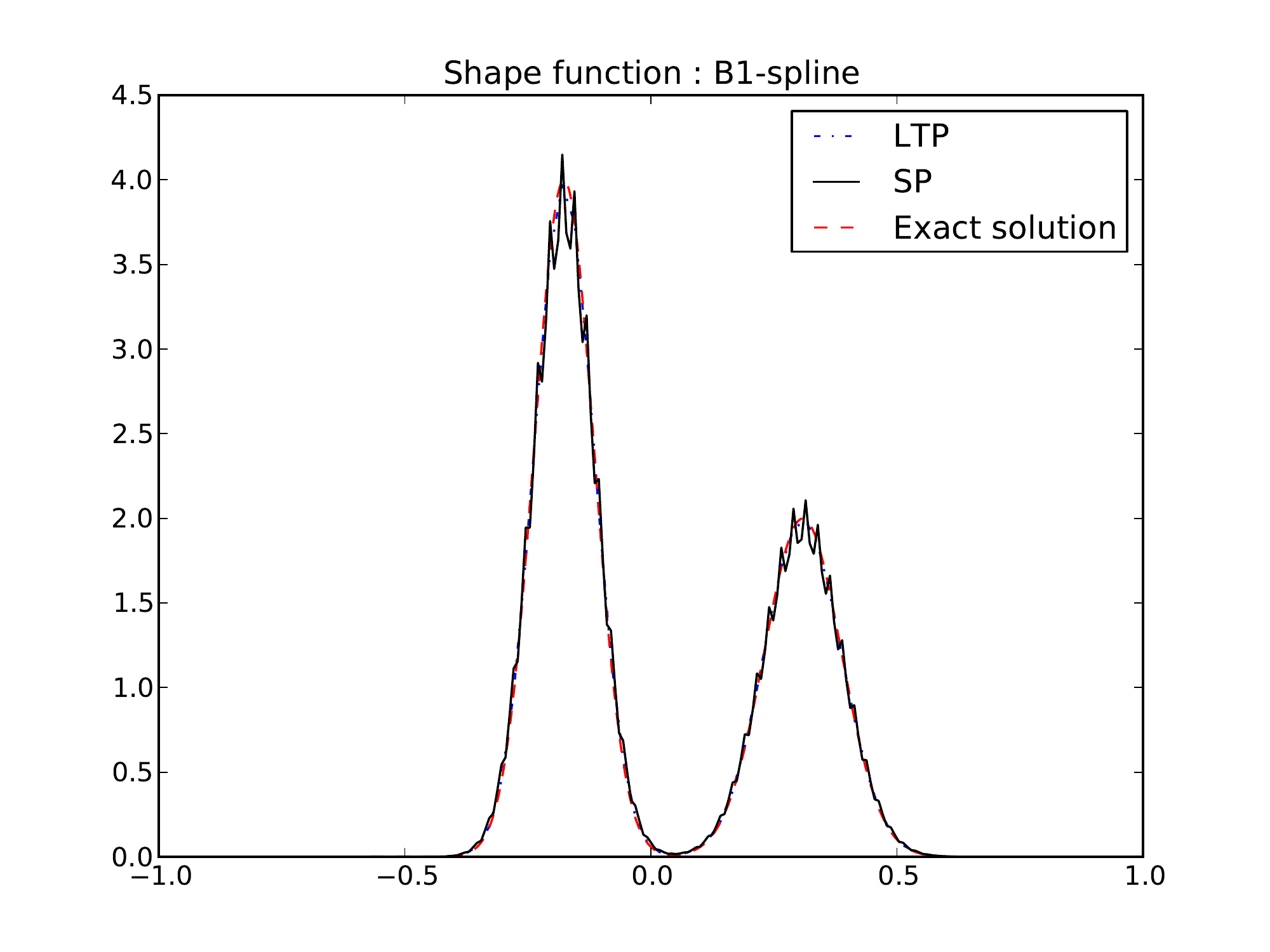}
   \end{minipage}
  \begin{minipage}[c]{0.46 \textwidth}
         \includegraphics[width=6.5cm]{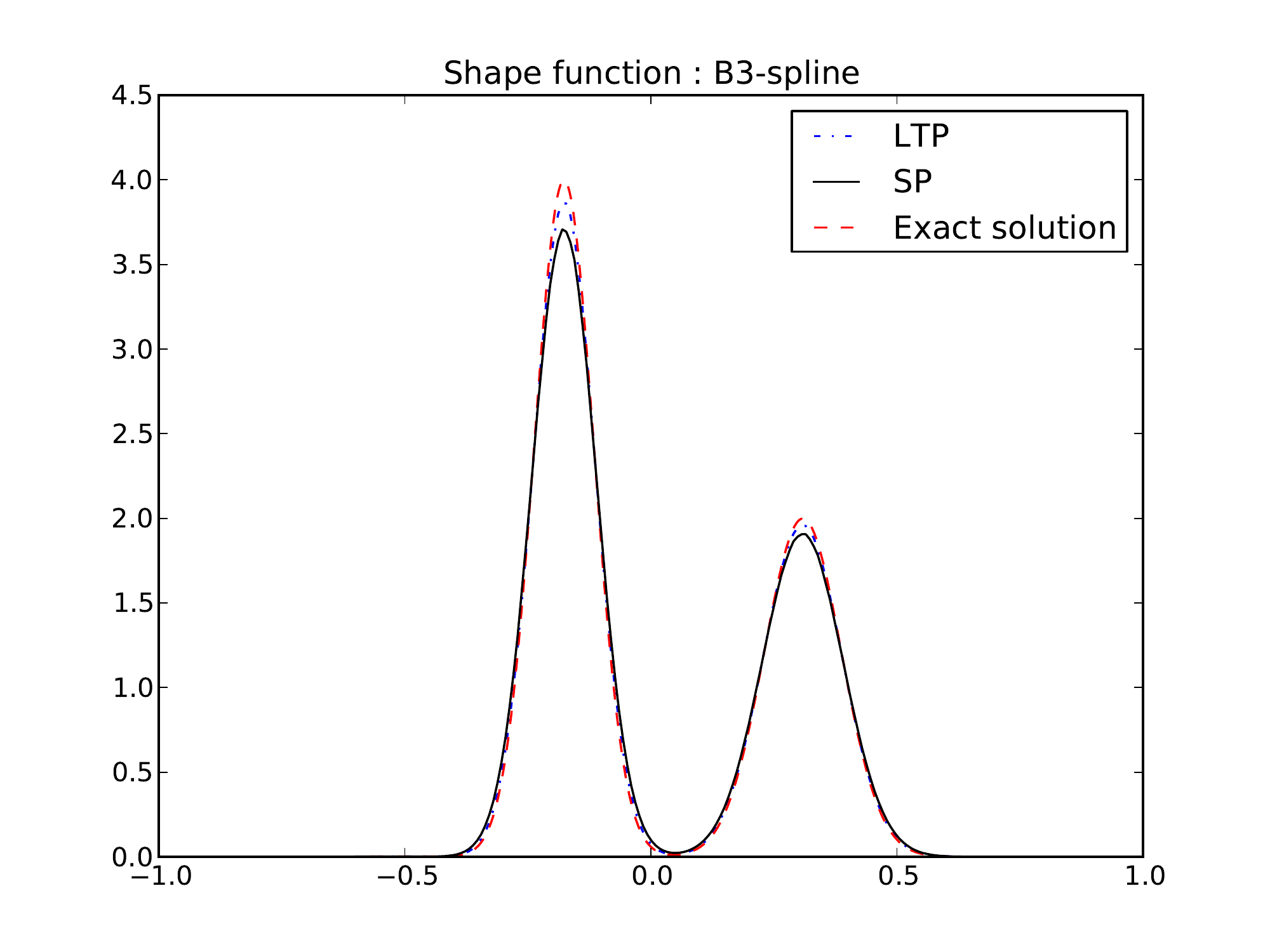}
   \end{minipage}
        \caption{\label{comparaisonB1etB3} 
Comparisons between exact solution (at $t=0.5$) and  approximated solutions $\rho_h^n$ with LTP method and SP method, with $W(x)=x^2/2$, $\rho^0$ given by \eqref{rhoini1},  $h=1/25$ and $\Delta t=10^{-3}$. On the left Figure, the shape function $\varphi$ is a hat function, whereas on the right Figure, $\varphi$ is a B3-spline.} 
\end{figure}

\subsection{Numerical Simulations}

We now take advantage of the method to explore the behavior for other attractive potentials of type $W(x)=\tfrac{|x|^{a}}{a}$, $a>1$. 
Notice that for $a\geq 2$ the potential is smooth while for $1<a<2$ is singular once $W$ is cut-off at infinity or if the initial 
data is compactly supported since the effective values of the potential lie on a bounded set and $W$ can be cut-off at infinity 
without changing the solution.
Figure \ref{figWa} presents the numerical results obtained by the LTP method in the case of $a=1.5$ and $a=2.5$. We represent the approximation of the density $\rho_h^n$, and also the reconstructed velocity $u_h^n$ and the reconstructed size of particles $h^n$ by piecewise linear interpolation such that
$$
u_h^n(x_k^n) = - \nabla W * \rho_h^n(x_k^n),
\qquad
\mbox{and}
\qquad
h^n(x_k^n) = h  \prod_{m=0}^{n-1} j_k^m.
$$
\begin{figure}[h!]
\centering
\begin{tabular}{|c|c|c|}
\hline
      & $a=1.5$ & $a=2.5$  \\
\hline
Potential $W$
   &    
    \begin{minipage}[c]{0.4 \linewidth} 
    \centering
    \includegraphics[width=5cm]{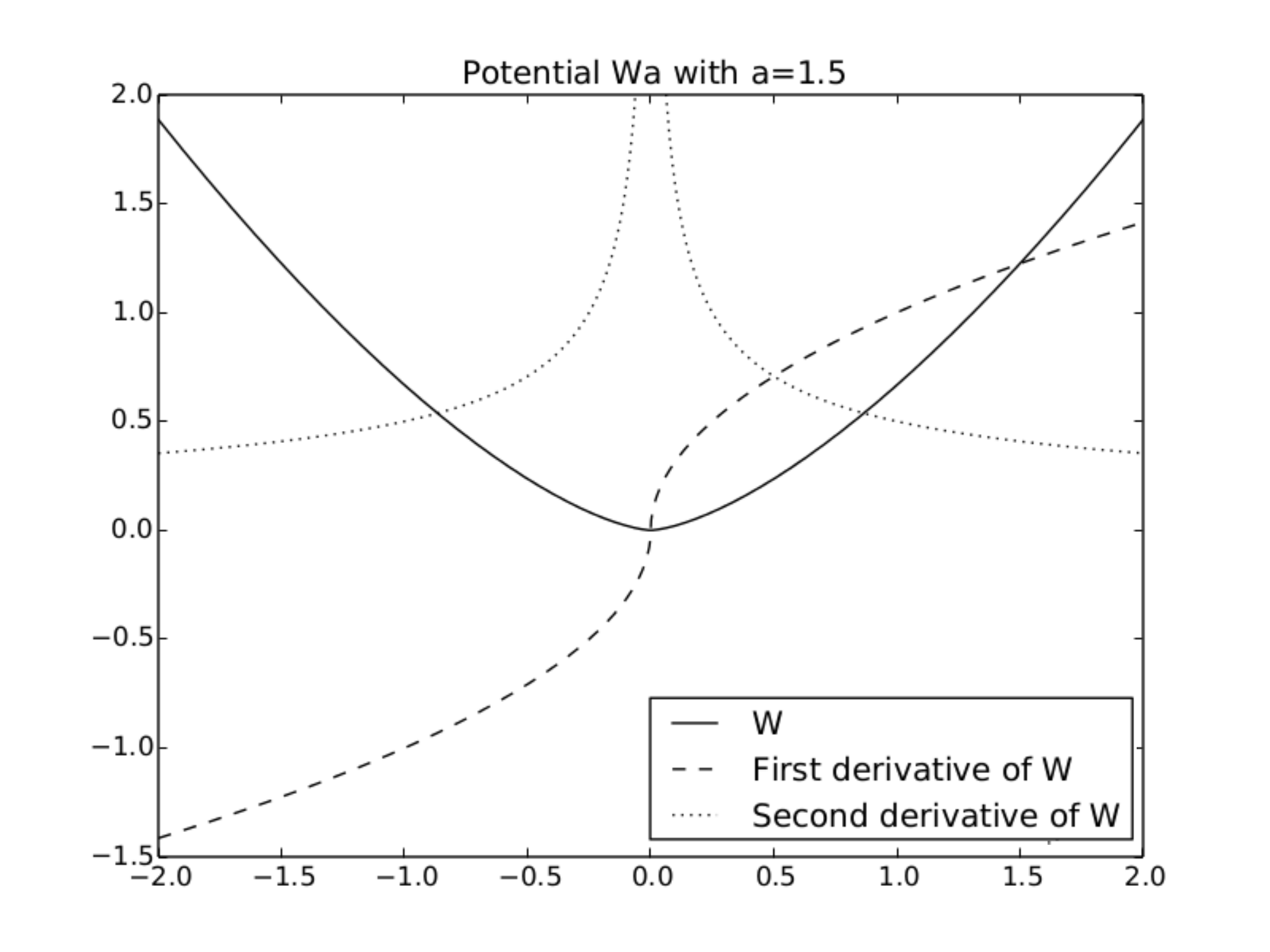} 
   \end{minipage} 
    & 
       \begin{minipage}[c]{0.4 \linewidth} \centering
    \includegraphics[width=5cm]{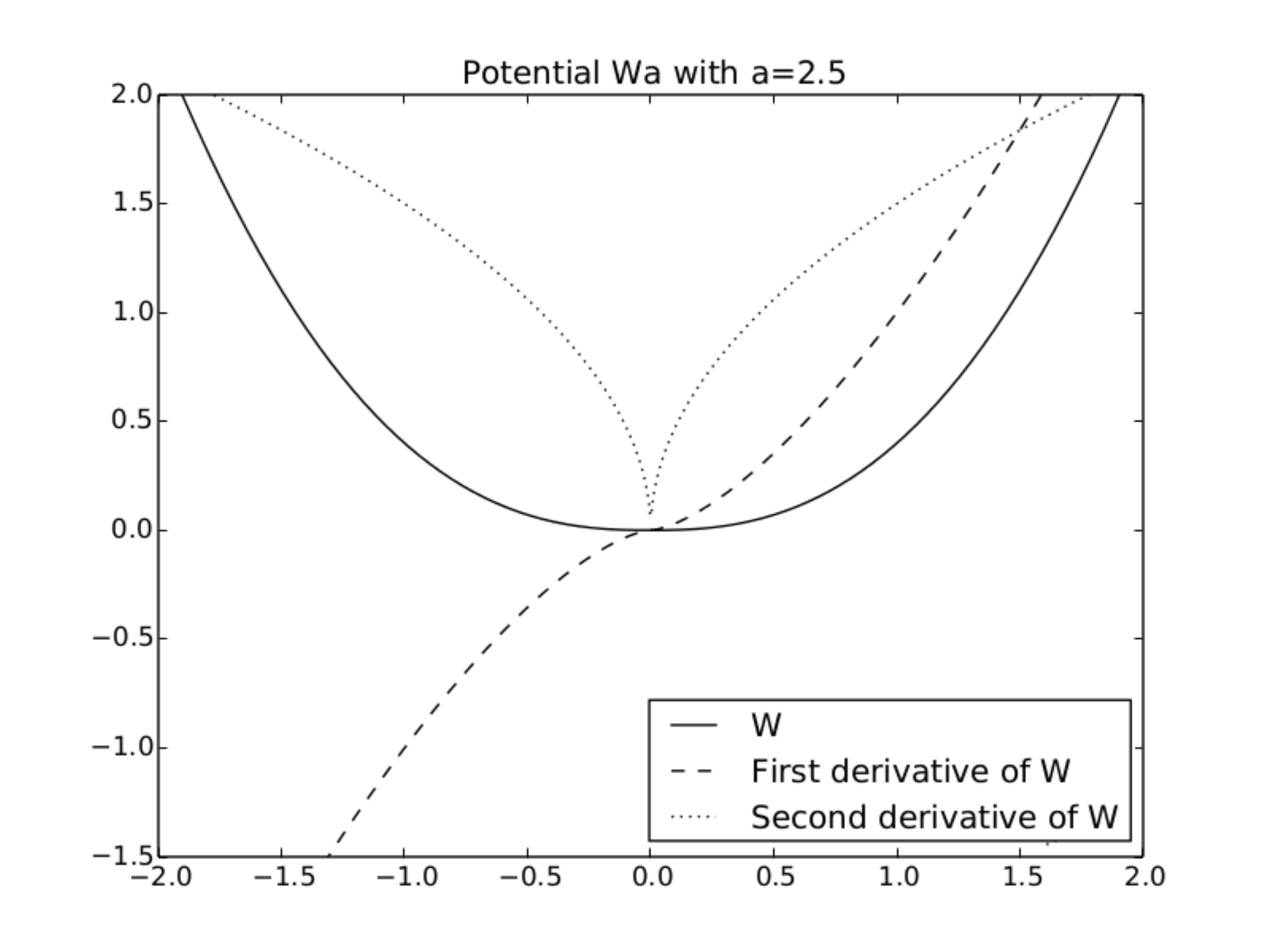} 
   \end{minipage}
\\
\hline
Density $\rho_h^n$ 
   & 
   \begin{minipage}[c]{0.4 \linewidth} \centering
    \includegraphics[width=5cm]{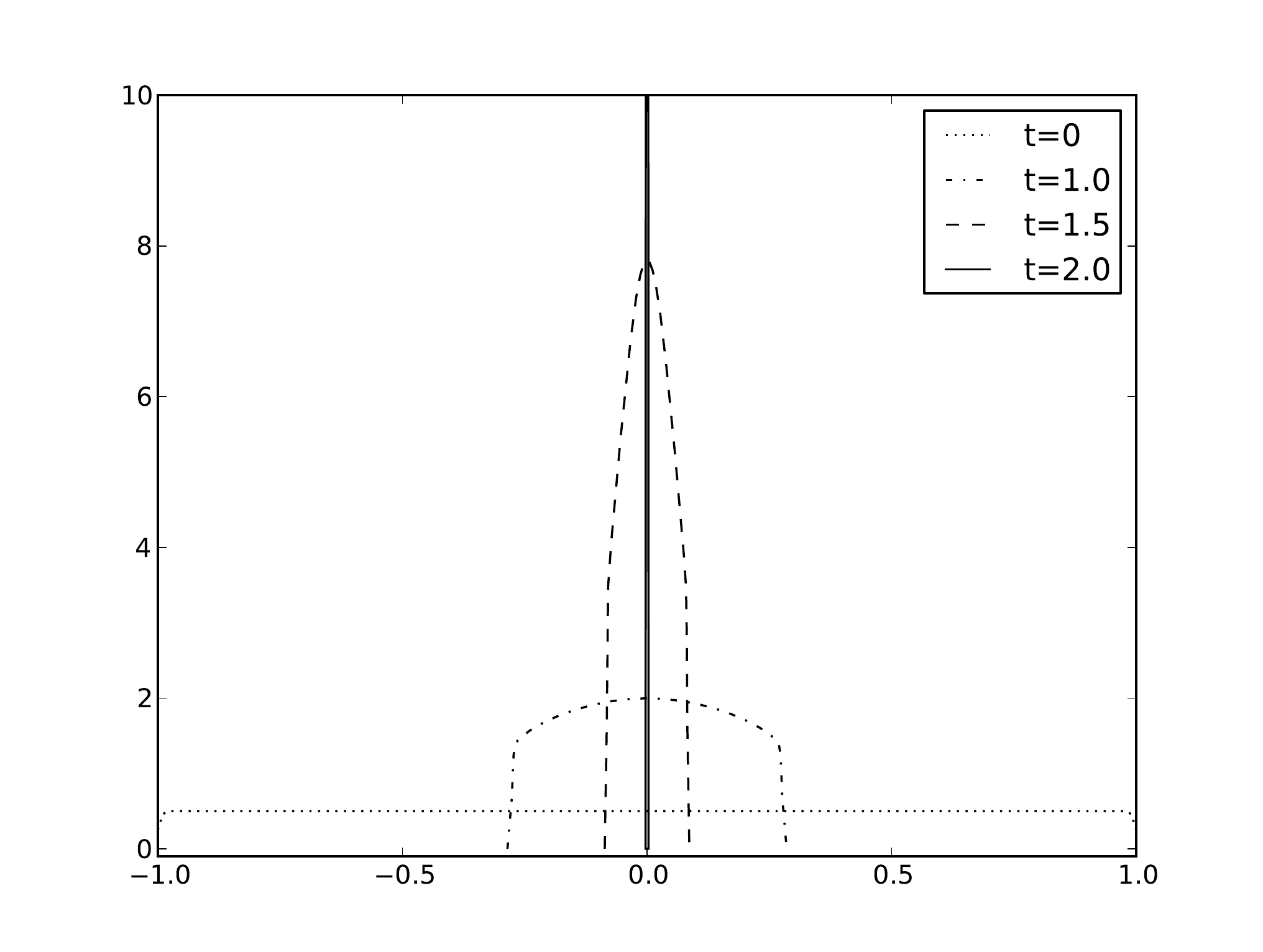}
   \end{minipage}  
   &
   \begin{minipage}[c]{0.4 \linewidth} \centering
    \includegraphics[width=5cm]{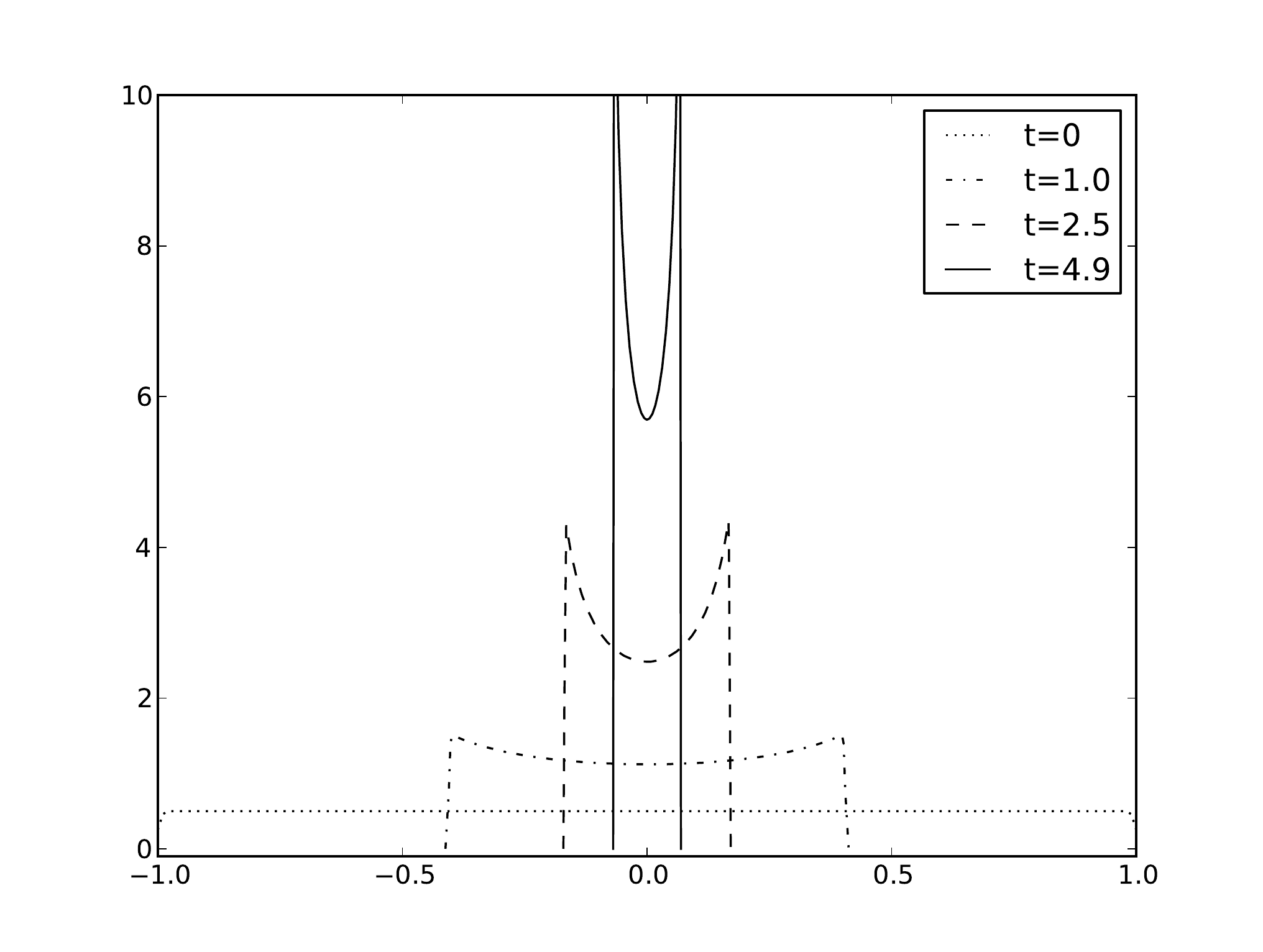}
   \end{minipage} \\
\hline
 \begin{minipage}[c]{ 0.18\linewidth}\centering
Velocity \\ $\nabla W \ast \rho_h^n$
    \end{minipage}
   &
    \begin{minipage}[c]{0.4 \linewidth} \centering
    \includegraphics[width=5cm]{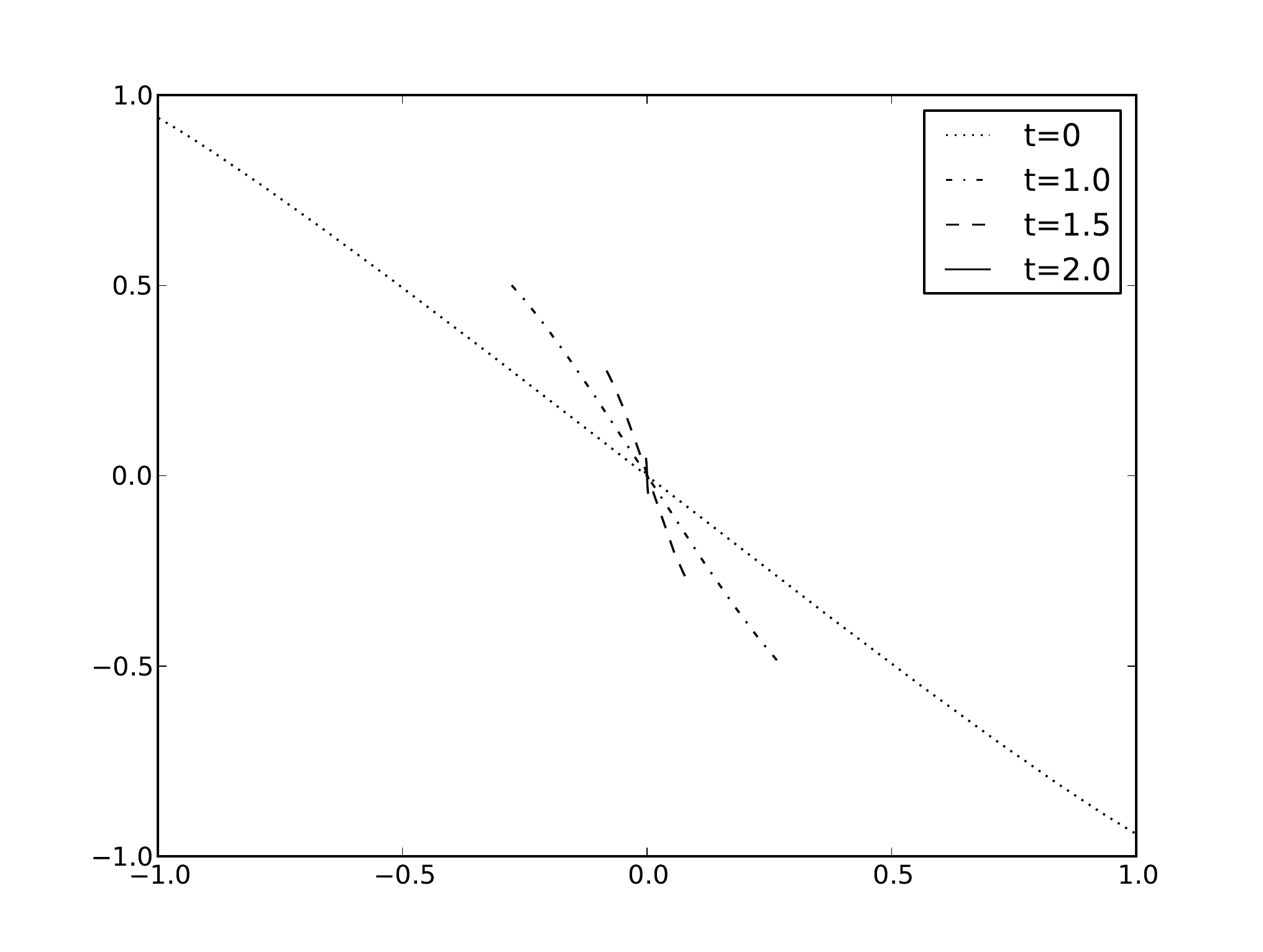} 
   \end{minipage}
  &
    \begin{minipage}[c]{0.4 \linewidth} \centering
    \includegraphics[width=5cm]{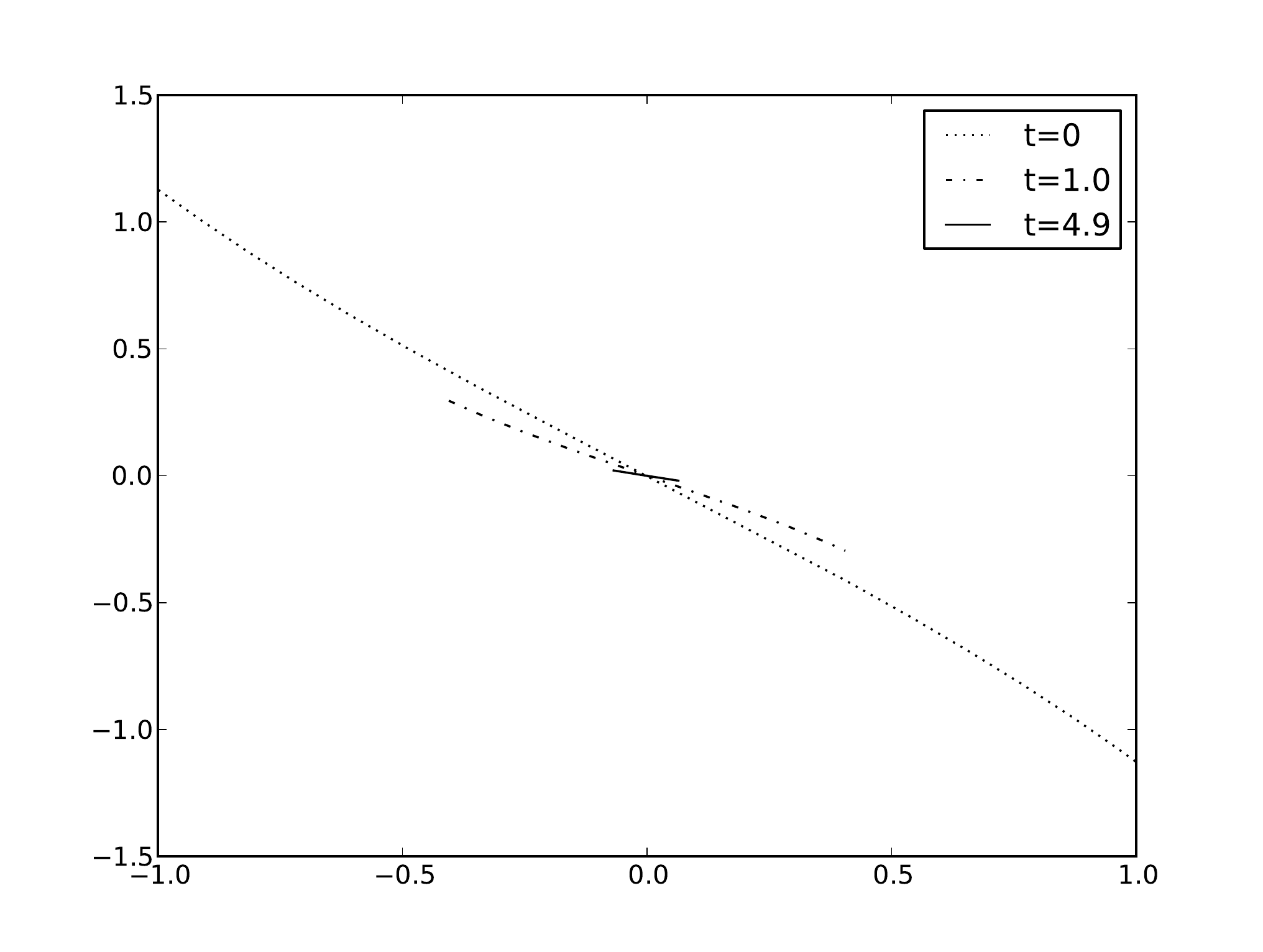}
   \end{minipage} \\
\hline
 \begin{minipage}[c]{ 0.18\linewidth}\centering
  Size of \\ particles $h_k^n$
    \end{minipage}
  &
    \begin{minipage}[c]{0.4 \linewidth} \centering
    \includegraphics[width=5cm]{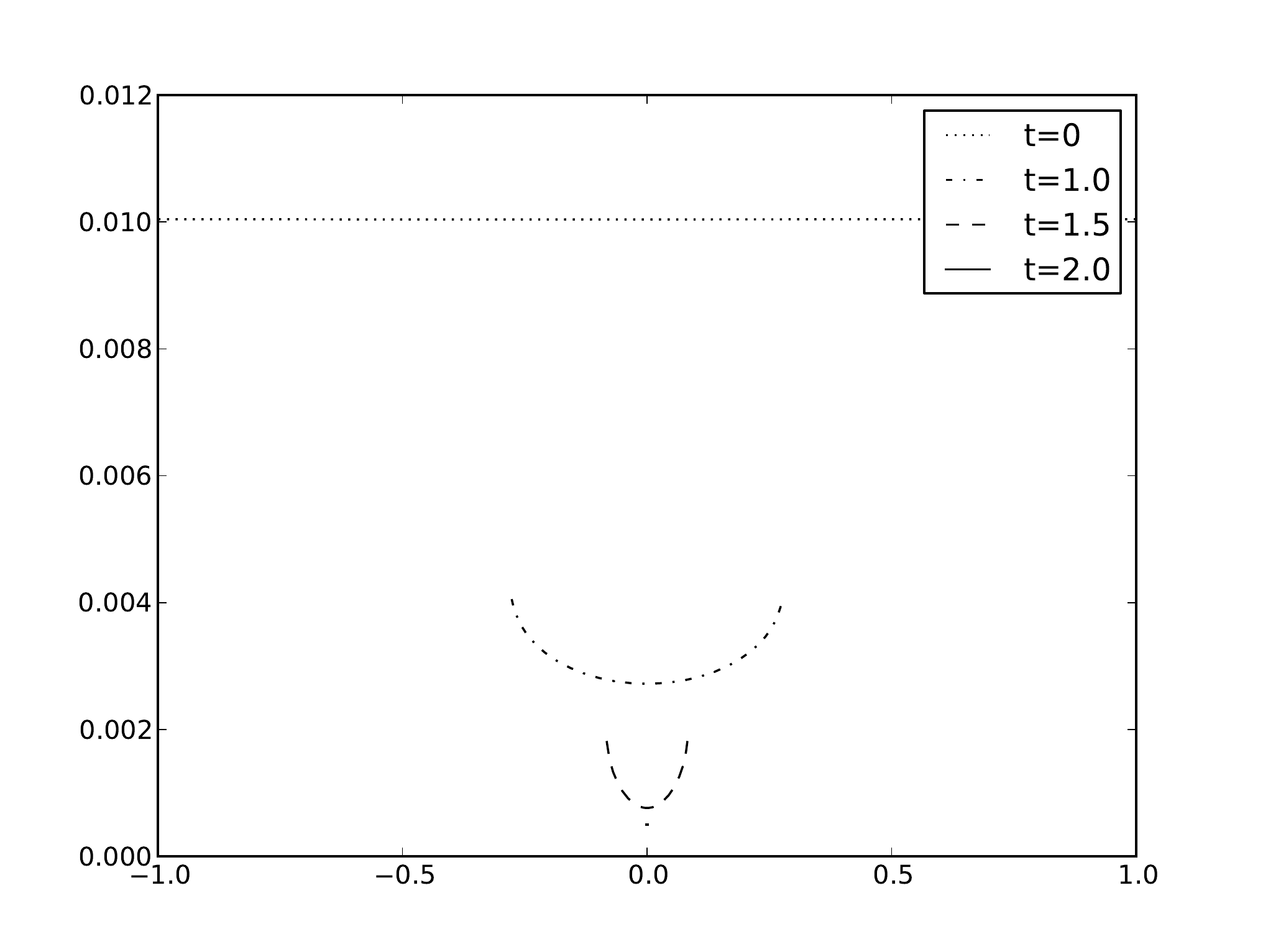}
   \end{minipage}
  &
    \begin{minipage}[c]{0.4 \linewidth} \centering
    \includegraphics[width=5cm]{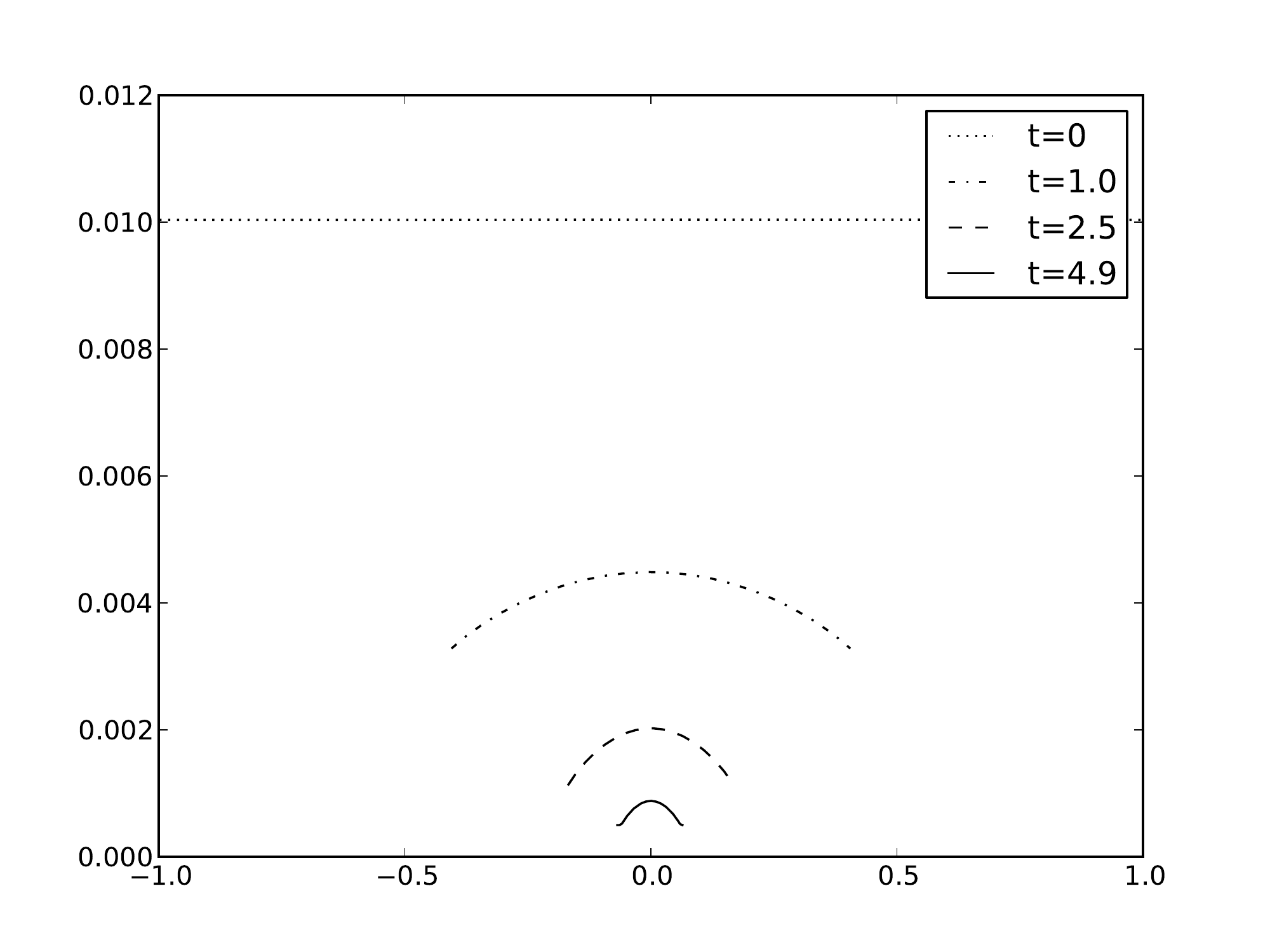}
   \end{minipage}\\
\hline
\end{tabular}
\caption{\label{figWa} Approximated density and reconstructed velocity and size of particles computed by the LTP method with $h=0.01$ for $W(x)=\tfrac{|x|^a}{a}$, with $a=1.5$ or $a=2.5$ and $\rho^0$ given by \eqref{rhoini3} with the number of time-steps $N=200$.}
\end{figure} 
Potentials and their derivatives are also represented. In both cases, we observe that the density converges to a Dirac mass. Figure \ref{figWa} also shows that for $a=2.5$, $W'' \in L_{loc}^{\infty}$, no finite-time blow-up in $L^\infty$ appears, opposite to the case $a=1.5$ in agreement with the results proved in \cite{BCL}. Notice also the different qualitative behavior in their trend to blow-up as studied in \cite{HB}.

\begin{figure}[h!]
\hspace{-1cm}
\begin{tabular}{|c|c|c|c|}
\hline
& $b=1.1$ & $b=1.5$ & $b=2.5$ \\
\hline
Potential $W$
    &    
    \begin{minipage}[c]{0.3 \linewidth} 
    \includegraphics[width=4cm]{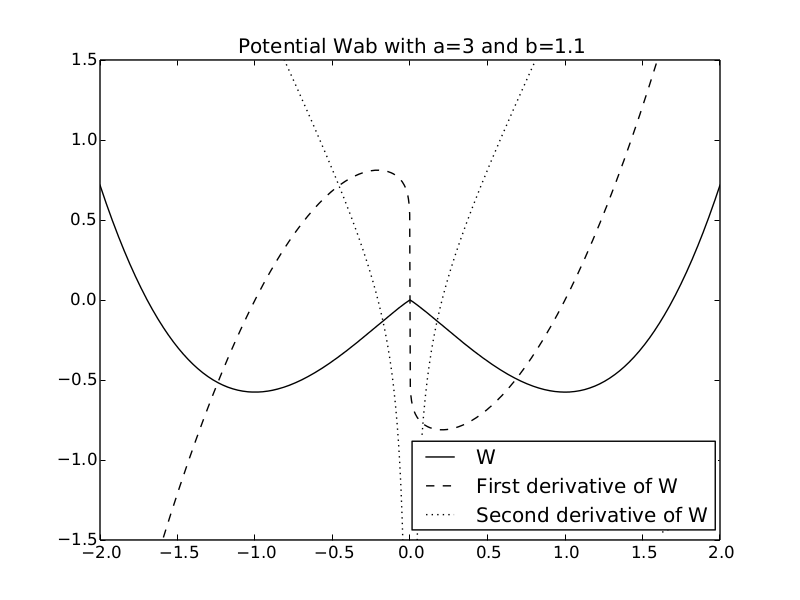} 
   \end{minipage}
    & 
    \begin{minipage}[c]{0.3 \linewidth} 
    \includegraphics[width=4cm]{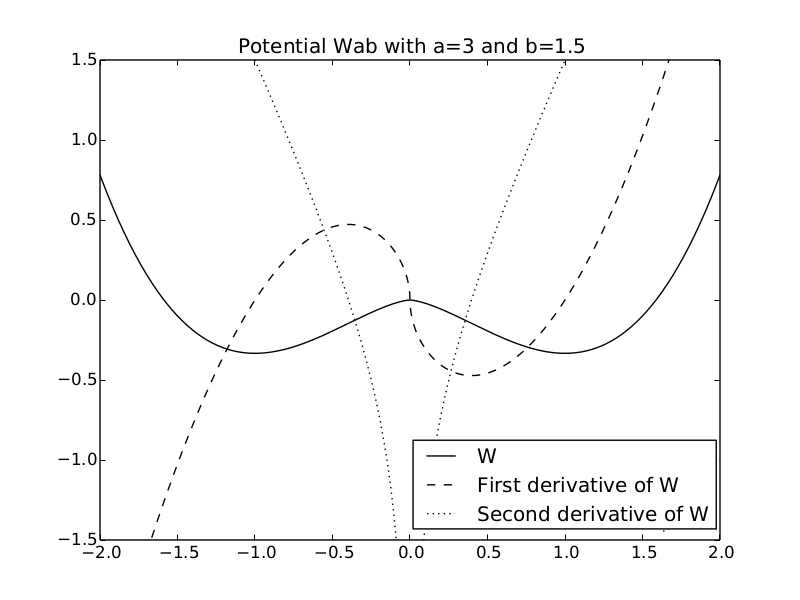} 
   \end{minipage}
    &
    \begin{minipage}[c]{0.3 \linewidth} 
    \includegraphics[width=4cm]{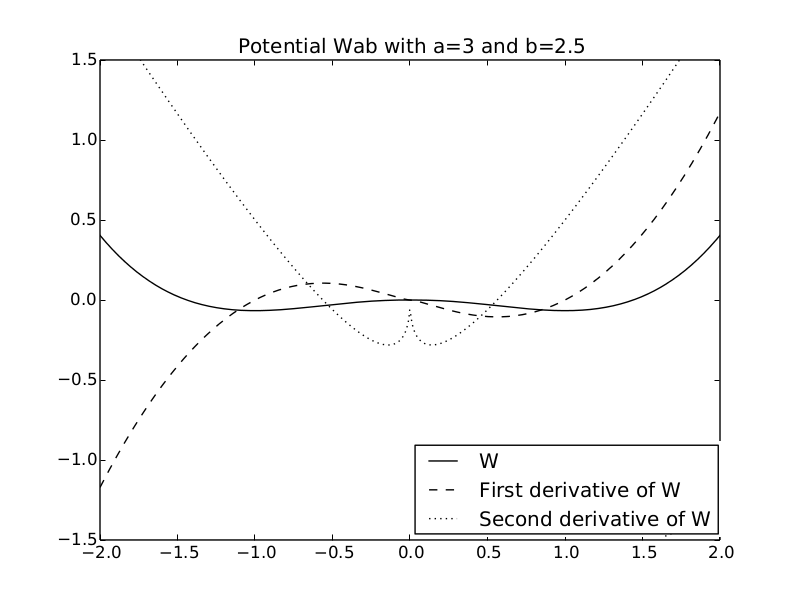} 
   \end{minipage}
\\
\hline
Density $\rho_h^n$
   &
  \begin{minipage}[c]{0.3 \linewidth} 
        \includegraphics[width=4cm]{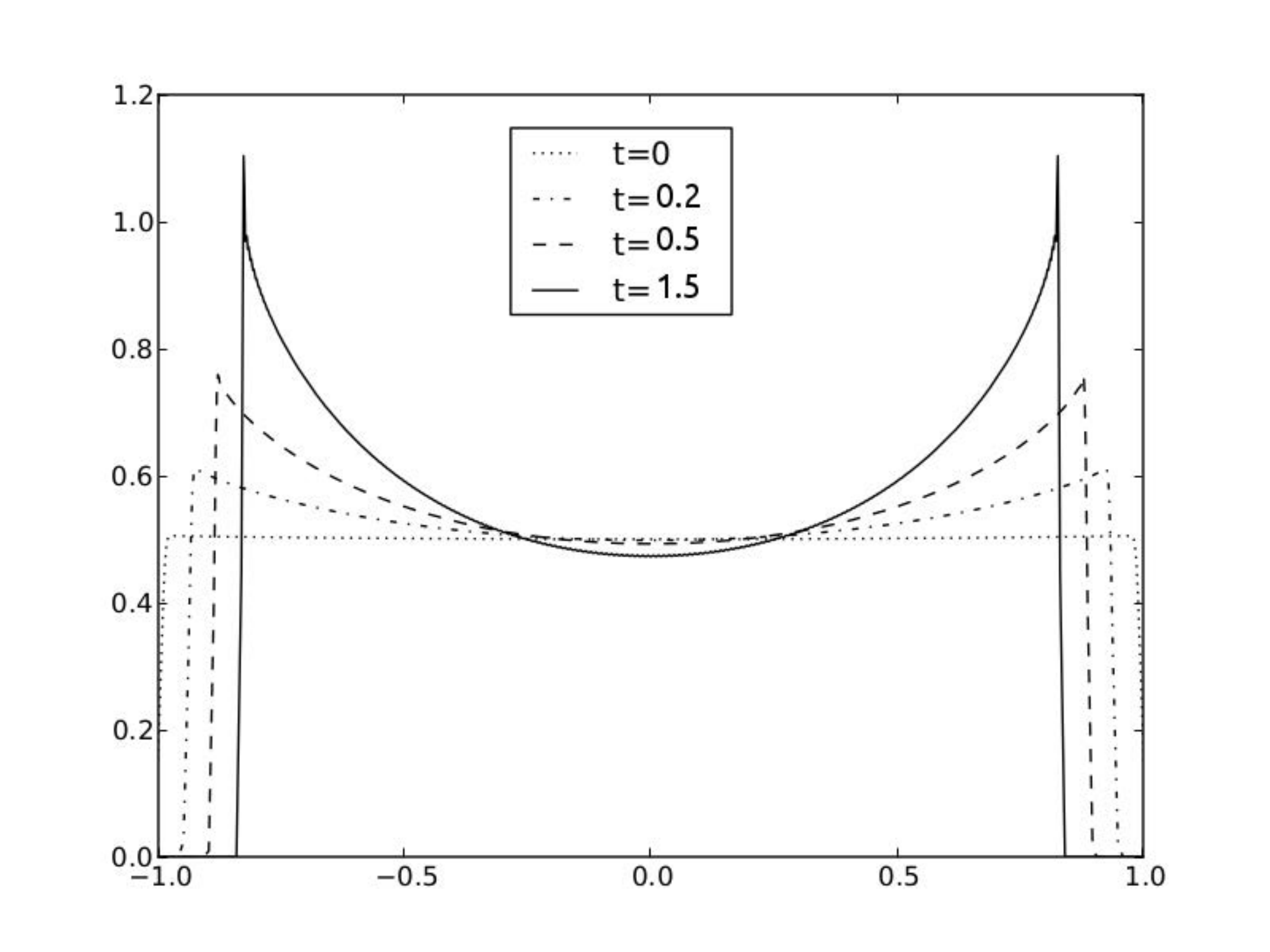}
   \end{minipage}
   &
    \begin{minipage}[c]{0.3 \linewidth} 
    \includegraphics[width=4cm]{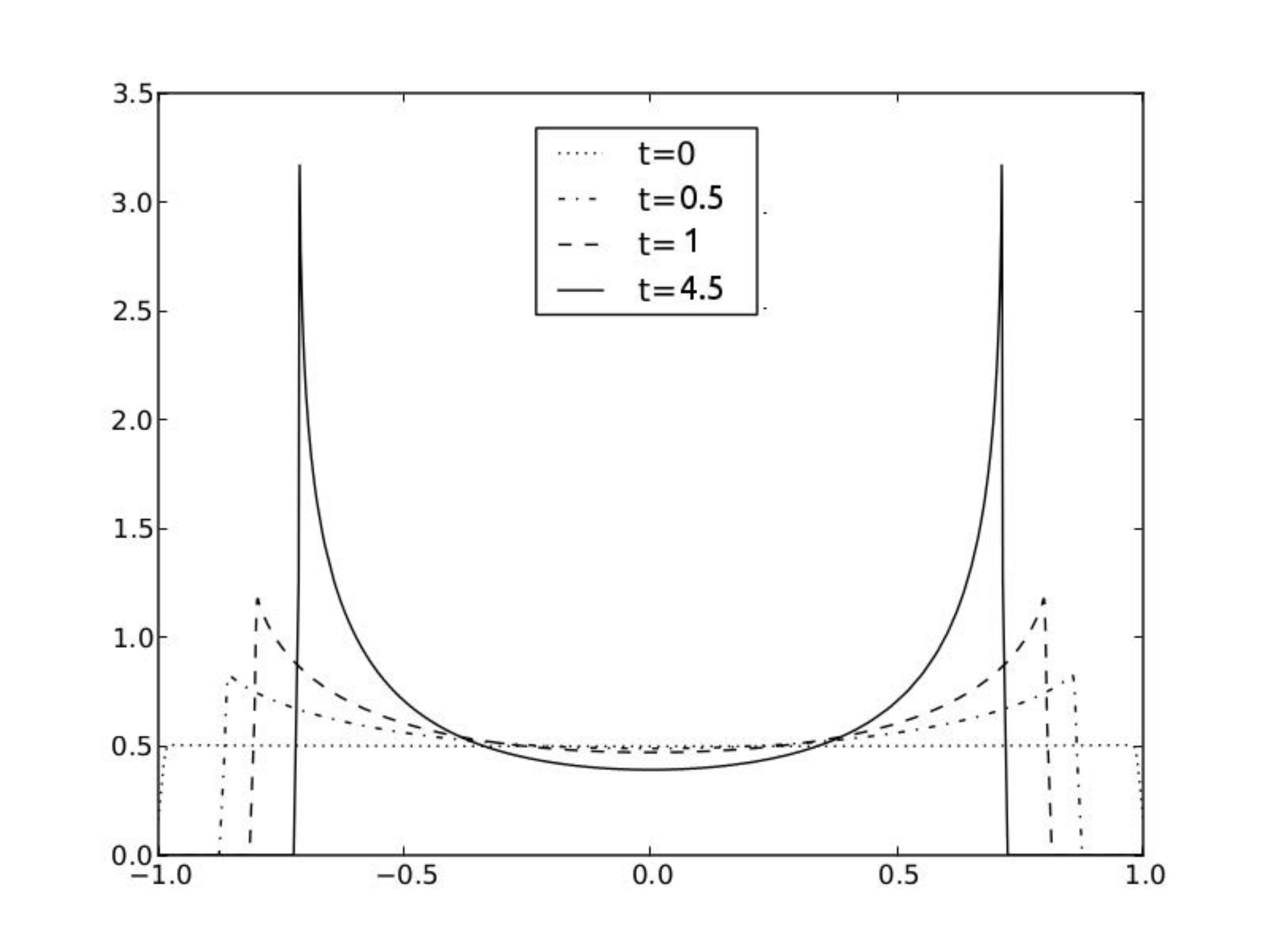}
   \end{minipage}
   &
    \begin{minipage}[c]{0.3 \linewidth} 
     \includegraphics[width=4cm]{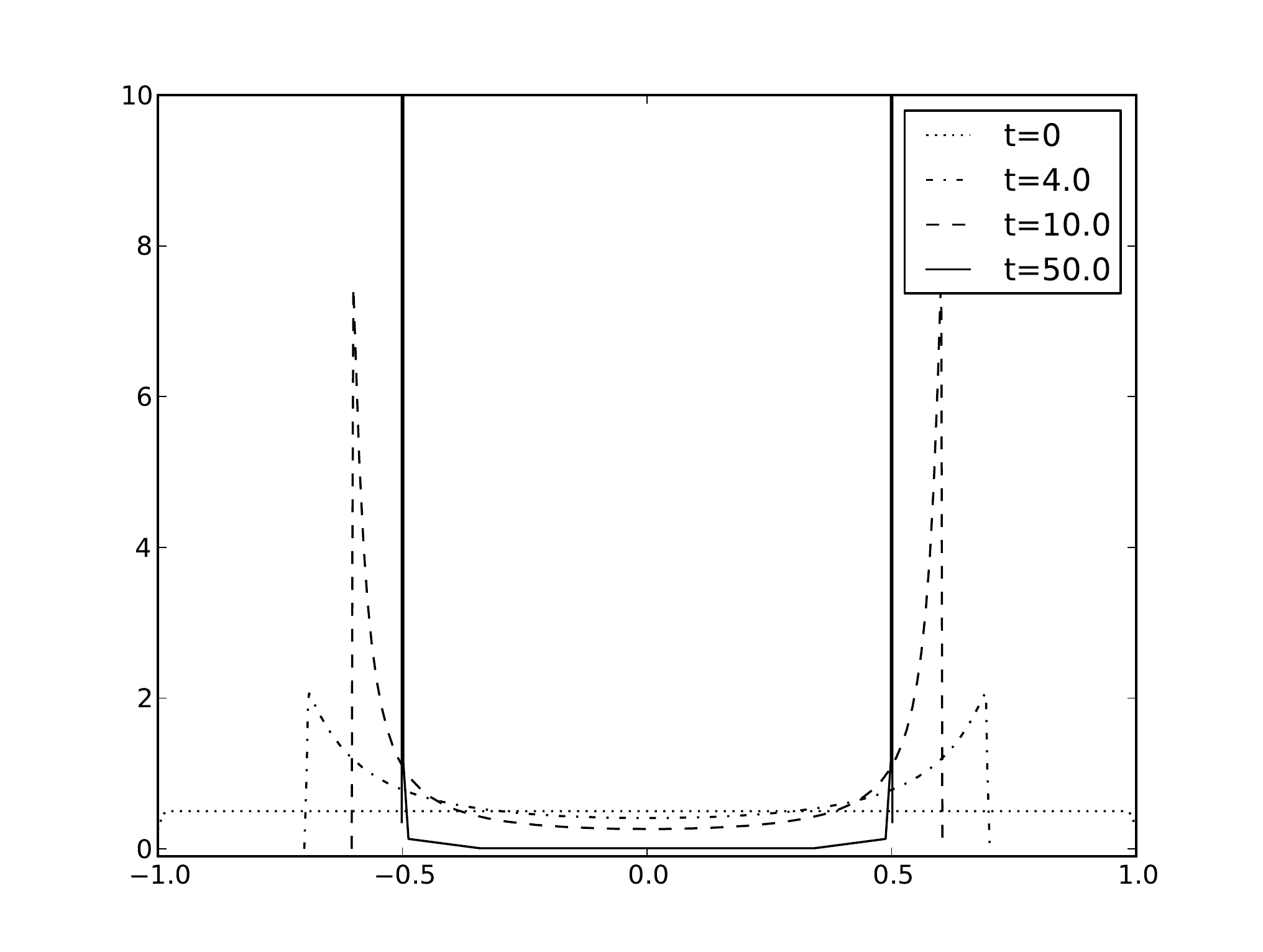}
   \end{minipage}
\\
\hline
 \begin{minipage}[c]{ 0.18\linewidth}\centering
Velocity \\ $\nabla W \ast \rho_h^n$
    \end{minipage}
   &
   \begin{minipage}[c]{0.3 \linewidth} 
         \includegraphics[width=4cm]{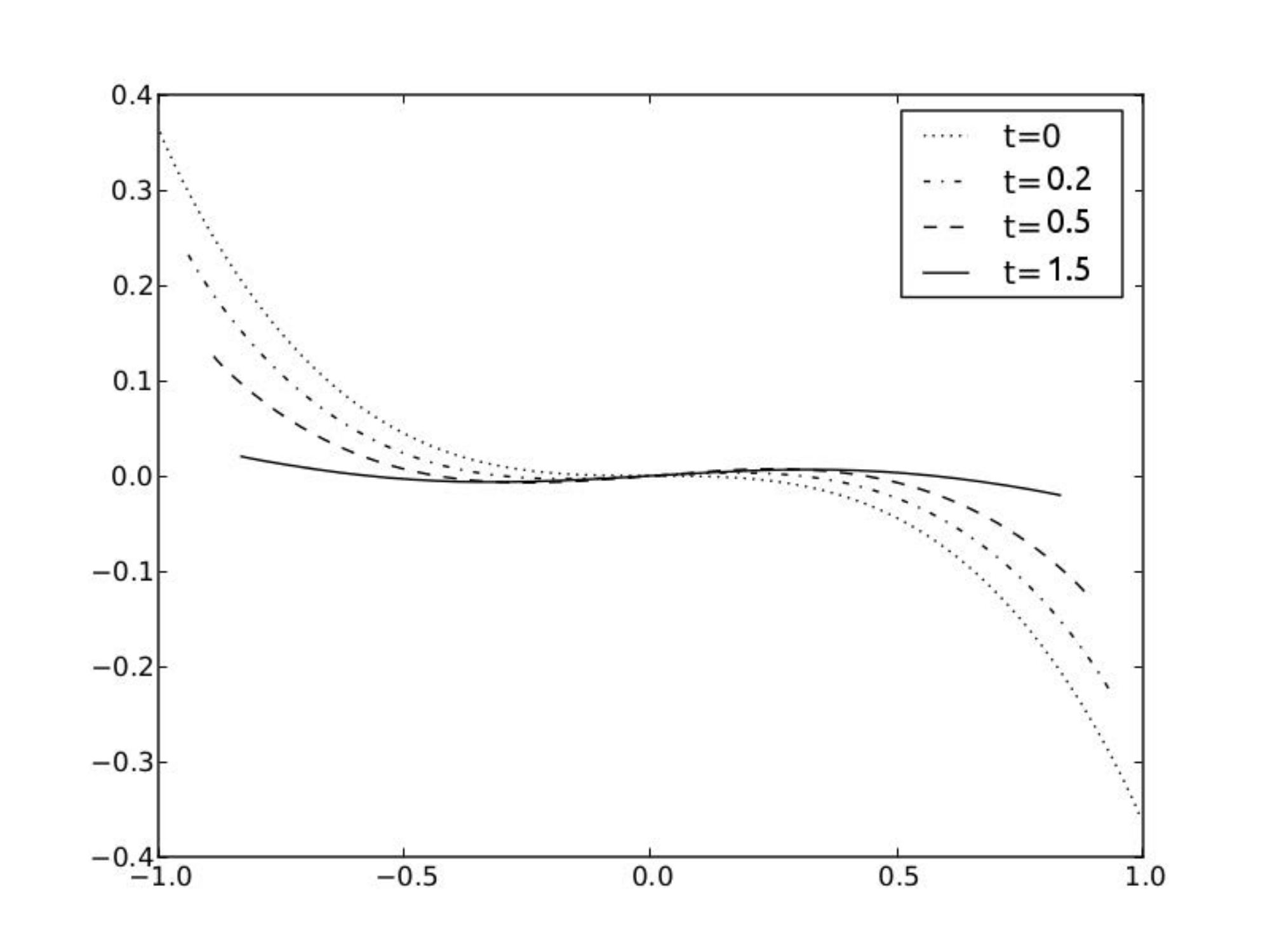}
   \end{minipage}   
  &
    \begin{minipage}[c]{0.3 \linewidth} 
     \includegraphics[width=4cm]{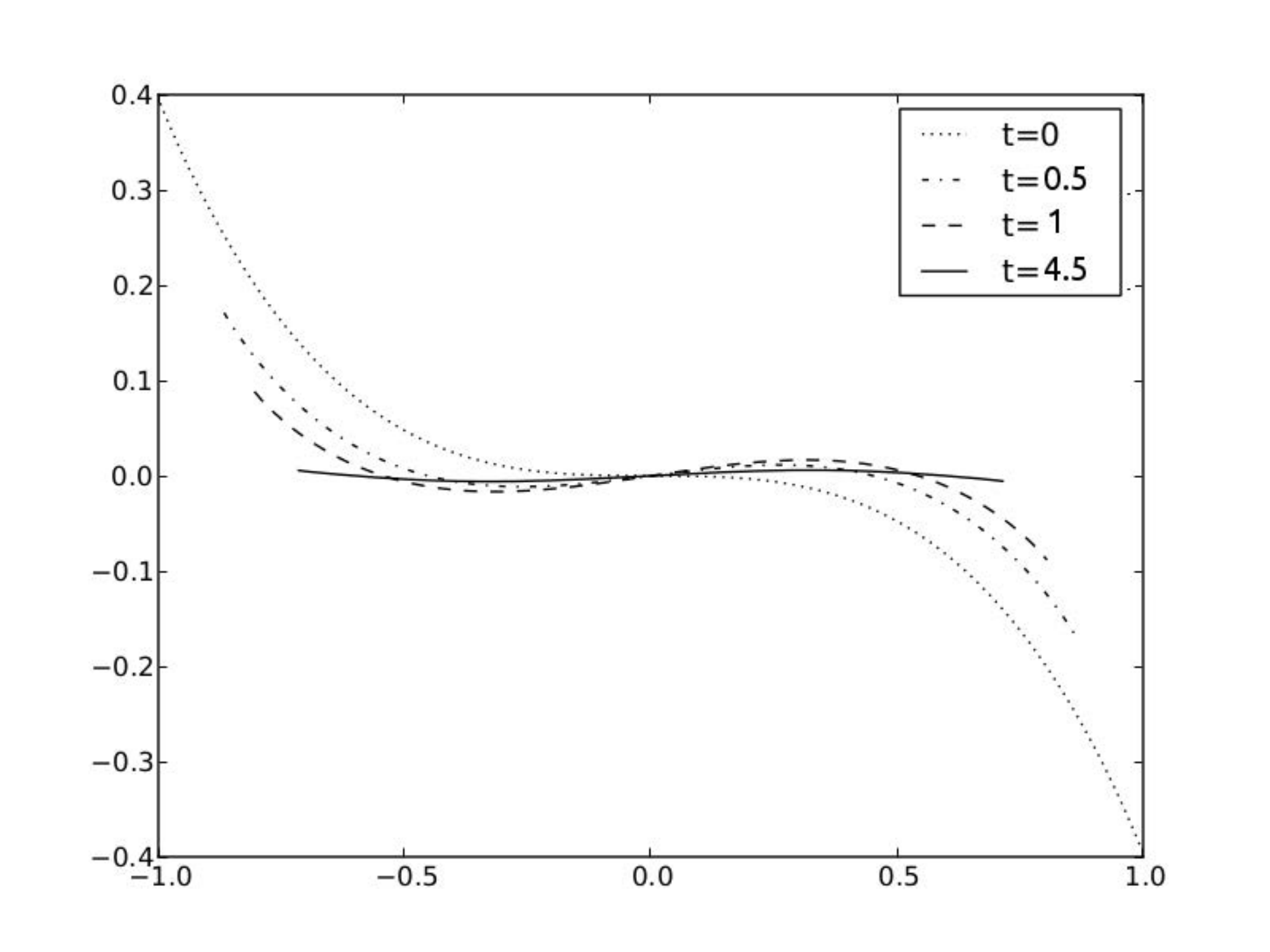}
   \end{minipage}
  &
    \begin{minipage}[c]{0.3 \linewidth} 
     \includegraphics[width=4cm]{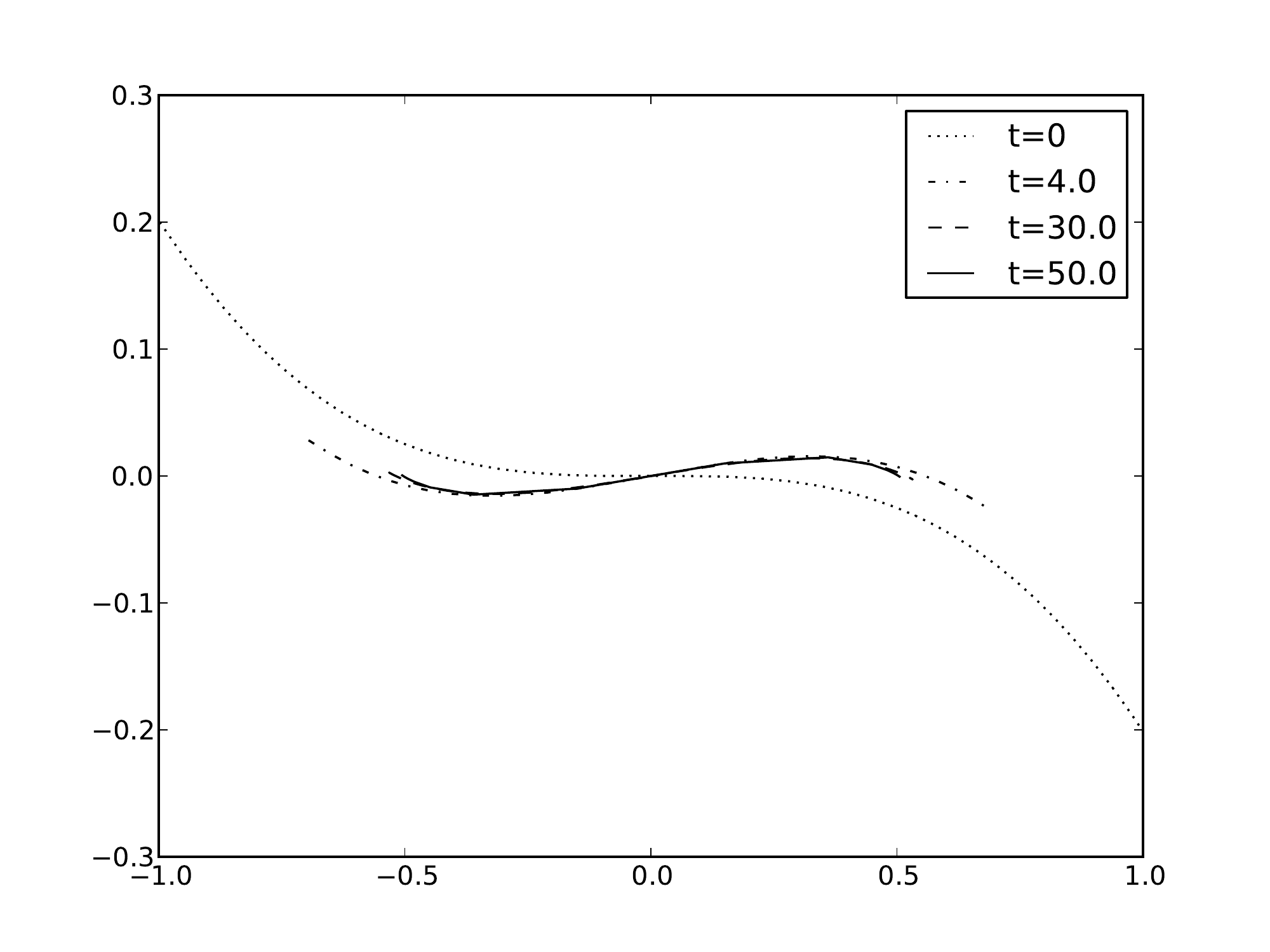}
   \end{minipage}
\\
\hline
 \begin{minipage}[c]{ 0.18\linewidth}\centering
  Size of \\ particles $h_k^n$
    \end{minipage}
  &
   \begin{minipage}[c]{0.3 \linewidth} 
         \includegraphics[width=4cm]{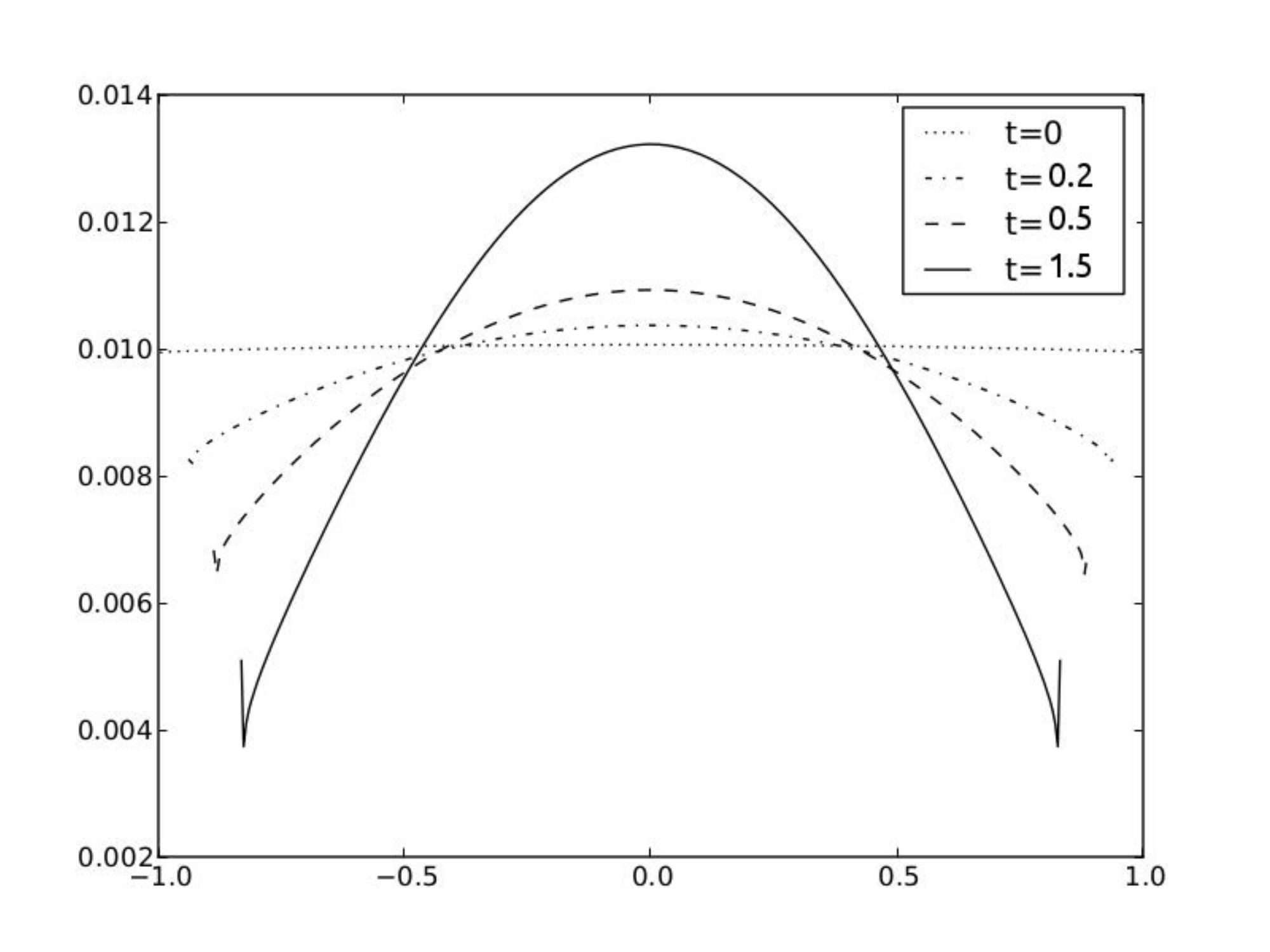}
   \end{minipage}
  &
   \begin{minipage}[c]{0.3 \linewidth} 
    \includegraphics[width=4cm]{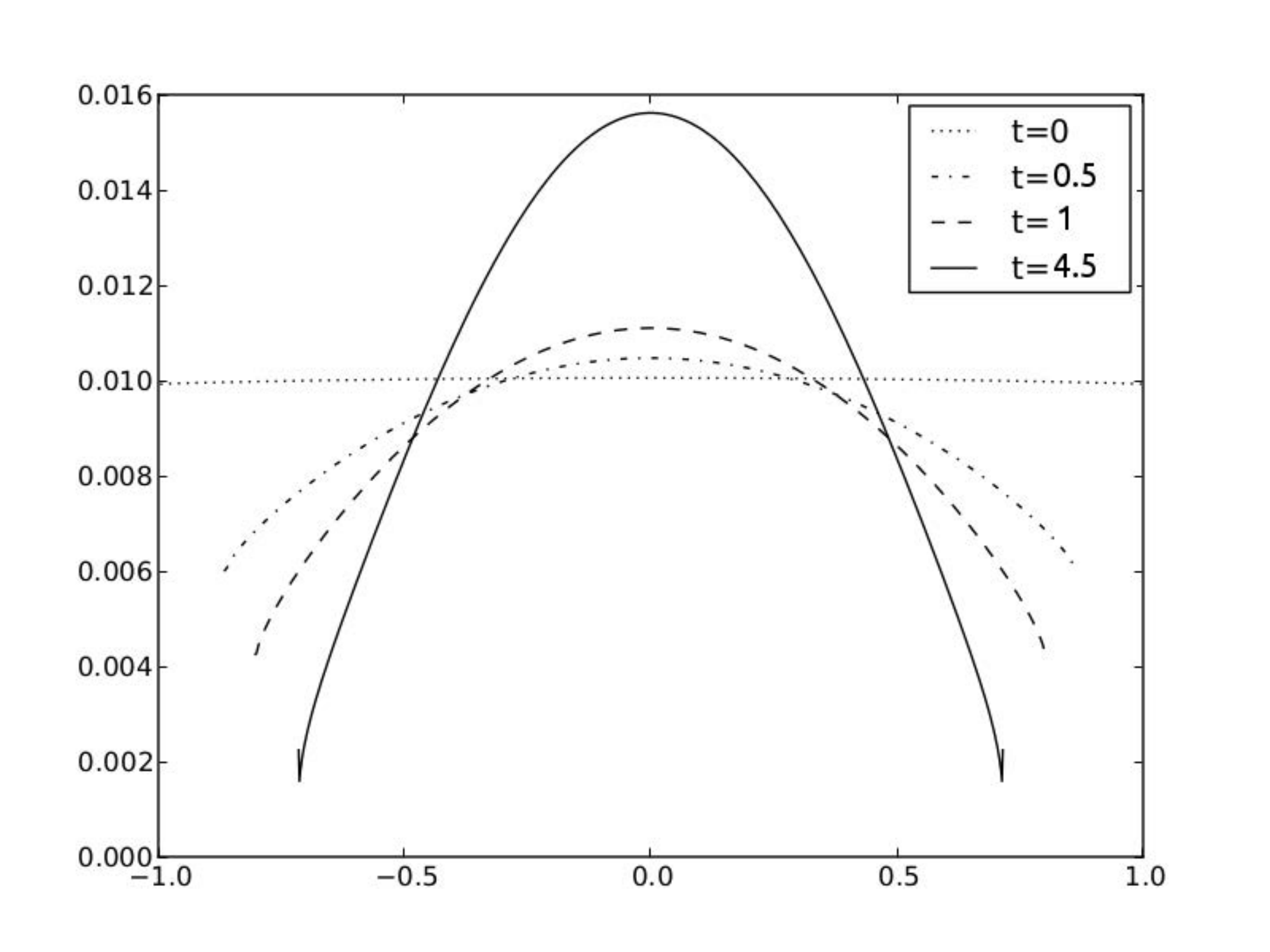}
   \end{minipage}
  &
    \begin{minipage}[c]{0.3 \linewidth} 
     \includegraphics[width=4cm]{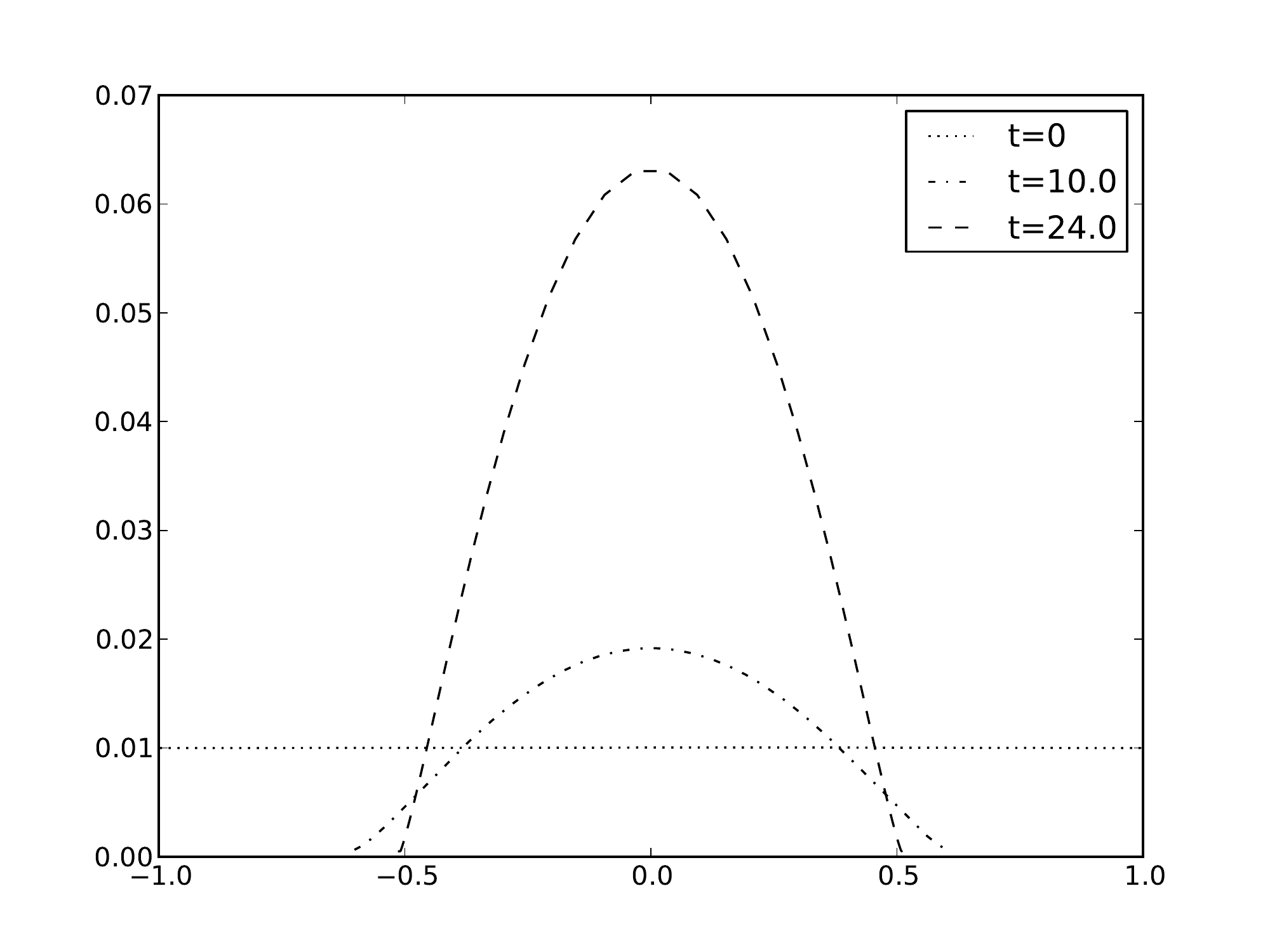}
   \end{minipage}
\\
\hline
\end{tabular}
\caption{\label{figWab} Approximated density and reconstructed velocity and size of particles computed by the LTP method with $h=0.01$ for $W(x)=\frac{|x|^a}{a} -\frac{|x|^b}{b}$, with $a=3$ and $b=1.5$ or $b=2.5$ and $\rho^0$ given by \eqref{rhoini3} with the number of time-steps $N=200$.}
\end{figure}
Now, we further analyse the blow-up behavior by looking at the case of attractive-repulsive potentials $W(x)=\tfrac{|x|^a}{a} -\tfrac{|x|^b}{b}$, $1<b<a$.  
Notice again that for $b\geq 2$ the potential is smooth while for $1<b<2$ is singular once $W$ is cut-off at infinity or if 
the initial data is compactly supported as discussed above.
Figure \ref{figWab} presents the approximated density $\rho_h^n$,  reconstructed velocity $u_h^n $ and size of particles $h^n$ obtained by the LTP method in the case of the attractive-repulsive potentials with $(a,b)=(3,1.5)$ and $(a,b)=(3,2.5)$. In this case $\rho^0$ is given by \eqref{rhoini3}. 

We observe that the long time asymptotics for $b=2.5$ are characterized by the concentration of mass equally onto Dirac deltas at two points in infinite time, while for $b=1.5$ we obtain a convergence in time 
towards a steady $L^1$ density profile seemingly diverging at the boundary of the support. This last behavior has been reported in several simulations and related problems \cite{BT2}. However, it has not been rigorously proven yet. Let us point out that the set of stationary states when the interaction potential is analytic in 1D consists of a finite number of Dirac deltas as proven in \cite{FellnerRaoul1,FellnerRaoul2}. This result also holds for $W(x)=\tfrac{|x|^a}{a} -\tfrac{|x|^b}{b}$, $2<b<a$, as it will be reported in \cite{CFP2}. 

Figure \ref{figWab3-2p5-3D} also represents the time evolution of the approximated density
for $(a,b)=(3,2.5)$, with $\rho^0$ given by \eqref{rhoini4}. Solutions in the range $2<b<a$ for initial data in $L^1\cap L^\infty$ exist globally in time, see \cite{Golse}. The numerical evidence shows that all solutions converge towards stationary states consisting of finite number of Dirac Deltas as $t\to\infty$ in this range.

\begin{figure}[h]
   \centering
    \includegraphics[width=8cm]{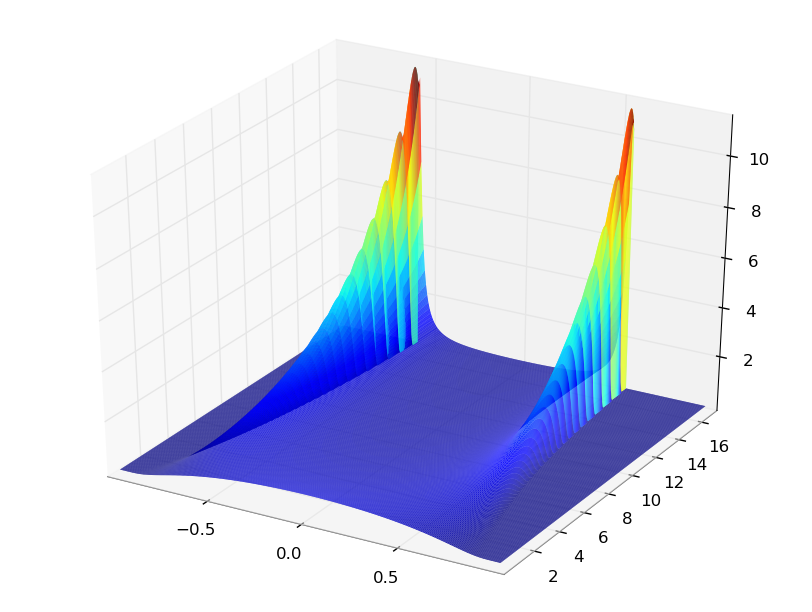}
   \caption{\label{figWab3-2p5-3D} Time evolution of the density for $W(x)=\frac{|x|^a}{a} -\frac{|x|^b}{b}$, with $a=3$, $b=2.5$ and $\rho^0$ given by \eqref{rhoini4} with $N=100$.}
\end{figure} 

Finally, we show in Figure \ref{comparison} the results of the stationary state of the SP method versus the LTP method for the potential $(a,b)=(4,2.5)$ with $N=100$. We observe how the good local adaption of the size of the particles makes our approximation much better with no oscillations with respect to the SP method showing the good performance of the LTP method in this case and its good properties at work.
As mentioned in the introduction, vortex-blob type methods have been shown to converge for the aggregation equation \eqref{main-eq}. 
They obtained convergence estimates in suitable $L^p$ norms for the velocity fields and the associated characteristics fields while the error for the densities was controlled in suitable $W^{-1,p}$-norms in \cite[Th. 3.8]{BC-blob}. The error estimates for vortex-blob and SP methods depend as usual on the regularization of particles and the fixed particle size related in a suitable way to get convergence. We have proven that the LTP method has in contrast direct error estimates for the densities in $L^p$  depending on the initial mesh size showing that the local adaptation of the shape has this benefit on the error estimates too.

\begin{figure}[h]
   \centering
    \includegraphics[width=6cm]{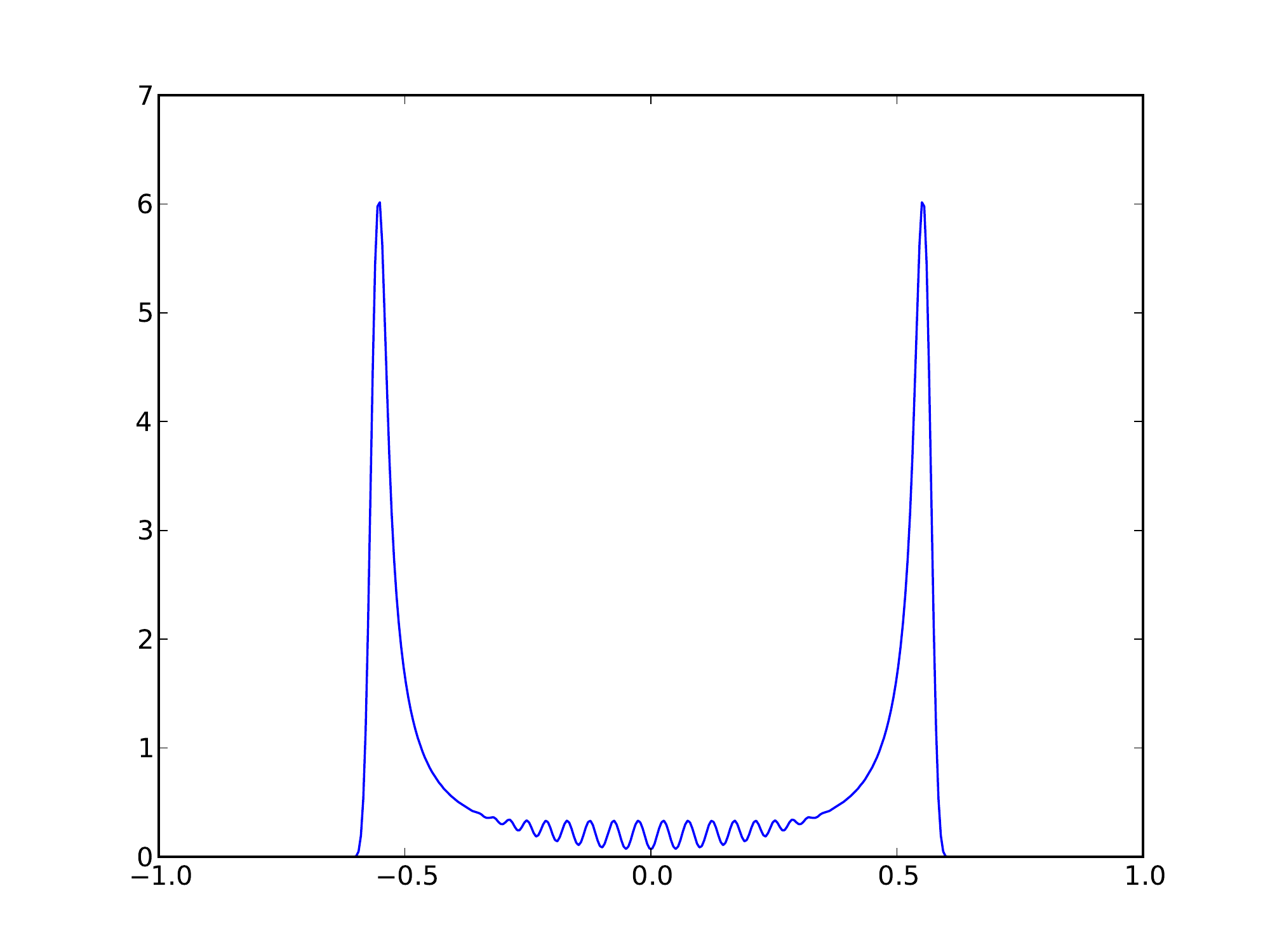}
    \includegraphics[width=6cm]{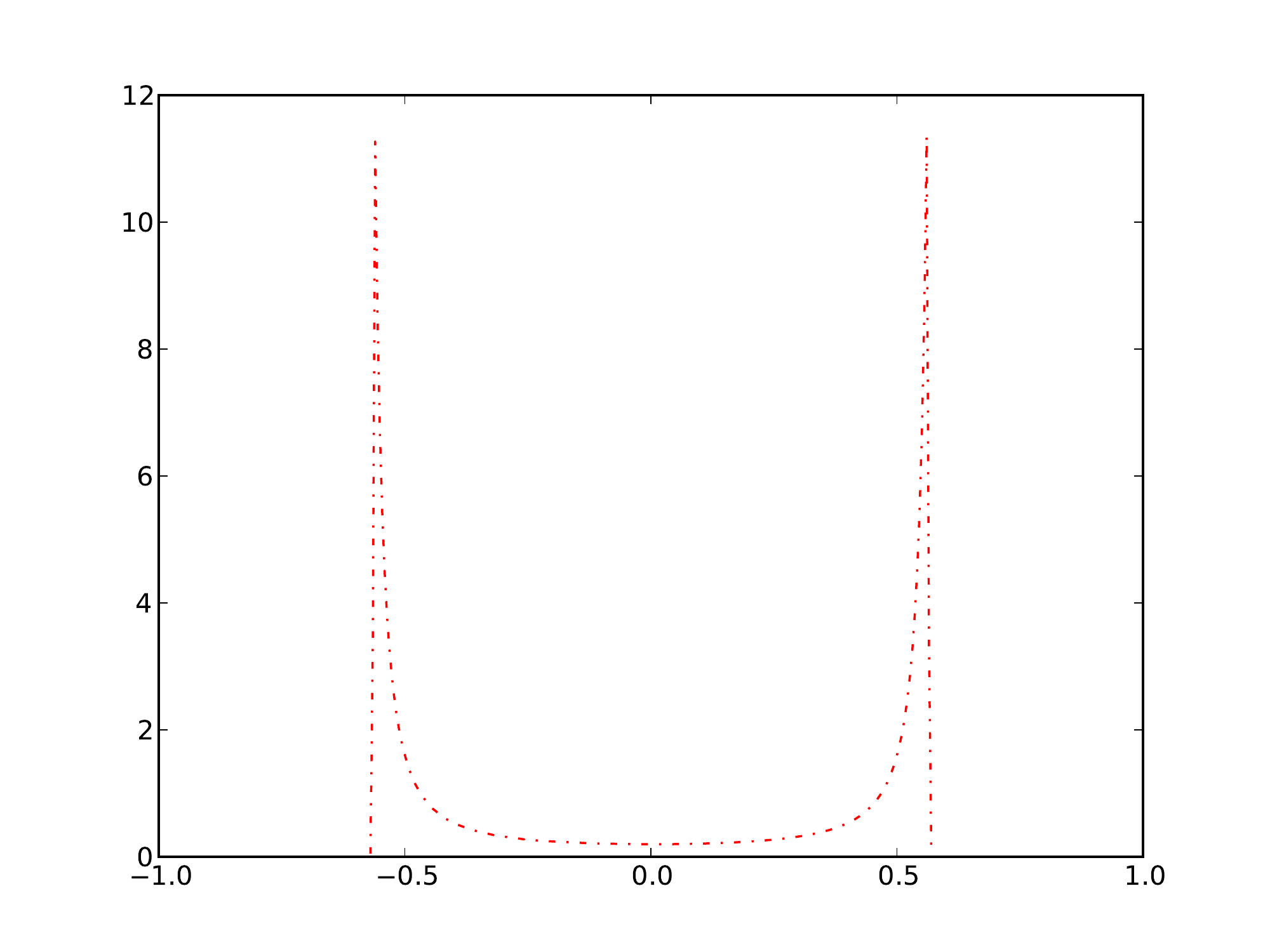}\\
    \includegraphics[width=6cm]{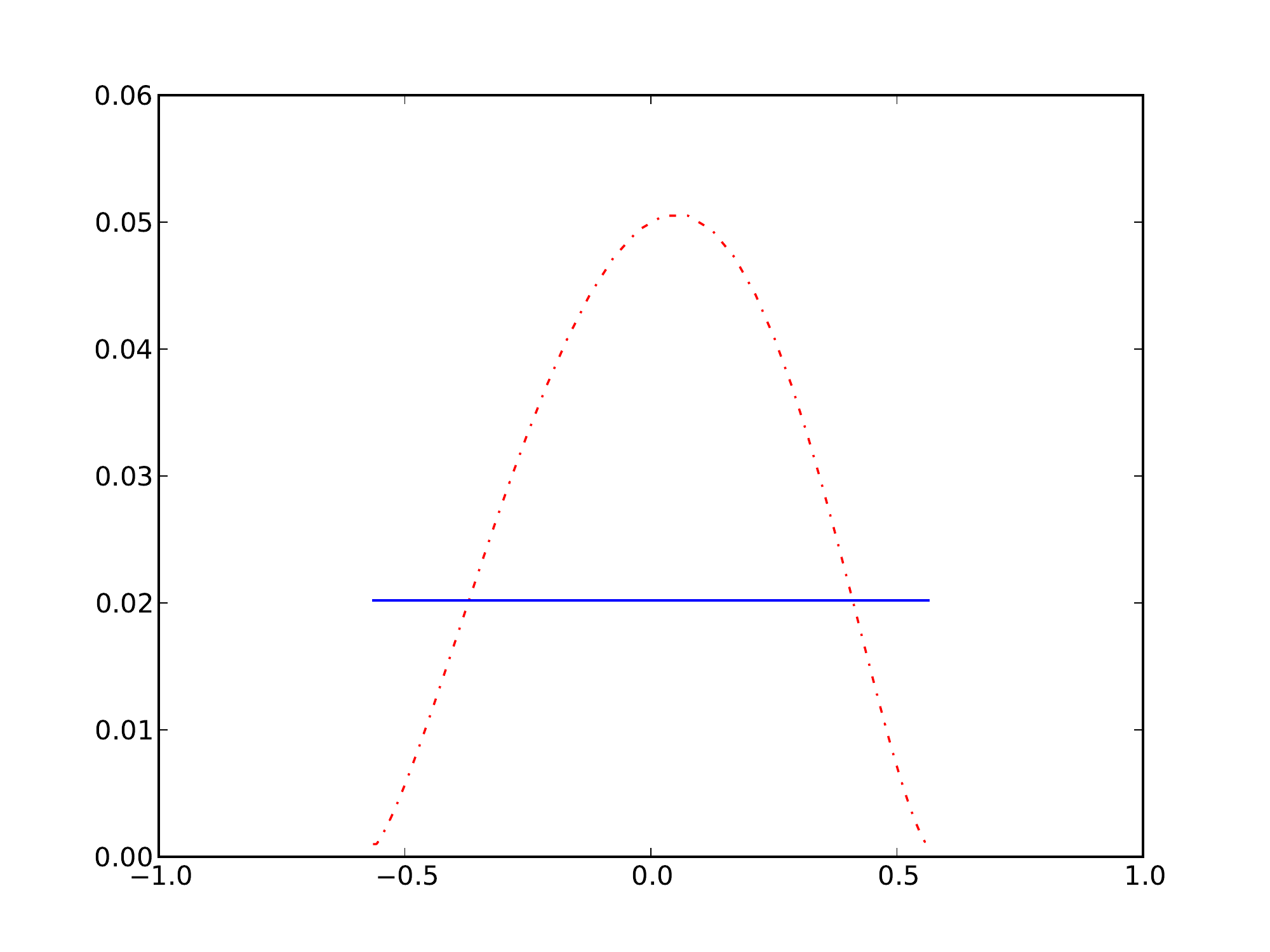}
   \caption{\label{comparison} Densities at steady state for $W(x)=\frac{|x|^a}{a} -\frac{|x|^b}{b}$, with $a=4$, $b=2.5$ and $\rho^0$ given by \eqref{rhoini3}. Left Top: SP method (solid line), right top: LTP method (dotted line), bottom: size of particles (SP versus LTP) with $N=100$.}
\end{figure} 

\appendix
\section{A priori estimates on the regularity of solutions}\label{sec_appendix}
\setcounter{equation}{0}

In this part, we deduce a priori estimates on the regularity of equation \eqref{main-eq} that combined with the global/local in time well posedness theory obtained in \cite{Dob,Golse,BTR,CCH}, leads to the existence of solutions with the desired properties to apply the convergence results of previous sections. 

As we remind the reader in the introduction and in several places along the text, there are two different well-posedness settings: for smooth and for singular potentials. In both cases under the assumptions on the initial data the velocity fields are continuous in time and Lipschitiz continuous in space. In the smooth potential case, this property holds globally in time leading to unique global in time measure solutions \cite{Dob,Golse}. In the singular potential case, this property holds locally in time only since there exist blowing-up of solutions for fully atractive potentials, see \cite{BCL,BTR,CCH}. In both cases, the flow map associated to the velocity field is well-defined and solutions are obtained by pushing forward the initial data through the flow map. 

In this Appendix, we present first a global-in-time  propagation of regularity result in the smooth potential case adapted to our hypotheses on the convergence results. On the other hand, we show a local-in-time  propagation of regularity result in the singular potential case.

\begin{prop} \label{rhoW11}
Assume that the interaction potential $W$ satisfies $\nabla W \in \mw^{1,\infty}(\R^d)$. Let $T >0$ be given and $\rho$ be the unique weak solution to the system \eqref{main-eq} with initial data $\rho^0 \in \mw_+^{1,1}(\R^d)$ obtained in \cite{Dob,Golse}, then
\[
\sup_{0 \leq t \leq T}\|\rho(t,\cdot)\|_{\mw_+^{1,1}} \leq C,
\]
where $C$ is a positive constant depending only on $T$, $L$, and $\|\rho^0\|_{\mw_+^{1,1}}$. Furthermore, if we assume that the initial data $\rho^0 \in (L^\infty \cap \mw_+^{1,1})(\R^d)$, then 
\[
\sup_{0 \leq t \leq T}\|\rho(t,\cdot)\|_{L^\infty \cap \mw^{1,1}_+} \leq C,
\]
where $C$ is a positive constant depending only on $T$, $L$, and $\|\rho^0\|_{L^\infty \cap \mw_+^{1,1}}$.

\end{prop}

\begin{proof} It follows from the conservation of mass and our assumption on the initial data $\rho^0$ that
\[
\int_{\R^d} \rho(t,x) \,dx = \int_{\R^d} \rho^0(x) \,dx = 1.
\]
For the estimate of $\|\rho\|_{L^\infty(0,T;\dot{\mw}^{1,1})}$, we take $\nabla$ to \eqref{main-eq} to get
\begin{align}\label{eq_diff}
\begin{aligned}
\pa_t \nabla \rho(t,x) &+ D^2 \rho(t,x) u(t,x) + \nabla u(t,x) \nabla \rho(t,x) \cr
&\quad + \nabla(\nabla \cdot u(t,x)) \rho(t,x) + \nabla \cdot u(t,x) \nabla \rho(t,x) = 0.
\end{aligned}
\end{align}
We next multiply \eqref{eq_diff} by $\nabla \rho(t,x)/|\nabla \rho(t,x)|$ to obtain
\begin{align}\label{est_ext1}
\begin{aligned}
\partial_t |\nabla \rho| &+ u\cdot \nabla |\nabla \rho(t,x)| + \nabla \cdot  u|\nabla\rho(t,x)| = \cr & \quad - \nabla u \nabla \rho \cdot \frac{\nabla \rho}{|\nabla \rho|} - \nabla(\nabla \cdot u) \rho \cdot \frac{\nabla \rho}{|\nabla \rho|},
\end{aligned}
\end{align}
due to the symmetry of $D^2 \rho$. By integrating \eqref{est_ext1} over $\R^d$ and using integration by parts, we deduce
\begin{align}\label{est_ext2}
\begin{aligned}
\frac{d}{d t}\int_{\R^d}|\nabla \rho|\,dx = &-\int_{\R^d} \nabla u \nabla \rho \cdot \frac{\nabla \rho}{|\nabla \rho|}\,dx \cr &- \int_{\R^d} \nabla(\nabla \cdot u) \rho \cdot \frac{\nabla \rho}{|\nabla \rho|}\,dx\leq 2L |\nabla \rho|\,dx,
\end{aligned}
\end{align}
where we used $\|\nabla u(t,x)\|_{L^\infty} \leq \|\nabla W\|_{\mw^{1,\infty}} =L$ and
\[
\|\nabla(\nabla \cdot u)\|_{L^\infty} \leq L\int_{\R^d} |\nabla \rho|\,dx.
\]
Thus we have
\[
\sup_{0 \leq t \leq T}\|\nabla \rho(t,\cdot)\|_{L^1} \leq \|\nabla \rho^0\|_{L^1}\exp\lt(2LT\rt).
\]
Finally, we estimate $\|\rho\|_{L^\infty}$. For this, we recall that the flow map $F^{0,t}(x)$ satisfies
\[
\frac{d F^{0,t}(x)}{dt} = u(t, F^{0,t}(x)) \quad \mbox{with} \quad F^{0,0}(x) = x.
\]
Using that $\rho(t)= F^{0,t}\#\rho^0$, we can write
\[
\frac{\pa}{\pa t}\rho(t,F^{0,t}(x)) = -\nabla \cdot u(t,F^{0,t}(x))\rho(t,F^{0,t}(x)),
\]
and this yields
\[
\rho(t,F^{0,t}(x)) = \rho^0(x)\exp\left(-\int_0^t \nabla \cdot u(s,F^{0,s}(x))\,ds\right).
\]
Since $u \in \mw^{1,\infty}(\R^d)$, we obtain 
\beq\label{est_rho_inf}
\sup_{0 \leq t \leq T}\|\rho(t,\cdot)\|_{L^\infty} \leq \|\rho^0\|_{L^\infty}\exp \left( LT\right).
\eeq
This completes the proof.
\end{proof}

\begin{rem} If we further assume that $\rho^0 \in \mw_+^{1,\infty}(\R^d)$, we have
\[
\sup_{0 \leq t \leq T}\|\rho(t,\cdot)\|_{\mw^{1,\infty}} \leq C,
\]
where $C$ is a positive constant depending only on $T$, $L$, and $\|\rho^0\|_{\mw^{1,\infty}_+}$. Indeed, we can similarly find from \eqref{eq_diff} that for $i =1,\cdots, d$
$$\begin{aligned}
\frac{\pa}{\pa t}\pa_i \rho(t,F^{0,t}(x)) &= -\pa_i u(t,F^{0,t}(x))\nabla  \rho(t,F^{0,t}(x)) -\nabla \cdot u(t,F^{0,t}(x)) \pa_i  \rho(t,F^{0,t}(x)) \cr
&\quad - \rho(t,F^{0,t}(x))\nabla \cdot \pa_i u(t,F^{0,t}(x)).
\end{aligned}$$
This implies
$$\begin{aligned}
\|\nabla \rho(t,\cdot)\|_{L^\infty} &\leq \|\nabla \rho_0\|_{L^\infty} \exp\lt(C\int_0^t \|\nabla u(s,\cdot)\|_{L^\infty}\,ds \rt) \cr
&\quad + \exp\lt(C\int_0^t \|\nabla u(s,\cdot)\|_{L^\infty}\,ds \rt)\int_0^t \|\rho(s,\cdot)\|_{L^\infty}\|D^2 u(s,\cdot)\|_{L^\infty}\, ds,\cr
&\leq C\|\nabla \rho_0\|_{L^\infty} + C\int_0^t \|D^2 u(s,\cdot)\|_{L^\infty}\,ds
\end{aligned}$$
where we used $u \in \mw^{1,\infty}(\R^d)$ and the estimate \eqref{est_rho_inf}. On the other hand, $\|D^2 u(s,\cdot)\|_{L^\infty}$ can be estimated by
\[
\|D^2 u(s,\cdot)\|_{L^\infty} \leq \|\nabla W\|_{\mw^{1,\infty}}\|\nabla \rho(s,\cdot)\|_{L^\infty}.
\]
Hence, we have 
\[
\|\nabla \rho(t,\cdot)\|_{L^\infty} \leq C\|\nabla \rho_0\|_{L^\infty} + C\int_0^t \|\nabla \rho(s,\cdot)\|_{L^\infty}\,ds,
\]
and by applying Gronwall's inequality to conclude the desired result. Similar arguments were used in \cite{BCLR} to construct classical solutions.
\end{rem}

We next provide the a priori estimate of solutions to the system \eqref{main-eq} in $\mw^{1,1}_+(\R^d) \cap \mw^{1,p}(\R^d)$. For notational simplicity, we set 
\[
\widetilde{\mw}^{k,p}_+(\R^d) := \mw^{k,1}_+(\R^d) \cap \mw^{k,p}(\R^d) \quad \mbox{for} \quad k \geq 0.
\]

\begin{prop}\label{prop_lp} Assume that the interaction potential $W$ satisfies \eqref{hypW} for some $1 \leq q \leq \frac{d}{\alpha + 1}$. Let $\rho$ be the unique local-in-time solution to \eqref{main-eq} constructed in \cite{CCH} with initial data $\rho^0$ satisfying $\rho^0 \in (L^\infty \cap \widetilde\mw_+^{1,p})(\R^d)$ where $p$ is the Sobolev conjugate of $q$. Then there exists a $T^* > 0$ such that
\[
\sup_{0 \leq t \leq T^*}\|\rho(t,\cdot)\|_{\widetilde\mw^{1,p}_+} \leq C,
\]
where $C$ is a positive constant depending only on $T^*$, $\alpha$, $p$, and $\|\rho^0\|_{\widetilde\mw^{1,p}_+}$.
\end{prop}
\begin{proof} 
The local-in-time well-posedness theory in \cite{CCH} that
\beq\label{est_ap_s}
\frac{d}{dt}\|\rho\|_{\widetilde\mw^{0,p}_+} \leq C\|\rho\|_{\widetilde\mw^{0,p}_+}^2\,.
\eeq
It also follows from \eqref{eq_diff}-\eqref{est_ext2} that
\[
\frac{d}{dt}\|\nabla \rho\|_{L^1} \lesssim \|\rho\|_{\widetilde\mw^{0,p}_+}\|\nabla \rho\|_{L^1} + \|\nabla  \rho\|_{\widetilde\mw^{0,p}_+} \leq \|\rho\|_{\widetilde\mw^{1,p}_+}^2,
\]
where we used $\|D^k u(t,x)\|_{L^\infty} \leq C\|D^{k-1} \rho\|_{\widetilde\mw^{0,p}_+}$ for $k \geq 1$ and $\|\rho\|_{\widetilde\mw^{0,p}_+} \geq 1$. For the estimate of $\|\nabla \rho\|_{L^p}$, we obtain
$$\begin{aligned}
\frac{d}{dt}\int_{\R^d} |\nabla \rho|^pdx &= -p\int_{\R^d} |\nabla \rho|^{p-2}\nabla \rho  \cdot \left( D^2 \rho u  + \nabla u \nabla \rho + \nabla(\nabla \cdot u) \rho + \nabla \cdot u \nabla \rho\right)dx\cr
&= (a) + (b) + (c) + (d),
\end{aligned}$$
where $(a), (b), (c)$, and $(d)$ are estimated as follows.
$$\begin{aligned}
(a) & = -\int_{\R^d} u \cdot \nabla |\nabla \rho|^pdx = \int_\R \nabla\cdot u |\nabla \rho|^pdx \lesssim \|\rho\|_{\widetilde\mw^{0,p}_+}\|\nabla \rho\|_{L^p}^p,\cr
(b) &\leq p\int_{\R^d} |\nabla u| |\nabla \rho|^pdx \lesssim\|\rho\|_{\widetilde\mw^{0,p}_+}\|\nabla \rho\|_{L^p}^p,\cr
(c) &\leq p\|\nabla^2 u\|_{L^\infty}\|\rho\|_{L^p}\|\nabla \rho\|_{L^p}^{p-1} \lesssim \|\nabla \rho\|_{\widetilde\mw^{0,p}_+}\|\rho\|_{L^p}\|\nabla \rho\|_{L^p}^{p-1},\cr
(d) &\leq p\int_{\R^d} |\nabla \cdot u| |\nabla \rho|^pdx \lesssim\|\rho\|_{\widetilde\mw^{0,p}_+}\|\nabla \rho\|_{L^p}^p.
\end{aligned}$$
Thus, we get
\beq\label{est_ap_s2}
\frac{d}{dt}\|\nabla \rho\|_{L^p} \leq C\|\rho\|_{\widetilde\mw^{1,p}_+}^2.
\eeq
Now, we combine \eqref{est_ap_s} and \eqref{est_ap_s2} to deduce
\[
\frac{d}{dt}\|\rho\|_{\widetilde\mw^{1,p}_+} \leq C\|\rho\|_{\widetilde\mw^{1,p}_+}^2,
\]
and this concludes that there exists a $T^*>0$ such that
\[
\sup_{0 \leq t \leq T}\|\rho(t,\cdot)\|_{\widetilde\mw^{1,p}_+} \leq C,
\]
where $C$ is a positive constant depending only on $T^*$, $\alpha$, $p$, and $\|\rho^0\|_{\widetilde\mw^{1,p}_+}$.

\end{proof}

%
%
%
%

\section*{Acknowledgments}
\small{JAC was partially supported by the project MTM2011-27739-C04-02
DGI (Spain) and from the Royal Society by a Wolfson
Research Merit Award. JAC and YPC were supported by EPSRC grant with reference
EP/K008404/1. This work was partially done when FC was visiting
Imperial College funded by the EPSRC EP/I019111/1 (platform grant).}

%
%
%
%

\end{document}